\documentclass[12pt]{article}

\usepackage[english]{babel}

\usepackage[pdftex,paper=a4paper,portrait=true,textwidth=450pt,textheight=675pt,tmargin=3cm,marginratio=1:1]{geometry}

\usepackage{amsfonts}

\usepackage[dvips]{graphics}

\usepackage{amsmath}

\usepackage{amsthm}

\usepackage{amssymb}

\usepackage{eufrak}

\usepackage{cancel}

\input xy

\xyoption{all}

\newtheorem{theorem}{Theorem}[section]

\newtheorem{corollary}[theorem]{Corollary}

\newtheorem{proposition}[theorem]{Proposition}

\newtheorem{assumption}[theorem]{Assumption}

\theoremstyle{definition}

\newtheorem{definition}[theorem]{Definition}

\theoremstyle{definition}

\theoremstyle{definition}

\newtheorem{example}[theorem]{Example}

\makeatletter
\newcommand{\contraction}[5][1ex]{%
  \mathchoice
    {\contraction@\displaystyle{#2}{#3}{#4}{#5}{#1}}%
    {\contraction@\textstyle{#2}{#3}{#4}{#5}{#1}}%
    {\contraction@\scriptstyle{#2}{#3}{#4}{#5}{#1}}%
    {\contraction@\scriptscriptstyle{#2}{#3}{#4}{#5}{#1}}}%
\newcommand{\contraction@}[6]{%
  \setbox0=\hbox{$#1#2$}%
  \setbox2=\hbox{$#1#3$}%
  \setbox4=\hbox{$#1#4$}%
  \setbox6=\hbox{$#1#5$}%
  \dimen0=\wd2%
  \advance\dimen0 by \wd6%
  \divide\dimen0 by 2%
  \advance\dimen0 by \wd4%
  \vbox{%
    \hbox to 0pt{%
      \kern \wd0%
      \kern 0.5\wd2%
      \contraction@@{\dimen0}{#6}%
      \hss}%
    \vskip 0.5ex
    \vskip\ht2}}

\newcommand{\contraction@@}[3][0.05em]{%
  \hbox{%
    \vrule width #1 height 0pt depth #3%
    \vrule width #2 height 0pt depth #1%
    \vrule width #1 height 0pt depth #3%
    \relax}}
\makeatother

\begin{document}
\title{Fixed Point Loci of Moduli Spaces of Sheaves on Toric Varieties}
\author{Martijn Kool\thanks{Department of Mathematics, Imperial College, 180 Queens Gate, London, SW7 2AZ, United Kingdom, {\tt{m.kool@imperial.ac.uk}}.}}
\maketitle
\begin{abstract}
Extending work of Klyachko and Perling, we develop a combinatorial description of pure equivariant sheaves of any dimension on an arbitrary nonsingular toric variety $X$. Using geometric invariant theory (GIT), this allows us to construct explicit moduli spaces of pure equivariant sheaves on $X$ corepresenting natural moduli functors (similar to work of Payne in the case of equivariant vector bundles). The action of the algebraic torus on $X$ lifts to the moduli space of all Gieseker stable sheaves on $X$ and we express its fixed point locus explicitly in terms of moduli spaces of pure equivariant sheaves on $X$. One of the problems arising is to find an equivariant line bundle on the side of the GIT problem, which precisely recovers Gieseker stability. In the case of torsion free equivariant sheaves, we can always construct such equivariant line bundles. As a by-product, we get a combinatorial description of the fixed point locus of the moduli space of $\mu$-stable reflexive sheaves on $X$. As an application, we show in a sequel \cite{Koo} how these methods can be used to compute generating functions of Euler characteristics of moduli spaces of $\mu$-stable torsion free sheaves on nonsingular complete toric surfaces.
\end{abstract}

\section{Introduction}

Vakil has shown that the moduli space of Gieseker stable sheaves satisfies Murphy's Law, meaning every singularity type of finite type over $\mathbb{Z}$ appears on the moduli space \cite{Vak}. Hence the moduli space $\mathcal{M}_{P}^{s}$ of Gieseker stable sheaves with Hilbert polynomial $P$ on a projective variety $X$ with ample line bundle $\mathcal{O}_{X}(1)$ can become very complicated. Now assume $X$ is a nonsingular projective toric variety with torus $T$. We can lift the action of $T$ on $X$ to an action of $T$ on the moduli space $\mathcal{M}_{P}^{s}$. One of the goals of this paper is to find a combinatorial description of the fixed point locus $\left( \mathcal{M}_{P}^{s} \right)^{T}$ using techniques of toric geometry. 

Klyachko has given a combinatorial description of equivariant vector bundles and, more generally, reflexive equivariant and torsion free equivariant sheaves on a nonsingular toric variety \cite{Kly1}, \cite{Kly2}, \cite{Kly3}, \cite{Kly4}. This description gives a relatively easy way to compute Chern characters and sheaf cohomology of such sheaves. Klyachko's work has been reconsidered and extended by Knutson and Sharpe in \cite{KS1}, \cite{KS2}. They sketch how his combinatorial description can be used to construct moduli spaces of equivariant vector bundles and reflexive equivariant sheaves. Perling has given a general description of equivariant quasi-coherent sheaves on toric varieties in \cite{Per1}, \cite{Per2}. He gives a detailed study of the moduli space of rank 2 equivariant vector bundles on nonsingular toric surfaces in \cite{Per3}. A systematic construction of the moduli spaces of equivariant vector bundles on toric varieties has been given by Payne \cite{Pay}. He considers families of equivariant vector bundles on toric varieties and shows the moduli space of rank 3 equivariant vector bundles on toric varieties satisfies Murphy's Law. 

In the current paper, we will present a combinatorial description of pure equivariant sheaves on nonsingular toric varieties (Theorem \ref{ch. 1, sect. 2, thm. 2}), generalising the known combinatorial description of torsion free equivariant sheaves due to Klyachko \cite{Kly4}. Using this combinatorial description, we construct coarse moduli spaces of pure equivariant sheaves on nonsingular projective toric varieties (Theorem \ref{ch. 1, sect. 3, thm. 3}), corepresenting natural moduli functors. In order to achieve this, we develop an explicit description of families of pure equivariant sheaves on nonsingular projective toric varieties (Theorem \ref{ch. 1, sect. 3, thm. 1}), analogous to Payne's description in the case of families of equivariant vector bundles \cite{Pay}. The moduli spaces of pure equivariant sheaves on nonsingular projective toric varieties are constructed using GIT. It is important to note that these moduli spaces are explicit and combinatorial in nature, which makes them suitable for computations. We are interested in the case where GIT stability coincides with Gieseker stability, which is the natural notion of stability for coherent sheaves. Consequently, we would like the existence of an equivariant line bundle in the GIT problem, which precisely recovers Gieseker stability. In the case of reflexive equivariant sheaves and $\mu$-stability, some aspects of this issue are discussed in \cite{KS1}, \cite{KS2} and \cite{Kly4}. We construct ample equivariant line bundles matching GIT and Gieseker stability for torsion free equivariant sheaves in general (Theorem \ref{ch. 1, sect. 3, thm. 4}). Subsequently, we consider the moduli space $\mathcal{M}_{P}^{s}$ of all Gieseker stable sheaves with (arbitrary) fixed Hilbert polynomial $P$ on a nonsingular projective toric variety $X$ with torus $T$ and ample line bundle $\mathcal{O}_{X}(1)$. We lift the action of the torus $T$ to $\mathcal{M}_{P}^{s}$, describe the closed points of the fixed point locus $\left( \mathcal{M}_{P}^{s} \right)^{T}$ and study the difference between invariant and equivariant simple sheaves. We study deformation theoretic aspects of equivariant sheaves and describe the fixed point locus $\left( \mathcal{M}_{P}^{s} \right)^{T}$ explicitly, as a scheme, in terms of moduli spaces of pure equivariant sheaves on $X$. 
\begin{theorem} [Corollary \ref{ch. 1, sect. 4, cor. 1}] \label{ch. 1, sect. 1, thm. 1}
Let $X$ be a nonsingular projective toric variety. Let $\mathcal{O}_{X}(1)$ be an ample line bundle on $X$ and let $P$ be a choice of Hilbert polynomial of degree $\mathrm{dim}(X)$. Then there is a canonical isomorphism
\begin{equation} \nonumber
\left( \mathcal{M}_{P}^{s} \right)^{T} \cong \coprod_{\vec{\chi} \in \left( \mathcal{X}_{P}^{0} \right)^{fr}} \mathcal{M}_{\vec{\chi}}^{0,s}. 
\end{equation}
\end{theorem}

\noindent Here the right hand side of the equation is a finite disjoint union of moduli spaces of torsion free equivariant sheaves on $X$. It is important to note that the moduli spaces on the right hand side are explicit and combinatorial in nature and their construction is very different from the construction of $\mathcal{M}_{P}^{s}$, which makes use of Quot schemes and requires boundedness results \cite[Ch.~1--4]{HL}. The theorem gives us a combinatorial description of $\left( \mathcal{M}_{P}^{s} \right)^{T}$ as a scheme. Explicit knowledge of $\left( \mathcal{M}_{P}^{s} \right)^{T}$ is useful for computing invariants associated to $\mathcal{M}_{P}^{s}$, e.g.~the Euler characteristic of $\mathcal{M}_{P}^{s}$, using localisation techniques. We exploit these ideas in a sequel \cite{Koo} in the case $X$ is a nonsingular complete toric surface to obtain expressions for generating functions of Euler characteristics of moduli spaces of $\mu$-stable torsion free sheaves on $X$. These computations can be used to study wall-crossing phenomena, i.e.~study the dependence of these generating functions on choice of ample line bundle $\mathcal{O}_{X}(1)$ on $X$. We will mention some of these results in this paper without further details. Most of the formulation and proof of the above theorem holds similarly for $P$ of any degree. The only complication arising in the general case is to find an equivariant line bundle in the GIT problem, which precisely reproduces Gieseker stability. Currently, we can only achieve this in full generality for $P$ of degree $\mathrm{dim}(X)$, i.e.~for torsion free sheaves, though we will develop the rest of the theory for arbitrary $P$ (Theorem \ref{ch. 1, sect. 4, thm. 1}). As a by-product, we will construct moduli spaces of $\mu$-stable reflexive equivariant sheaves on nonsingular projective toric varieties (Theorem \ref{ch. 1, sect. 4, thm. 2}) and express the fixed point loci of moduli spaces of $\mu$-stable reflexive sheaves on nonsingular projective toric varieties in terms of them (Theorem \ref{ch. 1, sect. 4, thm. 3}). In the case of reflexive equivariant sheaves, we will construct particularly simple ample equivariant line bundles in the GIT problem, which precisely recover $\mu$-stability. 

\noindent \emph{Acknowledgments}. I would like to thank Tom Bridgeland, Frances Kirwan, Yinan Song, Bal\'azs Szendr\H{o}i and Richard Thomas for invaluable discussions and suggestions. I would also like to thank Daniel Huybrechts and Sven Meinhardt for very useful discussions during my visit to Bonn. I would especially like to thank my supervisor Dominic Joyce for his continuous support. This paper is part of my D.Phil.~project funded by an EPSRC Studentship, which is part of EPSRC Grant EP/D077990/1. My stay at the Mathematical Institute of the University of Bonn was funded by the BIGS Exchange Programme for Ph.D.~Students and supervised by Daniel Huybrechts.

\section{Pure Equivariant Sheaves on Toric Varieties}

In this section, we will give a combinatorial description of pure equivariant sheaves on nonsingular toric varieties. After recalling the notion of an equivariant and a pure sheaf, we will give the combinatorial description in the affine case. Subsequently, we will pass to the general case. Our main tool will be Perling's notion of $\sigma$-families. In order to avoid cumbersome notation, we will first treat the case of irreducible support in detail and discuss the general case at the end. 

We recall the notion of a $G$-equivariant sheaf. 
\begin{definition} \label{ch. 1, sect. 2, def. 1} 
Let $G$ be an affine algebraic group acting regularly on a scheme\footnote{In this paper, all schemes will be schemes over $k$ an algebraically closed field of characteristic 0, unless stated otherwise.} $X$ of finite type over $k$. Denote the group action by $\sigma : G \times X \longrightarrow X$, denote projection to the second factor by $p_{2} : G \times X \longrightarrow X$ and denote multiplication on $G$ by $\mu : G \times G \longrightarrow G$. Moreover, denote projection to the last two factors by $p_{23} : G \times G \times X \longrightarrow G \times X$. Let $\mathcal{E}$ be a sheaf of $\mathcal{O}_{X}$-modules on $X$. A \emph{$G$-equivariant structure} on $\mathcal{E}$ is an isomorphism $\Phi : \sigma^{*}\mathcal{E} \longrightarrow p_{2}^{*}\mathcal{E}$ such that
\begin{equation}
(\mu \times 1_{X})^{*}\Phi = p_{23}^{*}\Phi \circ (1_{G} \times \sigma)^{*}\Phi. \nonumber
\end{equation}
This equation is called \emph{the cocycle condition}. A sheaf of $\mathcal{O}_{X}$-modules endowed with a $G$-equivariant structure is called a \emph{$G$-equivariant sheaf}. A \emph{$G$-equivariant morphism} from a $G$-equivariant sheaf $(\mathcal{E}, \Phi)$ to a $G$-equivariant sheaf $(\mathcal{F},\Psi)$ is a morphism $\theta : \mathcal{E} \longrightarrow \mathcal{F}$ of sheaves of $\mathcal{O}_{X}$-modules such that $p_{2}^{*}\theta \circ \Phi = \Psi \circ \sigma^{*}\theta$. We denote the $k$-vector space of $G$-equivariant morphisms from $(\mathcal{E}, \Phi)$ to $(\mathcal{F}, \Psi)$ by $G\textrm{-}\mathrm{Hom}(\mathcal{E},\mathcal{F})$. \hfill $\oslash$
\end{definition}

\noindent Using the above definition, we can form the $k$-linear additive category of $G$-equivariant sheaves which we will denote by $\mathrm{Mod}^{G}(X)$. Similarly, one can construct the categories of $G$-equivariant (quasi-)coherent sheaves $\mathrm{Qco}^{G}(X)$ and $\mathrm{Coh}^{G}(X)$. These are abelian categories and $\mathrm{Qco}^{G}(X)$ has enough injectives \cite[Ch.~V]{Toh}. 

Now let $X$ be a toric variety and $G = T$ is the algebraic torus\footnote{When dealing with toric geometry, we use the notation of the standard reference \cite{Ful}.}. Denote the fan by $\Delta$, the character group by $M = X(T)$ and the group of one-parameter subgroups by $N$ (so $M = N^{\vee}$ and we have a natural pairing between the two lattices $\langle \cdot, \cdot \rangle : M \times N \longrightarrow \mathbb{Z}$). The elements $\sigma$ of $\Delta$ are in bijective correspondence with the invariant affine open subsets $U_{\sigma}$ of $X$. In particular, for a strongly convex rational polyhedral cone $\sigma \in \Delta$ (which lies in the lattice $N$) we have $U_{\sigma} = \mathrm{Spec}(k[S_{\sigma}])$, where $k[S_{\sigma}]$ is the semigroup algebra associated to the semigroup $S_{\sigma}$ defined by
\begin{align} 
\begin{split}
&S_{\sigma} = \sigma^{\vee} \cap M, \\ 
&\sigma^{\vee} = \left\{ u \in M \otimes_{\mathbb{Z}} \mathbb{R} \ | \ \langle u,v \rangle \geq 0 \ \mathrm{for \ all} \ v \in \sigma \right\}. \nonumber
\end{split} 
\end{align}
We will denote the element of $k[M]$ corresponding to $m \in M$ by $\chi(m)$ and write the group operation on $k[M]$ multiplicatively, so $\chi(m)\chi(m')=\chi(m+m')$. We obtain the following $M$-graded $k$-algebras
\begin{equation} \label{ch. 1, eqn1}
\Gamma(U_{\sigma}, \mathcal{O}_{X}) = \bigoplus_{m \in S_{\sigma}} k \chi(m) \subset \bigoplus_{m \in M} k \chi(m) = \Gamma(T, \mathcal{O}_{X}).
\end{equation}
There is a regular action of $T$ on $\Gamma(U_{\sigma},\mathcal{O}_{X})$. For $t \in T$ a closed point and $f : U_{\sigma} \longrightarrow k$ a regular function, one defines
\begin{equation}
(t \cdot f )(x) = f(t \cdot x). \nonumber
\end{equation}
The regular action of $T$ on $U_{\sigma}$ induces a decomposition into weight spaces (Complete Reducibility Theorem \cite[Thm.~2.30]{Per1}). This decomposition coincides precisely with the decomposition in equation (\ref{ch. 1, eqn1}). More generally, if $(\mathcal{E}, \Phi)$ is an equivariant quasi-coherent sheaf on $X$, there is a natural regular action of $T$ on $\Gamma(U_{\sigma},\mathcal{E})$ \cite[Subsect.~2.2.2, Ch.~4]{Per1}. This action can be described as follows. For any closed point $t \in T$, let $i_{t} : X \hookrightarrow T \times X$ be the inclusion induced by $t \hookrightarrow T$ and define $\Phi_{t} = i_{t}^{*}\Phi : t^{*}\mathcal{E} \longrightarrow  \mathcal{E}$. From the cocycle condition, we obtain $\Phi_{st} = \Phi_{t} \circ t^{*}\Phi_{s}$ for all closed points $s,t \in T$ (see Definition \ref{ch. 1, sect. 2, def. 1}). Also, for $f \in \Gamma(U_{\sigma},\mathcal{E})$ we have a canonically lifted section $t^{*}f \in \Gamma(U_{\sigma},t^{*}\mathcal{E})$, which allows us to define
\begin{equation}
t \cdot f = \Phi_{t}(t^{*}f). \nonumber
\end{equation}
Again, we get a decomposition into weight spaces \cite[Thm.~2.30]{Per1}
\begin{equation}
\Gamma(U_{\sigma}, \mathcal{E}) = \bigoplus_{m \in M} \Gamma(U_{\sigma},\mathcal{E})_{m}. \nonumber
\end{equation}
In particular, for $\mathcal{E} = \mathcal{O}_{X}$, we obtain $\Gamma(U_{\sigma},\mathcal{O}_{X})_{m} = k \chi(m)$ if $m \in S_{\sigma}$ and $\Gamma(U_{\sigma},\mathcal{O}_{X})_{m} = 0$ otherwise. It is not difficult to deduce from the previous discussion that the functor $\Gamma(U_{\sigma}, -)$ induces an equivalence between the category of equivariant quasi-coherent (resp.~coherent) sheaves on $U_{\sigma}$ and the category of $M$-graded (resp.~finitely generated $M$-graded) $S_{\sigma}$-modules \cite[Prop.~2.31]{Per1}. 

Before we proceed and use the previous notions to give Perling's characterisation of equivariant quasi-coherent sheaves on affine toric varieties in terms of $\sigma$-families, we remind the reader of the notion of a pure sheaf.
\begin{definition} \label{pure}
Let $\mathcal{E} \neq 0$ be a coherent sheaf on a scheme $X$ of finite type over $k$. The sheaf $\mathcal{E}$ is said to be \emph{pure} of dimension $d$ if $\mathrm{dim}(\mathcal{F}) = d$ for any coherent subsheaf $0 \neq \mathcal{F} \subset \mathcal{E}$. Here the \emph{dimension} of a coherent sheaf $\mathcal{F}$ is defined to be the dimension of the support $\mathrm{Supp}(\mathcal{F})$ of the coherent sheaf $\mathcal{F}$. In the case $X$ is in addition integral, we also refer to a pure sheaf on $X$ of dimension $\mathrm{dim}(X)$ as a \emph{torsion free} sheaf on $X$. \hfill $\oslash$
\end{definition}

\noindent For future purposes, we state the following easy results.
\begin{proposition} \label{ch. 1, sect. 2, prop. 1}
Let $X$ be a scheme of finite type over $k$, $\{U_{i}\}$ an open cover of $X$ and $\mathcal{E} \neq 0$ a coherent sheaf on $X$. Then $\mathcal{E}$ is pure of dimension $d$ if and only if for each $i$ the restriction $\mathcal{E}|_{U_{i}}$ is zero or pure of dimension $d$.
\end{proposition}
\begin{proof}
The ``if'' part is trivial. Assume $\mathcal{E} \neq 0$ is pure of dimension $d$ but there is a coherent subsheaf $0 \neq \mathcal{F} \subset \mathcal{E}|_{U_{i}}$ having dimension $<d$ for some $i$. Let $Z = \overline{\mathrm{Supp}(\mathcal{F})}$ (where the bar denotes closure in $X$) and consider the coherent subsheaf $\mathcal{E}_{Z} \subset \mathcal{E}$ defined by $\mathcal{E}_{Z}(U) = \mathrm{ker}(\mathcal{E}(U) \longrightarrow \mathcal{E}(U \setminus Z))$ for all open subsets $U \subset X$. This sheaf is nonzero because $0 \neq \mathcal{F} \subset \mathcal{E}_{Z}|_{U_{i}}$ yet $\mathrm{Supp}(\mathcal{E}_{Z}) \subset Z$ so $\mathrm{dim}(\mathcal{E}_{Z}) < d$, contradicting purity.
\end{proof}
\begin{proposition} \label{lemma}
Let $X$, $Y$ be schemes of finite type over $k$ and let $X$ be reduced. Denote by $p_{2} : X \times Y \longrightarrow Y$ projection to the second component. Let $\mathcal{E}$ be a coherent sheaf on $X \times Y$, $\mathcal{F}$ a coherent sheaf on $Y$ and $\Phi, \Psi : \mathcal{E} \longrightarrow p_{2}^{*}\mathcal{F}$ morphisms. For any closed point $x \in X$, let $i_{x} : Y \hookrightarrow X \times Y$ be the inclusion induced by $x \hookrightarrow X$. If $i_{x}^{*} \Phi = i_{x}^{*}\Psi$ for all closed points $x \in X$, then $\Phi = \Psi$.
\end{proposition}
\begin{proof}
Using open affine covers, it is enough to prove the case $X = \mathrm{Spec}(R)$, $Y = \mathrm{Spec}(S)$, where $R,S$ are finitely generated $k$-algebras and $R$ has no nilpotent elements. Consider the finitely generated $R \otimes_{k}S$-module $E = \Gamma(X,\mathcal{E})$ and the finitely generated $S$-module $F = \Gamma(Y,\mathcal{F})$. Let $\Phi, \Psi : E \longrightarrow F \otimes_{k} R$ be the induced morphisms \cite[Prop.~II.5.2]{Har1}. Let $e \in E$ and let $\xi = \Phi(e) - \Psi(e)$. We need to prove $\xi = 0$. But we know that for any maximal ideal $\mathfrak{m} \subset R$, the induced morphism 
\begin{equation*}
F \otimes_{k}R \longrightarrow F \otimes_{k}R/\mathfrak{m} \cong F,
\end{equation*}  
maps $\xi$ to zero \cite[Prop.~II.5.2]{Har1}. Since $R$ has no nilpotent elements, the intersection of all its maximal ideals is zero $\bigcap_{\mathfrak{m} \subset R} \mathfrak{m} = \sqrt{(0)} = (0)$ (\cite[Prop.~1.8]{AM}, \cite[Thm.~4.19]{Eis}), hence $\xi=0$. 
\end{proof}
\begin{proposition} \label{ch. 1, sect. 2, prop. 2}
Let $X$ be a scheme of finite type over $k$, $G$ an affine algebraic group acting regularly on $X$ and $\mathcal{E} \neq 0$ a $G$-equivariant coherent sheaf on $X$. Then $\mathcal{E}$ is pure of dimension $d$ if and only if all its nontrivial $G$-equivariant coherent subsheaves have dimension $d$.
\end{proposition}
\begin{proof}
There is a unique filtration
\begin{equation}
0 \subset T_{0}(\mathcal{E}) \subset \cdots \subset T_{d}(\mathcal{E}) = \mathcal{E}, \nonumber
\end{equation}
where $T_{i}(\mathcal{E})$ is the maximal coherent subsheaf of $\mathcal{E}$ of dimension $\leq i$. This filtration is called the torsion filtration of $\mathcal{E}$ \cite[Sect.~1.1]{HL}. We claim each $T_{i}(\mathcal{E})$ is an equivariant coherent subsheaf of $\mathcal{E}$, i.e.~the morphism 
\begin{equation*}
\sigma^{*}(T_{i}(\mathcal{E})) \hookrightarrow \sigma^{*}(\mathcal{E}) \stackrel{\Phi}{\longrightarrow} p_{2}^{*}(\mathcal{E}),
\end{equation*}
factors through $p_{2}^{*}(T_{i}(\mathcal{E}))$. This would imply the proposition. By definition of $T_{i}(\mathcal{E})$, the morphism 
\begin{equation*}
g^{*}(T_{i}(\mathcal{E})) \hookrightarrow g^{*}(\mathcal{E}) \stackrel{i_{g}^{*}\Phi}{\longrightarrow} \mathcal{E},
\end{equation*}
factors through $T_{i}(\mathcal{E})$ for any closed point $g \in G$. The result now follows from Proposition \ref{lemma} applied to the morphisms 
\begin{align*}
&\sigma^{*}(T_{i}(\mathcal{E})) \hookrightarrow \sigma^{*}(\mathcal{E}) \stackrel{\Phi}{\longrightarrow} p_{2}^{*}(\mathcal{E}) \longrightarrow p_{2}^{*}(\mathcal{E} / T_{i}(\mathcal{E})), \\
&\sigma^{*}(T_{i}(\mathcal{E})) \stackrel{0}{\longrightarrow} p_{2}^{*}(\mathcal{E} / T_{i}(\mathcal{E})).
\end{align*}
\end{proof}

Let $U_{\sigma}$ be an affine toric variety defined by a cone $\sigma$ in a lattice $N$. We have already seen that $\Gamma(U_{\sigma},-)$ induces an equivalence between the category of equivariant quasi-coherent sheaves on $U_{\sigma}$ and the category of $M$-graded $k[S_{\sigma}]$-modules. The latter category can be conveniently reformulated using Perling's notion of a $\sigma$-family \cite[Def.~4.2]{Per1}.
\begin{definition}
For $m,m' \in M$, $m \leq_{\sigma} m^{\prime}$ means $m^{\prime} - m \in S_{\sigma}$. A \emph{$\sigma$-family} consists of the following data: a family of $k$-vector spaces $\{E_{m}^{\sigma}\}_{m \in M}$ and $k$-linear maps $\chi_{m,m^{\prime}}^{\sigma} : E_{m}^{\sigma} \longrightarrow E_{m^{\prime}}^{\sigma}$ for all $m \leq_{\sigma} m'$, such that $\chi_{m,m}^{\sigma} = 1$ and $\chi_{m,m^{\prime\prime}}^{\sigma} = \chi_{m^{\prime},m^{\prime\prime}}^{\sigma} \circ \chi_{m,m^{\prime}}^{\sigma}$ for all $m \leq_{\sigma} m^{\prime} \leq_{\sigma} m^{\prime\prime}$. A \emph{morphism of $\sigma$-families} $\hat{\phi}^{\sigma} : \hat{E}^{\sigma} \longrightarrow \hat{F}^{\sigma}$ is a family of $k$-linear maps $\{\phi_{m} : E_{m}^{\sigma} \longrightarrow F_{m}^{\sigma} \}_{m \in M}$, such that $\phi_{m^{\prime}}^{\sigma} \circ (\chi_{E})_{m,m^{\prime}}^{\sigma}= (\chi_{F})_{m,m^{\prime}}^{\sigma} \circ \phi_{m}^{\sigma}$ for all $m \leq_{\sigma} m'$. \hfill $\oslash$
\end{definition}

\noindent Let $(\mathcal{E},\Phi)$ be an equivariant quasi-coherent sheaf on $U_{\sigma}$. Denote the corresponding $M$-graded $k[S_{\sigma}]$-module by $E^{\sigma} = \bigoplus_{m \in M} E_{m}^{\sigma}$. This gives us a $\sigma$-family $\{E_{m}^{\sigma}\}_{m \in M}$ by taking
\begin{equation}
\chi_{m,m'}^{\sigma} : E_{m}^{\sigma} \longrightarrow E_{m'}^{\sigma}, \ \chi_{m,m'}^{\sigma}(e) = \chi(m'-m) \cdot e, \nonumber
\end{equation}
for all $m \leq_{\sigma} m'$. This establishes an equivalence between the category of equivariant quasi-coherent sheaves on $U_{\sigma}$ and the category of $\sigma$-families \cite[Thm.~4.5]{Per1}. 

Recall that an affine toric variety $U_{\sigma}$ defined by a cone $\sigma$ of dimension $s$ in a lattice $N$ of rank $r$ is nonsingular if and only if $\sigma$ is generated by part of a $\mathbb{Z}$-basis for $N$. Assume this is the case, then $U_{\sigma} \cong k^{s} \times (k^{*})^{r-s}$. Let $\sigma(1) = \{\rho_{1}, \ldots, \rho_{s}\}$ be the rays (i.e.~1-dimensional faces) of $\sigma$. Let $n(\rho_{i})$ be the first integral lattice point on the ray $\rho_{i}$. Then $(n(\rho_{1}), \ldots, n(\rho_{s}))$ is part of a $\mathbb{Z}$-basis for $N$. Let $(m(\rho_{1}), \ldots, m(\rho_{s}))$ be the corresponding part of a dual basis for $M$. The cosets $([m(\rho_{1})], \ldots, [m(\rho_{s})])$ form a $\mathbb{Z}$-basis for $M/S_{\sigma}^{\perp}$. Here $S_{\sigma}^{\perp}$ denotes the subgroup $S_{\sigma}^{\perp} = \sigma^{\perp} \cap M$, where $\sigma^{\perp} = \{u \in M \otimes_{\mathbb{Z}} \mathbb{R} \ | \ \langle u,v \rangle = 0 \ \mathrm{for \ all} \ v \in \sigma \}$. We obtain $M / S_{\sigma}^{\perp} \cong \mathbb{Z}^{s}$. Let $\hat{E}^{\sigma}$ be a $\sigma$-family. We can repackage the data in $\hat{E}^{\sigma}$ somewhat more efficiently as follows. First of all, note that for all $m'-m \in S_{\sigma}^{\perp}$, the $k$-linear map $\chi_{m,m'}^{\sigma} : E_{m}^{\sigma} \longrightarrow E_{m'}^{\sigma}$ is an isomorphism, so we might just as well restrict attention to $\sigma$-families having $\chi_{m,m'}^{\sigma} = 1$ (and hence $E_{m}^{\sigma} = E_{m'}^{\sigma}$) for all $m'-m \in S_{\sigma}^{\perp}$. We can then rewrite for any $\lambda_{1}, \ldots, \lambda_{s} \in \mathbb{Z}$
\begin{align} 
\begin{split}
&E^{\sigma}(\lambda_{1}, \ldots, \lambda_{s}) = E_{m}^{\sigma}, \ \mathrm{where} \ m=\sum_{i=1}^{s} \lambda_{i} m(\rho_{i}), \\
&\chi_{1}^{\sigma}(\lambda_{1},\ldots,\lambda_{s}) : E^{\sigma}(\lambda_{1}, \ldots, \lambda_{s}) \longrightarrow E^{\sigma}(\lambda_{1}+1,\lambda_{2},\ldots, \lambda_{s}), \\
&\chi_{1}^{\sigma}(\lambda_{1},\ldots,\lambda_{s}) = \chi_{m,m'}^{\sigma}, \ \mathrm{where} \ m = \sum_{i=1}^{s} \lambda_{i}m(\rho_{i}), \ m' = m(\rho_{1}) + m, \\
&\ldots \nonumber
\end{split}
\end{align}
When we would like to suppress the domain, we also denote these maps somewhat sloppily by $x_{1} \cdot = \chi_{1}^{\sigma}(\lambda_{1}, \ldots, \lambda_{s})$, $\ldots$, $x_{s} \cdot = \chi_{s}^{\sigma}(\lambda_{1}, \ldots, \lambda_{s})$. These $k$-linear maps satisfy $x_{i}x_{j} = x_{j}x_{i}$ for all $i,j = 1, \ldots, s$. The equivalence between the category of equivariant quasi-coherent sheaves on $U_{\sigma}$ and the category of $\sigma$-families restricts to an equivalence between the full subcategories of equivariant coherent sheaves on $U_{\sigma}$ and finite $\sigma$-families (see \cite[Def.~4.10, Prop.~4.11]{Per1}). A finite $\sigma$-family is a $\sigma$-family $\hat{E}^{\sigma}$ such that all $E^{\sigma}(\lambda_{1}, \ldots, \lambda_{s})$ are finite-dimensional $k$-vector spaces, there are $A_{1}, \ldots, A_{s} \in \mathbb{Z}$ such that $E^{\sigma}(\lambda_{1}, \ldots, \lambda_{s}) = 0$ unless $\lambda_{1} \geq A_{1}$, $\ldots$, $\lambda_{s} \geq A_{s}$ and there are only finitely many $(\Lambda_{1}, \ldots, \Lambda_{s}) \in \mathbb{Z}^{s}$ such that  
\begin{align} 
\begin{split}
&E^{\sigma}(\Lambda_{1}, \ldots, \Lambda_{s}) \\
&\neq \mathrm{span}_{k}\left\{ x_{1}^{\Lambda_{1} - \lambda_{1}} \cdots x_{s}^{\Lambda_{s} - \lambda_{s}} e \ \right| \left. \ e \in E^{\sigma}(\lambda_{1}, \ldots, \lambda_{s}) \ \mathrm{with} \ \Lambda_{i} - \lambda_{i} \geq 0, \ \mathrm{not \ all \ } 0 \right\}. \nonumber
\end{split}
\end{align}

\subsection{Combinatorial Descriptions in the Case of Irreducible Support}

Going from affine toric varieties to general toric varieties, Perling introduces the notion of $\Delta$-families \cite[Sect.~4.2]{Per1}, which are basically collections of $\sigma$-families, for all cones $\sigma$ in the fan $\Delta$, satisfying certain compatibility conditions. We will not use this notion. Instead, we will first study pure equivariant sheaves on nonsingular affine toric varieties and then use gluing to go to general toric varieties. In order to avoid heavy notation, we will restrict to the case of irreducible support and defer the general case to the next subsection. Recall that for a toric variety $X$ defined by a fan $\Delta$ in a lattice $N$, there is a bijective correspondence between the elements of $\Delta$ and the invariant closed (irreducible) subvarieties of $X$ \cite[Sect.~3.1]{Ful}. The correspondence associates to a cone $\sigma \in \Delta$ the invariant closed subvariety $V(\sigma) \subset X$, which is defined to be the closure in $X$ of the unique orbit of minimal dimension in $U_{\sigma}$. If $\mathrm{dim}(\sigma)=s$, then $\mathrm{codim}(V(\sigma))=s$.
\begin{proposition} \label{ch. 1, sect. 2, prop. 3}
Let $U_{\sigma}$ be a nonsingular affine toric variety defined by a cone\footnote{From now on, in this setting we will always assume $\mathrm{dim}(\sigma) = r$, so $U_{\sigma} \cong \mathbb{A}^{r}$.} $\sigma$ in a lattice $N$ of rank $r$. Let $\mathcal{E} \neq 0$ be an equivariant coherent sheaf on $U_{\sigma}$ with irreducible support. Then $\mathrm{Supp}(\mathcal{E}) = V(\tau)$, for some $\tau \prec \sigma$. Now fix $\tau \prec \sigma$, let $(\rho_{1}, \ldots, \rho_{r})$ be the rays of $\sigma$ and $(\rho_{1}, \ldots, \rho_{s}) \subset (\rho_{1}, \ldots, \rho_{r})$ the rays of $\tau$. Then $\mathrm{Supp}(\mathcal{E}) = V(\tau)$ if and only if there are integers $B_{1}, \ldots, B_{s}$ such that $E^{\sigma}(\lambda_{1}, \ldots, \lambda_{r}) = 0$ unless $\lambda_{1} \leq B_{1}$, $\ldots$, $\lambda_{s} \leq B_{s}$, but for each $\lambda_{i} \neq \lambda_{1}, \ldots, \lambda_{s}$ there is no such upper bound.
\end{proposition}
\begin{proof}
Note that $V(\tau)$ is defined by the prime ideal $\mathcal{I}_{\tau} = \langle \chi(m(\rho_{1})), \ldots, \chi( m (\rho_{s}) ) \rangle$. Define the open subset $U = U_{\sigma} \setminus V(\tau) = D(\chi(m(\rho_{1}))) \cup \cdots \cup D(\chi(m(\rho_{s})))$, where $D(\chi(m(\rho_{i})))$ is the set of all prime ideals not containing $\chi(m(\rho_{i}))$. The open subset $D(\chi(m(\rho_{i}))) = \mathrm{Spec}(k[S_{\sigma}][\chi(-m(\rho_{i}))])$. Clearly \cite[Prop.~II.5.2]{Har1}
\begin{align} 
\begin{split}
\mathrm{Supp}(\mathcal{E}) \subset V(\tau) &\Longleftrightarrow \mathcal{E}|_{D(\chi(m(\rho_{1})))} = \cdots = \mathcal{E}|_{D(\chi(m(\rho_{s})))} = 0 \\
&\Longleftrightarrow \Gamma(U_{\sigma}, \mathcal{E}) \otimes_{k[S_{\sigma}]} k[S_{\sigma}][\chi(-m(\rho_{1}))]  =  0 \\
&\qquad \, \, \, \cdots \\
&\qquad \, \, \, \Gamma(U_{\sigma}, \mathcal{E}) \otimes_{k[S_{\sigma}]} k[S_{\sigma}][\chi(-m(\rho_{s}))] = 0. \nonumber
\end{split}
\end{align}
Since $\Gamma(U_{\sigma}, \mathcal{E})$ is finitely generated, we in fact have
\begin{align} 
\begin{split}
\mathrm{Supp}(\mathcal{E}) \subset V(\tau) \Longleftrightarrow & \exists \kappa_{1}, \ldots, \kappa_{s} \in \mathbb{Z}_{>0} \\ 
&\chi(m(\rho_{1}))^{\kappa_{1} } \Gamma(U_{\sigma}, \mathcal{E}) = \cdots = \chi(m_{\rho_{s}})^{\kappa_{s} } \Gamma(U_{\sigma}, \mathcal{E}) = 0. \nonumber
\end{split}
\end{align}
The proof now easily follows from the fact that the $\sigma$-family corresponding to $\mathcal{E}$ is finite.
\end{proof}
\noindent The following proposition describes pure equivariant sheaves with irreducible support on nonsingular affine toric varieties\footnote{The author would like to thank the referee for indicating the current proof of this proposition, which is quicker than the original.}.
\begin{proposition} \label{ch. 1, sect. 2, prop. 4}
Let $U_{\sigma}$ be a nonsingular affine toric variety defined by a cone $\sigma$ in a lattice $N$ of rank $r$. Let $\tau \prec \sigma$, let $(\rho_{1}, \ldots, \rho_{r})$ be the rays of $\sigma$ and $(\rho_{1}, \ldots, \rho_{s}) \subset (\rho_{1}, \ldots, \rho_{r})$ the rays of $\tau$. Then the category of pure equivariant sheaves $\mathcal{E}$ on $U_{\sigma}$ with support $V(\tau)$ is equivalent to the category of $\sigma$-families $\hat{E}^{\sigma}$ having the following properties:
\begin{enumerate}
	\item [$\mathrm{(i)}$] There are integers $A_{1} \leq B_{1}, \ldots, A_{s} \leq B_{s}, A_{s+1}, \ldots, A_{r}$ such that $E^{\sigma}(\lambda_{1}, \ldots \lambda_{r}) = 0$ unless $A_{1} \leq \lambda_{1} \leq B_{1}$, $\ldots$, $A_{s} \leq \lambda_{s} \leq B_{s}$, $A_{s+1} \leq \lambda_{s+1}$, $\ldots$, $A_{r} \leq \lambda_{r}$. 
	\item[$\mathrm{(ii)}$] For all integers $A_{1} \leq \Lambda_{1} \leq B_{1}$, $\ldots$, $A_{s} \leq \Lambda_{s} \leq B_{s}$, there is a finite dimensional $k$-vector space $E^{\sigma}(\Lambda_{1}, \ldots, \Lambda_{s}, \infty, \ldots, \infty)$ (not all of them zero) satisfying the following properties. All vector spaces  $E^{\sigma}(\Lambda_{1}, \ldots, \Lambda_{s}, \lambda_{s+1}, \ldots, \lambda_{r})$ are subspaces of $E^{\sigma}(\Lambda_{1}, \ldots, \Lambda_{s}, \infty, \ldots, \infty)$ and the maps $x_{s+1}, \ldots, x_{r}$ are inclusions. Moreover, there are integers $\lambda_{s+1}, \ldots, \lambda_{r}$ such that we have $E^{\sigma}(\Lambda_{1}, \ldots, \Lambda_{s}, \lambda_{s+1}, \ldots, \lambda_{r}) = E^{\sigma}(\Lambda_{1}, \ldots, \Lambda_{s}, \infty, \ldots, \infty)$.
\end{enumerate}
\end{proposition}
\begin{proof}
Let $\mathcal{E}$ be a pure equivariant sheaf with support $V(\tau)$ and corresponding $\sigma$-family $\hat{E}^{\sigma}$. Then (i) follows from Proposition \ref{ch. 1, sect. 2, prop. 3}. For (ii), it is enough to prove $x_{s+1}, \ldots, x_{r}$ are injective (the rest follows from the fact that $\hat{E}^{\sigma}$ is finite). Let $x_{i}$ not be injective. Let $E = \Gamma(U_{\sigma}, \mathcal{E})$ be the module corresponding to $\mathcal{E}$ and define $0 \neq F \subset E$ to be the submodule of elements annihilated by $x_{i}$. From purity, we deduce $x_{i} \in \mathrm{Ann}(F) \subset (x_{1}, \ldots, x_{s})$ so $i = 1, \ldots, s$. 

Conversely, let $\mathcal{E}$ be an equivariant quasi-coherent sheaf with corresponding $\sigma$-family $\hat{E}^{\sigma}$ as in (i), (ii).  It is easy to see that $\mathcal{E}$ is coherent and $\mathrm{Supp}(\mathcal{E}) \subset V(\tau)$ (see also proof of Proposition \ref{ch. 1, sect. 2, prop. 3}). It is enough to show that any nontrivial equivariant coherent subsheaf $\mathcal{F} \subset \mathcal{E}$ has support $V(\tau)$ by Proposition \ref{ch. 1, sect. 2, prop. 2}. Suppose not, and let $F = \Gamma(U_{\sigma}, \mathcal{F})$ be the module corresponding to $\mathcal{F}$. Then $\mathrm{Ann}(F) \nsubseteq (x_{1}, \ldots, x_{s})$. Hence there is a monomial $x_{s+1}^{\kappa_{1}} \cdots x_{r}^{\kappa_{r}}$ for some $\kappa_{i} > 0$ annihilating some nonzero homogeneous element of $F \subset E$, which contradicts injectivity of $x_{s+1}, \ldots, x_{r}$.  
\end{proof}
In order to generalise the result of the previous proposition to arbitrary nonsingular toric varieties, we need the following proposition for gluing purposes.
\begin{proposition} \label{ch. 1, sect. 2, prop. 5}
Let $U_{\sigma}$ be a nonsingular affine toric variety defined by a cone $\sigma$ in a lattice $N$ of rank $r$. Let $\mathcal{E}$ be a pure equivariant sheaf on $U_{\sigma}$ with support $V(\tau)$ where $\tau \prec \sigma$. Let $(\rho_{1}, \ldots, \rho_{r})$ be the rays of $\sigma$ and $(\rho_{1}, \ldots, \rho_{s}) \subset (\rho_{1}, \ldots, \rho_{r})$ the rays of $\tau$. Let $\nu \prec \sigma$ be a proper face and consider the equivariant coherent sheaf $\mathcal{E}|_{U_{\nu}}$. Then the $\nu$-family corresponding to $\mathcal{E}|_{U_{\nu}}$ is described in terms of the $\sigma$-family corresponding to $\mathcal{E}$ as follows:
\begin{enumerate}
	\item [$\mathrm{(i)}$] Assume $\tau$ is not a face of  $\nu$. Then $\mathcal{E}|_{U_{\nu}} = 0$.
	\item [$\mathrm{(ii)}$] Assume $\tau \prec \nu$. Let $(\rho_{1}, \ldots, \rho_{s}, \rho_{s+1}, \ldots, \rho_{s+t}) \subset (\rho_{1}, \ldots, \rho_{r})$ be the rays of $\nu$. Then for all $\lambda_{1}, \ldots, \lambda_{s+t} \in \mathbb{Z}$ we have
	\begin{align} 
  \begin{split}
	E^{\nu}(\lambda_{1}, \ldots, \lambda_{s+t}) &= E^{\sigma}(\lambda_{1}, \ldots, \lambda_{s+t}, \infty, \ldots, \infty), \\
	\chi^{\nu}_{i}(\lambda_{1}, \ldots, \lambda_{s+t}) &= \chi^{\sigma}_{i}(\lambda_{1}, \ldots, \lambda_{s+t}, \infty, \ldots, \infty), \ \forall i = 1, \ldots, s+t. \nonumber
	\end{split}
	\end{align}
\end{enumerate}
\end{proposition}
\begin{proof}
There is an integral element $m_{\nu} \in \mathrm{relative \ interior}(\nu^{\perp} \cap \sigma^{\vee})$, such that $S_{\nu} = S_{\sigma} + \mathbb{Z}_{\geq 0}(-m_{\nu})$ (e.g.~\cite[Thm.~3.14]{Per1}). Let $\rho_{i_{1}}, \ldots, \rho_{i_{p}}$ be the rays of $\nu$ and let $\rho_{j_{1}}, \ldots, \rho_{j_{q}}$ be all the other rays (so $p+q = r$). Then 
\begin{equation}
m_{\nu} = \sum_{k = 1}^{q} \gamma_{k} m(\rho_{j_{k}}), \nonumber
\end{equation}
where all $\gamma_{k} > 0$ integers. We obtain \cite[Prop.~II.5.2]{Har1}
\begin{align} \label{ch. 1, eqn2}
\begin{split}
\Gamma(U_{\nu},\mathcal{E}|_{U_{\nu}}) &\cong \Gamma(U_{\sigma}, \mathcal{E}) \otimes_{k[S_{\sigma}]} k[S_{\nu}] \\
&= \Gamma(U_{\sigma}, \mathcal{E}) \otimes_{k[S_{\sigma}]} k[S_{\sigma}][\chi(-m(\rho_{j_{1}}))^{\gamma_{j_{1}}}, \ldots, \chi(-m(\rho_{j_{q}}))^{\gamma_{j_{q}}}].
\end{split}
\end{align}
\emph{Case 1: $\tau$ is not a face of $\nu$.} Trivial because $V(\tau) \cap U_{\nu} = \varnothing$. 

\noindent \emph{Case 2: $\tau \prec \nu$.} In this case, we can number the rays $\rho_{i_{1}}, \ldots, \rho_{i_{p}}$ of $\nu$ as follows $(\rho_{1}, \ldots, \rho_{s}, \rho_{s+1}, \ldots, \rho_{s+t})$. Assume $\mathcal{E}$ is described by a $\sigma$-family $\hat{E}^{\sigma}$ as in Proposition \ref{ch. 1, sect. 2, prop. 4}. Note that $\Gamma(U_{\sigma}, \mathcal{E}) \otimes_{k[S_{\sigma}]} k[S_{\nu}]$ has a natural $M$-grading \cite[Sect.~2.5]{Per1}. In particular, for a fixed $m \in M$, the elements of degree $m$ are finite sums of expressions of the form $e \otimes \chi(m'')$, where $e \in E_{m'}^{\sigma}$, $m' \in M$, $m'' \in S_{\nu}$ such that $m'+m'' = m$. Now fix $m = \sum_{i=1}^{r} \lambda_{i} m(\rho_{i}) \in M$, $m' = \sum_{i=1}^{r} \alpha_{i} m(\rho_{i}) \in M$ and $m'' \in S_{\nu}$, so $m''=\sum_{i=1}^{r} \beta_{i}m(\rho_{i}) - u \sum_{i=s+t+1}^{r} \gamma_{i}m(\rho_{i})$ with $\beta_{1}, \ldots, \beta_{r}, u \geq 0$. Assume $m = m'+m''$ and consider the element $e \otimes \chi(m'')$ with $e \in E_{m'}^{\sigma}$. We can now rewrite $e \otimes \chi(m'') = e' \otimes \chi(m''')$, where
\begin{align} 
\begin{split}
e' &= \chi \left( \sum_{i=1}^{r} \beta_{i} m(\rho_{i}) \right) \cdot e \in E^{\sigma}(\lambda_{1}, \ldots, \lambda_{s+t},\alpha_{s+t+1} + \beta_{s+t+1}, \ldots, \alpha_{r}+\beta_{r} ), \\
\chi(m''') &= \chi\left(-u \sum_{i=s+t+1}^{r} \gamma_{i} m(\rho_{i}) \right). \nonumber
\end{split}
\end{align}
For $v > 0$ large enough 
\begin{align} 
\begin{split}
&e' \otimes \chi(m''') = \chi\left( v \sum_{i=s+t+1}^{r} \gamma_{i} m(\rho_{i}) \right) \cdot e' \otimes  \chi\left(-(u+v) \sum_{i=s+t+1}^{r} \gamma_{i} m(\rho_{i}) \right), \\
&\mathrm{where} \ \chi\left( v \sum_{i=s+t+1}^{r} \gamma_{i} m(\rho_{i}) \right) \cdot e' \in E^{\sigma}(\lambda_{1}, \ldots, \lambda_{s+t}, \infty, \ldots, \infty). \nonumber
\end{split}
\end{align}
From these remarks, one easily deduces the assertion.
\end{proof}

\noindent As a special case of the above proposition we get the following result. If we take $\nu = \tau$, then we obtain that for all integers $\lambda_{1}, \ldots, \lambda_{r}$
\begin{equation}
E^{\sigma}(\lambda_{1}, \ldots, \lambda_{r}) \subset E^{\sigma}(\lambda_{1}, \ldots, \lambda_{s}, \infty, \ldots, \infty) = E^{\tau}(\lambda_{1}, \ldots, \lambda_{s}). \nonumber
\end{equation}
We conclude that all $E^{\sigma}(\lambda_{1}, \ldots, \lambda_{r})$ are subspaces of $E^{\tau}(\lambda_{1}, \ldots, \lambda_{s})$. 

Combining Propositions \ref{ch. 1, sect. 2, prop. 4} and \ref{ch. 1, sect. 2, prop. 5}, we obtain a combinatorial description of pure equivariant sheaves with irreducible support on nonsingular toric varieties.
\begin{theorem} \label{ch. 1, sect. 2, thm. 1}
Let $X$ be a nonsingular toric variety with fan\footnote{From now on, in this setting we will always assume every cone of $\Delta$ is contained in a cone of dimension $r$. Therefore, we can cover $X$ by copies of $\mathbb{A}^{r}$.} $\Delta$ in a lattice $N$ of rank $r$. Let $\tau \in \Delta$ and consider the invariant closed subvariety $V(\tau)$. It is covered by $U_{\sigma}$, where $\sigma \in \Delta$ has dimension $r$ and $\tau \prec \sigma$. Denote these cones by $\sigma_{1}, \ldots, \sigma_{l}$. For each $i = 1, \ldots, l$, let $\left(\rho^{(i)}_{1}, \ldots, \rho^{(i)}_{r} \right)$ be the rays of $\sigma_{i}$ and let $\left(\rho_{1}^{(i)}, \ldots, \rho_{s}^{(i)}\right) \subset \left(\rho^{(i)}_{1}, \ldots, \rho^{(i)}_{r}\right)$ be the rays of $\tau$. The category of pure equivariant sheaves on $X$ with support $V(\tau)$ is equivalent to the category $\mathcal{C}^{\tau}$, which can be described as follows. An object $\hat{E}^{\Delta}$ of $\mathcal{C}^{\tau}$ consists of the following data:
\begin{enumerate}
	\item [$\mathrm{(i)}$] For each $i = 1,\ldots, l$ we have a $\sigma_{i}$-family $\hat{E}^{\sigma_{i}}$ as described in Proposition \ref{ch. 1, sect. 2, prop. 4}.
	\item [$\mathrm{(ii)}$] Let $i,j = 1, \ldots, l$. Let $\left\{\rho^{(i)}_{i_{1}}, \ldots, \rho^{(i)}_{i_{p}}\right\} \subset \left\{\rho^{(i)}_{1}, \ldots, \rho^{(i)}_{r}\right\}$ resp.~$\left\{\rho^{(j)}_{j_{1}}, \ldots, \rho^{(j)}_{j_{p}}\right\} \subset \left\{\rho^{(j)}_{1}, \ldots, \rho^{(j)}_{r}\right\}$ be the rays of $\sigma_{i} \cap \sigma_{j}$ in $\sigma_{i}$ respectively $\sigma_{j}$, labeled in such a way that $\rho^{(i)}_{i_{k}} = \rho^{(j)}_{j_{k}}$ for all $k = 1, \ldots, p$. Now let $\lambda^{(i)}_{1}, \ldots, \lambda^{(i)}_{r} \in \mathbb{Z} \cup \{\infty\}$, $\lambda^{(j)}_{1}, \ldots, \lambda^{(j)}_{r} \in \mathbb{Z} \cup \{\infty\}$ be such that $\lambda^{(i)}_{i_{k}} = \lambda^{(j)}_{j_{k}} \in \mathbb{Z}$ for all $k = 1, \ldots, p$ and $\lambda^{(i)}_{n} = \lambda^{(j)}_{n} = \infty$ otherwise. Then
  \begin{align} 
  \begin{split}
	E^{\sigma_{i}} \left( \sum_{k=1}^{r} \lambda^{(i)}_{k} m\left(\rho_{k}^{(i)}\right) \right) &= E^{\sigma_{j}} \left( \sum_{k=1}^{r} \lambda^{(j)}_{k} m\left(\rho_{k}^{(j)}\right) \right), \\
	\chi_{n}^{\sigma_{i}} \left( \sum_{k=1}^{r} \lambda^{(i)}_{k} m\left(\rho_{k}^{(i)}\right) \right) &= \chi_{n}^{\sigma_{j}} \left( \sum_{k=1}^{r} \lambda^{(j)}_{k} m\left(\rho_{k}^{(j)}\right) \right), \ \forall n = 1, \ldots, r. \nonumber
	\end{split}
	\end{align}
\end{enumerate}
The morphisms of $\mathcal{C}^{\tau}$ are described as follows. If $\hat{E}^{\Delta}$, $\hat{F}^{\Delta}$ are two objects, then a morphism $\hat{\phi}^{\Delta} : \hat{E}^{\Delta} \longrightarrow \hat{F}^{\Delta}$ is a collection of morphisms of $\sigma$-families $\{\hat{\phi}^{\sigma_{i}} : \hat{E}^{\sigma_{i}} \longrightarrow \hat{F}^{\sigma_{i}}\}_{i = 1, \ldots, l}$ such that for all $i,j$ as in $\mathrm{(ii)}$ one has
\begin{equation}
\phi^{\sigma_{i}}\left( \sum_{k=1}^{r} \lambda^{(i)}_{k} m\left(\rho_{k}^{(i)}\right) \right) = \phi^{\sigma_{j}}\left( \sum_{k=1}^{r} \lambda^{(j)}_{k} m\left(\rho_{k}^{(j)}\right) \right). \nonumber
\end{equation}
\end{theorem}
\begin{proof}
Note that $V(\tau)$ is covered by the star of $\tau$, i.e.~the cones $\sigma \in \Delta$ such that $\tau \prec \sigma$ \cite[Sect.~3.1]{Ful}. Let $\sigma_{1}, \ldots, \sigma_{l} \in \Delta$ be the cones of maximal dimension in the star of $\tau$. Let $\mathcal{E}$ be a pure equivariant sheaf on $X$ with support $V(\tau)$. Then $\mathcal{E}|_{U_{\sigma_{i}}}$ is a pure equivariant sheaf on $U_{\sigma_{i}}$ with support $V(\tau) \cap U_{\sigma_{i}}$ for all $i = 1, \ldots, l$ (using Proposition \ref{ch. 1, sect. 2, prop. 1}). Using Proposition \ref{ch. 1, sect. 2, prop. 4}, we get a $\sigma_{i}$-family $\hat{E}^{\sigma_{i}}$ for all $i = 1, \ldots, l$ (this gives (i) of the theorem). Using Proposition \ref{ch. 1, sect. 2, prop. 5}, we see that these $\sigma$-families have to glue as in (ii) (up to isomorphism).
\end{proof}

\noindent In the above theorem, we will refer to the category $\mathcal{C}^{\tau}$ as the category of pure $\Delta$-families with support $V(\tau)$. If we take $\tau = 0$ to be the apex in this theorem, we obtain the known combinatorial description of torsion free equivariant sheaves on nonsingular toric varieties initially due to Klyachko \cite{Kly4} and also discussed by Knutson and Sharpe \cite[Sect.~4.5]{KS1} and Perling \cite[Subsect.~4.4.2]{Per1}. The theorem generalises this description. In the case $\tau = 0$ is the apex, we will refer to the category $\mathcal{C}^{0}$ as the category of torsion free $\Delta$-families. In the above theorem, denote by $\mathcal{C}^{\tau, fr}$ the full subcategory of $\mathcal{C}^{\tau}$ consisting of those elements having all limiting vector spaces $E^{\sigma_{i}}(\Lambda_{1}, \ldots, \Lambda_{s}, \infty, \ldots, \infty)$ equal to $k^{\oplus r}$ for some $r$. We refer to $\mathcal{C}^{\tau, fr}$ as the category of framed pure $\Delta$-families with support $V(\tau)$. This notion does not make much sense now because $\mathcal{C}^{\tau, fr}$ is equivalent to $\mathcal{C}^{\tau}$, but framing will become relevant when looking at families.

\subsection{Combinatorial Descriptions in the General Case}

The results of the previous subsection generalise in a straightforward way to the case of general --not necessarily irreducible-- support. Since the proofs will require no essentially new ideas, we will just discuss the results. 

Let us first discuss the generalisation of Proposition \ref{ch. 1, sect. 2, prop. 3}. Let $U_{\sigma}$ be a nonsingular affine toric variety defined by a cone $\sigma$ in a lattice $N$ of rank $r$. Let $\mathcal{E} \neq 0$ be an equivariant coherent sheaf on $U_{\sigma}$. Then $\mathrm{Supp}(\mathcal{E}) = V(\tau_{1}) \cup \cdots \cup V(\tau_{a})$ for some faces $\tau_{1}, \ldots, \tau_{a} \prec \sigma$. Now fix faces $\tau_{1}, \ldots, \tau_{a} \prec \sigma$, let $(\rho_{1}, \ldots, \rho_{r})$ be the rays of $\sigma$ and let $\left(\rho_{1}^{(\alpha)}, \ldots, \rho_{s_{\alpha}}^{(\alpha)} \right) \subset (\rho_{1}, \ldots, \rho_{r})$ be the rays of $\tau_{\alpha}$ for all $\alpha = 1, \ldots, a$. Let $\tau_{\alpha} \nprec \tau_{\beta}$ for all $\alpha, \beta = 1, \ldots, a$ with $\alpha \neq \beta$. Then $\mathrm{Supp}(\mathcal{E}) = V(\tau_{1}) \cup \cdots \cup V(\tau_{a})$ if and only if the following property holds: 

\noindent $E^{\sigma}(\lambda_{1}, \ldots, \lambda_{r}) = 0$ unless $(\lambda_{1}, \ldots, \lambda_{r}) \in \mathcal{R}$ where $\mathcal{R} \subset N$ is defined by inequalities as follows: there are integers $A_{1}, \ldots, A_{r}$ and integers $B_{1}^{(\alpha)}, \ldots, B_{s_{\alpha}}^{(\alpha)}$ for each $\alpha = 1, \ldots, a$ such that the region $\mathcal{R}$ is defined by 
\begin{align} 
\begin{split}
&[A_{1} \leq \lambda_{1} \wedge \cdots \wedge A_{r} \leq \lambda_{r})] \\
&\wedge [(\lambda_{1}^{(1)} \leq B_{1}^{(1)} \wedge \cdots \wedge \lambda_{s_{1}}^{(1)} \leq B_{s_{1}}^{(1)}) \vee \cdots \vee (\lambda_{1}^{(a)} \leq B_{1}^{(a)} \wedge \cdots \wedge \lambda_{s_{a}}^{(a)} \leq B_{s_{a}}^{(a)})], \nonumber
\end{split}
\end{align}
moreover, there is no region $\mathcal{R}'$ of such a form with more upper bounds contained in $\mathcal{R}$ with the same property. Here $\lambda_{i}^{(j)}$ corresponds to the coordinate associated to the ray $\rho_{i}^{(j)}$ defined above. 

\noindent Note that if we assume in addition that $\mathcal{E}$ is pure, then all the $V(\tau_{\alpha})$ have the same dimension so $s_{1} = \cdots = s_{a} = s$. If $\mathrm{dim}(\mathcal{E}) = d$, then $\mathrm{Supp}(\mathcal{E}) = V(\tau_{1}) \cup \cdots \cup V(\tau_{a})$, where $\tau_{1}, \ldots, \tau_{a}$ are some faces of $\sigma$ of dimension $s = r - d$. One possible support would be taking $\tau_{1}, \ldots, \tau_{a}$ all faces of $\sigma$ of dimension $s = r - d$. In this case, the region $\mathcal{R}$ will be a disjoint union of the following form
\begin{align} 
\begin{split} \label{ch. 1, eqna}
&\{ [A_{1}, B_{1}] \times \cdots \times [A_{s}, B_{s}] \times (B_{s+1}, \infty) \times \cdots \times (B_{r}, \infty) \} \\
&\sqcup \cdots \\ 
&\sqcup \{ (B_{1}, \infty) \times \cdots \times (B_{r-s}, \infty) \times [A_{r-s+1}, B_{r-s+1}] \times \cdots \times [A_{r}, B_{r}] \} 
\end{split} 
\end{align}
\begin{align}
\begin{split} \label{ch. 1, eqnb}
&\sqcup \{ [A_{1}, B_{1}] \times \cdots \times [A_{s+1}, B_{s+1}] \times (B_{s+2}, \infty) \times \cdots \times (B_{r}, \infty) \} \\ 
&\sqcup \cdots \\ 
&\sqcup \{ (B_{1}, \infty) \times \cdots \times (B_{r-s-1}, \infty) \times [A_{r-s}, B_{r-s}] \times \cdots \times [A_{r}, B_{r}] 
\} 
\end{split} \\
\begin{split}
&\sqcup \cdots \nonumber 
\end{split} \\
\begin{split} \label{ch. 1, eqnc}
&\sqcup \{ [A_{1}, B_{1}] \times \cdots \times [A_{r}, B_{r}] \}, 
\end{split} 
\end{align}
for some integers $A_{1}, \ldots, A_{r}$ and $B_{1}, \ldots, B_{r}$. Here (\ref{ch. 1, eqna}) is a disjoint union of $\binom{r}{s}$ regions with $s$ upper bounds. Denote these regions of $\mathcal{R}$ by $\mathcal{R}_{\mu}^{s}$, where $\mu = 1, \ldots, \binom{r}{s}$. Here (\ref{ch. 1, eqnb}) is a disjoint union of $\binom{r}{s+1}$ regions with $s+1$ upper bounds. Denote these regions of $\mathcal{R}$ by $\mathcal{R}_{\mu}^{s+1}$, where $\mu = 1, \ldots, \binom{r}{s+1}$. Et cetera. Finally, (\ref{ch. 1, eqnc}) is a disjoint union of $\binom{r}{r} = 1$ regions with $r$ upper bounds (i.e.~only $[A_{1}, B_{1}] \times \cdots \times [A_{r}, B_{r}]$). Denote this region of $\mathcal{R}$ by $\mathcal{R}_{\mu}^{r}$. We will use this notation later on. Using the techniques of the previous subsection, one easily proves the following proposition.  
\begin{proposition} \label{ch. 1, sect. 2, prop. 6}
Let $U_{\sigma}$ be a nonsingular affine toric variety defined by a cone $\sigma$ in a lattice $N$ of rank $r$. Let $\tau_{1}, \ldots, \tau_{a} \prec \sigma$ be all faces of dimension $s$. Then the category of pure equivariant sheaves $\mathcal{E}$ on $U_{\sigma}$ with support $V(\tau_{1}) \cup \cdots \cup V(\tau_{a})$ is equivalent to the category of $\sigma$-families $\hat{E}^{\sigma}$ satisfying the following properties:
\begin{enumerate}
	\item [$\mathrm{(i)}$] There are integers $A_{1}, \ldots, A_{r}$, $B_{1}, \ldots, B_{r}$ such that $E^{\sigma}(\lambda_{1}, \ldots, \lambda_{r}) = 0$ unless $(\lambda_{1}, \ldots, \lambda_{r}) \in \mathcal{R}$, where the region $\mathcal{R}$ is as above.
	\item[$\mathrm{(ii)}$] Any region $\mathcal{R}_{\mu}^{i} = [A_{1}, B_{1}] \times \cdots \times [A_{i}, B_{i}] \times (B_{i+1}, \infty) \times \cdots \times (B_{r}, \infty)$ of $\mathcal{R}$ satisfies the following properties\footnote{Without loss of generality, we denote this region in such a way that the $i$ upper bounds occur in the first $i$ intervals. The general case is clear.}. Firstly, for any integers $A_{1} \leq \Lambda_{1} \leq B_{1}$, $\ldots$, $A_{i} \leq \Lambda_{i} \leq B_{i}$ there is a finite-dimensional $k$-vector space $E^{\sigma}(\Lambda_{1},    \ldots, \Lambda_{i}, \infty, \ldots, \infty)$, such that $E^{\sigma}(\Lambda_{1}, \ldots, \Lambda_{i}, \lambda_{i+1}, \ldots, \lambda_{r}) = E^{\sigma}(\Lambda_{1}, \ldots, \Lambda_{i}, \infty, \ldots, \infty)$ for some integers $\lambda_{i+1} > B_{i+1}, \ldots, \lambda_{r} > B_{r}$. Moreover, if $\mathcal{R}_{\mu}^{i}$ is one of the regions $\mathcal{R}_{\mu}^{s}$, not all $E^{\sigma}(\Lambda_{1},    \ldots, \Lambda_{i}, \infty, \ldots, \infty)$ are zero. Secondly, $\chi^{\sigma}_{i+1}(\vec{\lambda}), \ldots, \chi^{\sigma}_{r}(\vec{\lambda})$ are inclusions for all $\vec{\lambda} \in \mathcal{R}_{\mu}^{i}$. Finally, if $j_{1}, \ldots, j_{s+1} \in \{1, \ldots, i\}$ are distinct, then for any $\vec{\lambda} \in \mathcal{R}_{\mu}^{i}$ the following $k$-linear map is injective
\begin{align} 
	\begin{split} \label{ch. 1, eqnnew}
	&E^{\sigma}(\vec{\lambda}) \hookrightarrow E^{\sigma}(\lambda_{1}, \ldots, \lambda_{j_{1}-1},B_{j_{1}}+1,\lambda_{j_{1}+1}, \ldots, \lambda_{r}) \\     &\qquad \qquad \oplus \cdots \\
	&\qquad \qquad \oplus E^{\sigma}(\lambda_{1}, \ldots, \lambda_{j_{s+1}-1},B_{j_{s+1}}+1,\lambda_{j_{s+1}+1}, \ldots, \lambda_{r}), \\
	&\left( \chi_{j_{1}}^{\sigma}(\lambda_{1}, \ldots, \lambda_{j_{1}-1},B_{j_{1}},\lambda_{j_{1}+1}, \ldots, \lambda_{r}) \circ \cdots \right. \\
	&\quad \left. \circ \chi_{j_{1}}^{\sigma}(\lambda_{1}, \ldots, \lambda_{j_{1}-1},\lambda_{j_{1}},\lambda_{j_{1}+1}, \ldots, \lambda_{r}) \right) \\ 
	&\oplus \cdots 
	\end{split} \displaybreak \\
	\begin{split} \nonumber
	&\oplus \left( \chi_{j_{s+1}}^{\sigma}(\lambda_{1}, \ldots, \lambda_{j_{s+1}-1},B_{j_{s+1}},\lambda_{j_{s+1}+1}, \ldots, \lambda_{r}) \circ \cdots      \right. \\
	&\quad \left. \circ \chi_{j_{s+1}}^{\sigma}(\lambda_{1}, \ldots, \lambda_{j_{s+1}-1},\lambda_{j_{s+1}},\lambda_{j_{s+1}+1}, \ldots, \lambda_{r})        \right). 
          \end{split}
\end{align}
\end{enumerate}
\end{proposition}

\noindent Note that in this proposition, the only essentially new type of condition compared to Proposition \ref{ch. 1, sect. 2, prop. 4} is condition (\ref{ch. 1, eqnnew}). We obtain the following theorem.
\begin{theorem} \label{ch. 1, sect. 2, thm. 2}
Let $X$ be a nonsingular toric variety with fan $\Delta$ in a lattice $N$ of rank $r$. Let $\sigma_{1}, \ldots, \sigma_{l}$ be all cones of $\Delta$ of dimension $r$. Denote the rays of $\sigma_{i}$ by $\left(\rho^{(i)}_{1}, \ldots, \rho^{(i)}_{r} \right)$ for all $i = 1, \ldots, l$. Let $\tau_{1}, \ldots, \tau_{a}$ be all cones of $\Delta$ of dimension $s$. The category of pure equivariant sheaves on $X$ with support $V(\tau_{1}) \cup \cdots \cup V(\tau_{a})$ is equivalent to the category $\mathcal{C}^{\tau_{1}, \ldots, \tau_{a}}$, which can be described as follows. An object $\hat{E}^{\Delta}$ of $\mathcal{C}^{\tau_{1}, \ldots, \tau_{a}}$ consists of the following data:
\begin{enumerate}
	\item [$\mathrm{(i)}$] For each $i = 1,\ldots, l$ we have a $\sigma_{i}$-family $\hat{E}^{\sigma_{i}}$ as described in Proposition \ref{ch. 1, sect. 2, prop. 6}.
  \item [$\mathrm{(ii)}$] Let $i,j = 1, \ldots, l$. Let $\left\{\rho^{(i)}_{i_{1}}, \ldots, \rho^{(i)}_{i_{p}}\right\} \subset \left\{\rho^{(i)}_{1}, \ldots, \rho^{(i)}_{r}\right\}$ resp.~$\left\{\rho^{(j)}_{j_{1}}, \ldots, \rho^{(j)}_{j_{p}}\right\} \subset \left\{\rho^{(j)}_{1}, \ldots, \rho^{(j)}_{r}\right\}$ be the rays of $\sigma_{i} \cap \sigma_{j}$ in $\sigma_{i}$ respectively $\sigma_{j}$, labeled in such a way that $\rho^{(i)}_{i_{k}} = \rho^{(j)}_{j_{k}}$ for all $k = 1, \ldots, p$. Now let $\lambda^{(i)}_{1}, \ldots, \lambda^{(i)}_{r} \in \mathbb{Z} \cup \{\infty\}$, $\lambda^{(j)}_{1}, \ldots, \lambda^{(j)}_{r} \in \mathbb{Z} \cup \{\infty\}$ be such that $\lambda^{(i)}_{i_{k}} = \lambda^{(j)}_{j_{k}} \in \mathbb{Z}$ for all $k = 1, \ldots, p$ and $\lambda^{(i)}_{n} = \lambda^{(j)}_{n} = \infty$ otherwise. Then
  \begin{align} 
  \begin{split}
	E^{\sigma_{i}} \left( \sum_{k=1}^{r} \lambda^{(i)}_{k} m\left(\rho_{k}^{(i)}\right) \right) &= E^{\sigma_{j}} \left( \sum_{k=1}^{r} \lambda^{(j)}_{k} m\left(\rho_{k}^{(j)}\right) \right), \\
	\chi_{n}^{\sigma_{i}} \left( \sum_{k=1}^{r} \lambda^{(i)}_{k} m\left(\rho_{k}^{(i)}\right) \right) &= \chi_{n}^{\sigma_{j}} \left( \sum_{k=1}^{r} \lambda^{(j)}_{k} m\left(\rho_{k}^{(j)}\right) \right), \ \forall n = 1, \ldots, r. \nonumber
	\end{split}
	\end{align}
\end{enumerate}
The morphisms of $\mathcal{C}^{\tau_{1}, \ldots, \tau_{a}}$ are described as follows. If $\hat{E}^{\Delta}$, $\hat{F}^{\Delta}$ are two objects, then a morphism $\hat{\phi}^{\Delta} : \hat{E}^{\Delta} \longrightarrow \hat{F}^{\Delta}$ is a collection of morphisms of $\sigma$-families $\{\hat{\phi}^{\sigma_{i}} : \hat{E}^{\sigma_{i}} \longrightarrow \hat{F}^{\sigma_{i}}\}_{i = 1, \ldots, l}$ such that for all $i,j$ as in $\mathrm{(ii)}$ one has
\begin{equation}
\phi^{\sigma_{i}}\left( \sum_{k=1}^{r} \lambda^{(i)}_{k} m\left(\rho_{k}^{(i)}\right) \right) = \phi^{\sigma_{j}}\left( \sum_{k=1}^{r} \lambda^{(j)}_{k} m\left(\rho_{k}^{(j)}\right) \right). \nonumber
\end{equation}
\end{theorem}

Although we only described the ``maximally reducible'' case in Proposition \ref{ch. 1, sect. 2, prop. 6} and Theorem \ref{ch. 1, sect. 2, thm. 2}, the reader will have no difficulty writing down the case of arbitrary reducible support. We refrain from doing this since the notation will become too cumbersome, whereas the ideas are the same.

\section{Moduli Spaces of Equivariant Sheaves on Toric Varieties}

In this section, we discuss how the combinatorial description of pure equivariant sheaves on nonsingular toric varieties of Theorems \ref{ch. 1, sect. 2, thm. 1} and \ref{ch. 1, sect. 2, thm. 2} can be used to define a moduli problem and a coarse moduli space of such sheaves using GIT. We will start by defining the relevant moduli functors and studying families. Subsequently, we will perform GIT quotients and show we have obtained coarse moduli spaces. Again, for notational convenience, we will first treat the case of irreducible support and discuss the general case only briefly afterwards. The GIT construction gives rise to various notions of GIT stability depending on a choice of equivariant line bundle. In order to recover geometric results, we need an equivariant line bundle which precisely recovers Gieseker stability. We will construct such (ample) equivariant line bundles for torsion free equivariant sheaves in general. As a by-product, for reflexive equivariant sheaves, we can always construct particularly simple ample equivariant line bundles matching GIT stability and $\mu$-stability (see also subsection 4.4).

\subsection{Moduli Functors}

We start by defining some topological data.
\begin{definition} \label{ch. 1, sect. 3, def. 1}
Let $X$ be a nonsingular toric variety and use notation as in Theorem \ref{ch. 1, sect. 2, thm. 1}. Recall that $\sigma_{1}, \ldots, \sigma_{l}$ are the cones of maximal dimension having $\tau$ as a face. Let $\mathcal{E}$ be a pure equivariant sheaf on $X$ with support $V(\tau)$. The \emph{characteristic function} $\vec{\chi}_{\mathcal{E}}$ of $\mathcal{E}$ is defined to be the map
\begin{align} 
\begin{split}
&\vec{\chi}_{\mathcal{E}} : M \longrightarrow \mathbb{Z}^{l}, \\
&\vec{\chi}_{\mathcal{E}}(m) = (\chi_{\mathcal{E}}^{\sigma_{1}}(m), \ldots, \chi_{\mathcal{E}}^{\sigma_{l}}(m)) = (\mathrm{dim}_{k}(E_{m}^{\sigma_{1}}), \ldots, \mathrm{dim}_{k}(E_{m}^{\sigma_{l}})). \nonumber
\end{split}
\end{align}
We denote the set of all characteristic functions of pure equivariant sheaves on $X$ with support $V(\tau)$ by $\mathcal{X}^{\tau}$. \hfill $\oslash$
\end{definition}

Assume $X$ is a nonsingular projective toric variety. Let $\mathcal{O}_{X}(1)$ be an ample line bundle on $X$, so we can speak of Gieseker (semi)stable sheaves on $X$ \cite[Def.~1.2.4]{HL}. Let $\vec{\chi} \in \mathcal{X}^{\tau}$. We will be interested in moduli problems of Gieseker (semi)stable pure equivariant sheaves on $X$ with support $V(\tau)$ and characteristic function $\vec{\chi}$. This means we need to define moduli functors, i.e.~we need an appropriate notion of a family. Let $Sch/k$ be the category of $k$-schemes of finite type. Let $S$ be a $k$-scheme of finite type and, for any $x \in S$, define the natural morphism $\iota_{x} : \mathrm{Spec}(k(x)) \longrightarrow S$, where $k(x)$ is the residue field of $x$. We define an equivariant $S$-flat family to be an equivariant coherent sheaf $\mathcal{F}$ on $X \times S$ ($S$ with trivial torus action), which is flat w.r.t.~the projection $p_{S} : X \times S \longrightarrow S$. Such a family $\mathcal{F}$ is said to be Gieseker semistable with support $V(\tau)$ and characteristic function $\vec{\chi}$, if $\mathcal{F}_{x} = (1_{X} \times \iota_{x})^{*}\mathcal{F}$ is Gieseker semistable with support $V(\tau) \times \mathrm{Spec}(k(x))$ and characteristic function $\vec{\chi}$ for all $x \in S$. Two such families $\mathcal{F}_{1}, \mathcal{F}_{2}$ are said to be equivalent if there is a line bundle $L \in \mathrm{Pic}(S)$ and an equivariant isomorphism $\mathcal{F}_{1} \cong \mathcal{F}_{2} \otimes p_{S}^{*}L$, where $L$ is being considered as an equivariant sheaf on $S$ with trivial equivariant structure. Denote the set of Gieseker semistable equivariant $S$-flat families with support $V(\tau)$ and characteristic function $\vec{\chi}$ modulo equivalence by $\underline{\mathcal{M}}_{\vec{\chi}}^{\tau, ss}(S)$. We obtain a moduli functor
\begin{align} 
\begin{split}
\underline{\mathcal{M}}_{\vec{\chi}}^{\tau, ss} : (Sch/k)^{o} &\longrightarrow Sets, \\
S &\mapsto \underline{\mathcal{M}}_{\vec{\chi}}^{\tau, ss}(S), \\
(f : S' \longrightarrow S) &\mapsto \underline{\mathcal{M}}_{\vec{\chi}}^{\tau, ss}(f) = f^{*} : \underline{\mathcal{M}}_{\vec{\chi}}^{\tau, ss}(S) \longrightarrow \underline{\mathcal{M}}_{\vec{\chi}}^{\tau, ss}(S'). \nonumber
\end{split}
\end{align}   
Similarly, we obtain a moduli problem and a moduli functor $\underline{\mathcal{M}}_{\vec{\chi}}^{\tau, s}$ in the geometrically Gieseker stable case. Also note that we could have defined alternative moduli functors $\underline{\mathcal{M}}_{\vec{\chi}}^{\prime \tau, ss}$, $\underline{\mathcal{M}}_{\vec{\chi}}^{\prime \tau, s}$ by using just equivariant isomorphism as the equivalence relation instead of the one above. We start with the following proposition.
\begin{proposition} \label{ch. 1, sect. 3, prop. 1}
Let $X$ be a nonsingular toric variety. Let $S$ be a connected $k$-scheme of finite type and let $\mathcal{F}$ be an equivariant $S$-flat family. Then the characteristic functions of the fibres $\vec{\chi}_{\mathcal{F}_{x}}$ are constant on $S$. 
\end{proposition}
\begin{proof}
Let $\sigma$ be a cone of the fan $\Delta$. Let $V = \mathrm{Spec}(A) \subset S$ be an affine open subset\footnote{Note that from now on, for $R$ a commutative ring, we often sloppily write $R$ instead of $\mathrm{Spec}(R)$, when no confusion is likely to arise.}. It is enough to prove that for all $m \in M$ 
\begin{equation}
\chi_{\mathcal{F}_{x}}^{\sigma}(m) = \mathrm{dim}_{k(x)} \ \Gamma(U_{\sigma} \times k(x), \mathcal{F}_{x}|_{U_{\sigma} \times k(x)})_{m}, \nonumber
\end{equation}
is constant for all $x \in V$. Note that the equivariant coherent sheaf $\mathcal{F}|_{U_{\sigma} \times V}$ corresponds to a finitely generated $M$-graded $k[S_{\sigma}] \otimes_{k} A$-module
\begin{equation}
\Gamma(U_{\sigma} \times V,\mathcal{F}) = \bigoplus_{m \in M} F_{m}^{\sigma}, \nonumber
\end{equation}
where all $F_{m}^{\sigma}$ are in fact finitely generated $A$-modules, so they correspond to coherent sheaves $\mathcal{F}_{m}^{\sigma}$ on $V$. Since $\mathcal{F}$ is $S$-flat, each $\mathcal{F}_{m}^{\sigma}$ is a locally free sheaf of some finite rank $r(m)$ \cite[Prop.~III.9.2]{Har1}. Fix $x \in V$ and consider the natural morphism $A \longrightarrow k(x)$, then
\begin{equation}
\Gamma(U_{\sigma} \times k(x), \mathcal{F}_{x}|_{U_{\sigma} \times k(x)}) \cong \bigoplus_{m \in M} F_{m}^{\sigma} \otimes_{A} k(x) \cong \bigoplus_{m \in M} k(x)^{\oplus r(m)}. \nonumber
\end{equation} 
Consequently, $\chi_{\mathcal{F}_{x}}^{\sigma}(m) = r(m)$ for all $m \in M$.
\end{proof}

\subsection{Families}

The question now arises to what extent the various moduli functors defined in the previous subsection are corepresentable. In order to answer this question, we will give a combinatorial description of a family. Payne has studied a similar problem in \cite{Pay} for equivariant vector bundles on toric varieties. 

We start with some straight-forward generalisations of the theory in section 2.
\begin{definition} \label{ch. 1, sect. 3, def. 2}
Let $U_{\sigma}$ be an affine toric variety defined by a cone $\sigma$ in a lattice $N$. Let $S$ be a $k$-scheme of finite type. A \emph{$\sigma$-family over $S$} consists of the following data: a family of quasi-coherent sheaves $\{\mathcal{F}_{m}^{\sigma}\}_{m \in M}$ on $S$ and morphisms $\chi_{m,m^{\prime}}^{\sigma} : \mathcal{F}_{m}^{\sigma} \longrightarrow \mathcal{F}_{m^{\prime}}^{\sigma}$ for all $m \leq_{\sigma} m'$, such that $\chi_{m,m}^{\sigma} = 1$ and $\chi_{m,m^{\prime\prime}}^{\sigma} = \chi_{m^{\prime},m^{\prime\prime}}^{\sigma} \circ \chi_{m,m^{\prime}}^{\sigma}$ for all $m \leq_{\sigma} m^{\prime} \leq_{\sigma} m^{\prime\prime}$. A \emph{morphism} $\hat{\phi}^{\sigma} : \hat{\mathcal{F}}^{\sigma} \longrightarrow \hat{\mathcal{G}}^{\sigma}$ \emph{of $\sigma$-families over $S$} is a family of morphisms $\{\phi_{m} : \mathcal{F}_{m}^{\sigma} \longrightarrow \mathcal{G}_{m}^{\sigma} \}_{m \in M}$, such that $\phi_{m^{\prime}}^{\sigma} \circ (\chi_{\mathcal{F}})_{m,m^{\prime}}^{\sigma}= (\chi_{\mathcal{G}})_{m,m^{\prime}}^{\sigma} \circ \phi_{m}^{\sigma}$ for all $m \leq_{\sigma} m'$. \hfill $\oslash$
\end{definition}
\begin{proposition} \label{ch. 1, sect. 3, prop. 2}
Let $U_{\sigma}$ be a nonsingular affine toric variety defined by a cone $\sigma$ and let $S$ be a $k$-scheme of finite type. The category of equivariant quasi-coherent sheaves on $U_{\sigma} \times S$ is equivalent to the category of $\sigma$-families over $S$. 
\end{proposition}
\begin{proof}
Let $S = \mathrm{Spec}(A)$ be affine and $(\mathcal{F}, \Phi)$ an equivariant coherent sheaf. We have a regular action of $T$ on $F^{\sigma} = \Gamma(U_{\sigma} \times S, \mathcal{F})$, inducing a decomposition into weight spaces $F^{\sigma} = \bigoplus_{m \in M} F_{m}^{\sigma}$ (Complete Reducibility Theorem, \cite[Thm.~2.30]{Per1}). This time however, the $F_{m}^{\sigma}$ are $A$-modules instead of $k$-vector spaces. This gives the desired equivalence of categories. It is easy to see that the same holds for arbitrary $S$ by gluing.  
\end{proof} 

\noindent In the context of the previous proposition, a $\sigma$-family $\hat{\mathcal{F}}^{\sigma}$ over $S$ is called finite if all $\mathcal{F}^{\sigma}_{m}$ are coherent sheaves on $S$, there are integers $A_{1}, \ldots, A_{r}$ such that $\mathcal{F}^{\sigma}(\lambda_{1}, \ldots, \lambda_{r}) = 0$ unless $A_{1} \leq \lambda_{1}$, $\ldots$, $A_{r} \leq \lambda_{r}$ and there are only finitely many $m \in M$ such that the morphism
\begin{equation}
\bigoplus_{m' <_{\sigma} m} \mathcal{F}_{m'}^{\sigma} \longrightarrow \mathcal{F}_{m}^{\sigma}, \nonumber
\end{equation}
is not surjective.
\begin{proposition} \label{ch. 1, sect. 3, prop. 3}
Let $U_{\sigma}$ be a nonsingular affine toric variety defined by a cone $\sigma$ and let $S$ be a $k$-scheme of finite type. The category of equivariant $S$-flat families is equivalent to the category of finite $\sigma$-families $\hat{\mathcal{F}}^{\sigma}$ over $S$ with all $\mathcal{F}_{m}^{\sigma}$ locally free sheaves on $S$ of finite rank. 
\end{proposition}
\begin{proof}
It is not difficult to derive that the equivalence of Proposition \ref{ch. 1, sect. 3, prop. 2} restricts to an equivalence between the category of equivariant coherent sheaves on $X \times S$ and the category of finite $\sigma$-families over $S$. In the proof of Proposition \ref{ch. 1, sect. 3, prop. 1}, we saw that $S$-flatness gives rise to locally free of finite rank.
\end{proof}

\noindent In the context of the previous proposition, let $\hat{\mathcal{F}}^{\sigma}$ be a $\sigma$-family over $S$ corresponding to an equivariant $S$-flat family $\mathcal{F}$. For each connected component $C \subset S$, the characteristic function $\chi_{\mathcal{F}_{x}} : M \longrightarrow \mathbb{Z}$, for any $x \in C$, gives us the ranks of the $\mathcal{F}_{m}^{\sigma}$ on $C$.  
\begin{proposition} \label{ch. 1, sect. 3, prop. 4}
Let $U_{\sigma}$ be a nonsingular affine toric variety defined by a cone $\sigma$ in a lattice $N$ of rank $r$. Let $\tau \prec \sigma$ and let $(\rho_{1}, \ldots, \rho_{s}) \subset (\rho_{1}, \ldots, \rho_{r})$ be the rays of $\tau$ respectively $\sigma$. Let $S$ be a $k$-scheme of finite type and $\chi \in \mathcal{X}^{\tau}$. Let $\mathcal{F}$ be an equivariant $S$-flat family such that $\chi_{\mathcal{F}_{x}} = \chi$ for all $x \in S$ and let $\hat{\mathcal{F}}^{\sigma}$ be the corresponding $\sigma$-family over $S$. Then the fibres $\mathcal{F}_{x}$ are pure equivariant with support $V(\tau) \times k(x)$ if and only if the following properties are satisfied:
\begin{enumerate}
	\item [$\mathrm{(i)}$] There are integers $A_{1} \leq B_{1}, \ldots, A_{s} \leq B_{s}, A_{s+1}, \ldots, A_{r}$ such that $\mathcal{F}^{\sigma}(\lambda_{1}, \ldots, \lambda_{r}) = 0$ unless $A_{1} \leq \lambda_{1} \leq B_{1}$, $\ldots$, $A_{s} \leq \lambda_{s} \leq B_{s}$, $A_{s+1} \leq \lambda_{s+1}$, $\ldots$, $A_{r} \leq \lambda_{r}$ and for $\lambda_{i} \neq \lambda_{1}, \ldots, \lambda_{s}$ there is no such upper bound.
	\item [$\mathrm{(ii)}$] For any $x \in S$, $\mathcal{F}_{x}$ has a corresponding $\sigma$-family $\hat{F_{x}}^{\sigma}$ as in Proposition \ref{ch. 1, sect. 2, prop. 4} (over ground field $k(x)$) with bounding integers $A_{1} \leq B_{1}, \ldots, A_{s} \leq B_{s}, A_{s+1}, \ldots, A_{r}$.
\end{enumerate}
\end{proposition}
\begin{proof}
Note first of all that the entire theory of subsection 2.1 and 2.2 works over any ground field of characteristic 0, so we can replace $k$ by $k(x)$ in all the results ($x \in S$). Note that if $F_{m}^{\sigma}$ is the $A$-module corresponding to $\mathcal{F}_{m}^{\sigma}|_{V}$ for $V = \mathrm{Spec}(A) \subset S$ an affine open subset, then $(F_{x})^{\sigma}_{m} \cong F_{m}^{\sigma} \otimes_{A} k(x)$ for all $x \in V$. In particular, $\chi(m) = \mathrm{rk}(\mathcal{F}_{m}^{\sigma}) = \mathrm{dim}_{k(x)}( (F_{x})_{m}^{\sigma})$ for all $m \in M, x \in S$. The result easily follows from Proposition \ref{ch. 1, sect. 2, prop. 4}.
\end{proof}

Before we proceed, we need a technical result.
\begin{proposition} \label{ch. 1, sect. 3, prop. T}
Let $S$ be a $k$-scheme of finite type and let $\phi : \mathcal{E} \longrightarrow \mathcal{F}$ be a morphism of locally free sheaves of finite rank on $S$. Let $\iota_{x} : \mathrm{Spec}(k(x)) \longrightarrow S$ be the natural morphism for all $x \in S$. Then $\phi$ is injective and $\mathrm{coker}(\phi)$ is $S$-flat if and only if $\iota^{*}_{x}\phi$ is injective for all $x \in S$. 
\end{proposition}
\begin{proof}
By taking an open affine cover over which both locally free sheaves trivialise, it is easy to see we are reduced to proving the following: 

\noindent \emph{Claim.} Let $(R, \mathfrak{m})$ be a local ring. Let $k = R / \mathfrak{m}$ be the residue field and let $\phi : R^{\oplus a} \longrightarrow R^{\oplus b}$ be an $R$-module homomorphism. Then $\phi$ is injective and $\mathrm{coker}(\phi)$ is free of finite rank if and only if the induced map $\overline{\phi} : k^{\oplus a} \longrightarrow k^{\oplus b}$ is injective. 

\noindent \emph{Proof of Claim: $\Leftarrow$.} Let $M = R^{\oplus b} / \mathrm{im}(\phi)$, then we have an exact sequence $R^{\oplus a} \stackrel{\phi}{\longrightarrow} R^{\oplus b} \longrightarrow M \longrightarrow 0$. Applying $- \otimes_{R} k$ and using the assumption, we obtain a short exact sequence $0 \longrightarrow k^{\oplus a} \stackrel{\overline{\phi}}{\longrightarrow} k^{\oplus b} \longrightarrow M/\mathfrak{m}M \longrightarrow 0$. Here $M/\mathfrak{m}M$ is a $c = b-a$ dimensional $k$-vector space. Take $c$ basis elements of $M / \mathfrak{m}M$, then their representatives in $M$ generate $M$ as an $R$-module (Nakayama's Lemma). Take preimages $x_{1}, \ldots, x_{c}$ in $R^{\oplus b}$. Denote the standard generators of $R^{\oplus c}$ by $e_{1}, \ldots, e_{c}$ and define $\psi : R^{\oplus c} \longrightarrow R^{\oplus b}, \ e_{i} \mapsto x_{i}$. We get a diagram
\begin{displaymath}
\xymatrix
{
0 \ar[r] & R^{\oplus a} \ar[r]^>>>>>{\iota} \ar[d]_{1} & R^{\oplus a} \oplus R^{\oplus c} \ar[r]^>>>>>{\pi} \ar[d]_{\phi + \psi} & R^{\oplus c} \ar[r] \ar[d] & 0 \\
0 \ar[r] & R^{\oplus a} \ar[r]^{\phi} & R^{\oplus b} \ar[r] & M \ar[r] & 0.  
}
\end{displaymath}
Here the top sequence is split exact and in the lower sequence we still have to verify $\phi$ is injective. Furthermore, all squares commute and $\phi + \psi$ is an isomorphism, because $\phi + \psi$ is surjective \cite[Cor.~4.4]{Eis}. Therefore all vertical maps are isomorphisms. The statement now follows. 

\noindent \emph{Proof of Claim: $\Rightarrow$.} We have a short exact sequence $0 \longrightarrow R^{\oplus a} \stackrel{\phi}{\longrightarrow} R^{\oplus b} \longrightarrow M \longrightarrow 0$, where $M = R^{\oplus b} / \mathrm{im}(\phi) \cong R^{\oplus c}$ for some $c$. The long exact sequence of $\mathrm{Tor}_{i}(k,-)$ reads
\begin{equation}
\cdots \longrightarrow \mathrm{Tor}_{1}(k,M) \longrightarrow k^{\oplus a} \longrightarrow k^{\oplus b} \longrightarrow M / \mathfrak{m}M \longrightarrow 0. \nonumber
\end{equation} 
But clearly $\mathrm{Tor}_{1}(k,M) = 0$, since $M \cong R^{\oplus c}$ for some $c$. 
\end{proof}

\noindent We can now derive a combinatorial description of the type of families we are interested in for the affine case. 
\begin{proposition} \label{ch. 1, sect. 3, prop. 5}
Let $U_{\sigma}$ be a nonsingular affine toric variety defined by a cone $\sigma$ in a lattice $N$ of rank $r$. Let $\tau \prec \sigma$ and let $(\rho_{1}, \ldots, \rho_{s}) \subset (\rho_{1}, \ldots, \rho_{r})$ be the rays of $\tau$ respectively $\sigma$. Let $S$ be a $k$-scheme of finite type and $\chi \in \mathcal{X}^{\tau}$. The category of equivariant $S$-flat families $\mathcal{F}$ with fibres $\mathcal{F}_{x}$ pure equivariant with support $V(\tau) \times k(x)$ and characteristic function $\chi$ is equivalent to the category of $\sigma$-families $\hat{\mathcal{F}}^{\sigma}$ over $S$ having the following properties:
\begin{enumerate}
	\item [$\mathrm{(i)}$] There are integers $A_{1} \leq B_{1}, \ldots, A_{s} \leq B_{s}, A_{s+1}, \ldots, A_{r}$ such that $\mathcal{F}^{\sigma}(\lambda_{1}, \ldots, \lambda_{r}) = 0$ unless $A_{1} \leq \lambda_{1} \leq B_{1}$, $\ldots$, $A_{s} \leq \lambda_{s} \leq B_{s}$, $A_{s+1} \leq \lambda_{s+1}$, $\ldots$, $A_{r} \leq \lambda_{r}$.
	\item [$\mathrm{(ii)}$] For all integers $A_{1} \leq \Lambda_{1} \leq B_{1}$, $\ldots$, $A_{s} \leq \Lambda_{s} \leq B_{s}$, there is a locally free sheaf $\mathcal{F}^{\sigma}(\Lambda_{1}, \ldots, \Lambda_{s}, \infty, \ldots, \infty)$ on $S$ of finite rank (not all zero) having the following properties. All $\mathcal{F}^{\sigma}(\Lambda_{1}, \ldots, \Lambda_{s}, \lambda_{s+1}, \ldots, \lambda_{r})$ are quasi-coherent subsheaves of $\mathcal{F}^{\sigma}(\Lambda_{1}, \ldots, \Lambda_{s}, \infty, \ldots, \infty)$, the maps $x_{s+1}, \ldots, x_{r}$ are inclusions with $S$-flat cokernels and there are integers $\lambda_{s+1}, \ldots, \lambda_{r}$ such that $\mathcal{F}^{\sigma}(\Lambda_{1}, \ldots, \Lambda_{s}, \lambda_{s+1}, \ldots, \lambda_{r})=\mathcal{F}^{\sigma}(\Lambda_{1}, \ldots, \Lambda_{s}, \infty, \ldots, \infty)$.
	\item [$\mathrm{(iii)}$] For any $m \in M$, we have $\chi(m) = \mathrm{rk}(\mathcal{F}^{\sigma}_{m})$.
\end{enumerate}
\end{proposition}
\begin{proof}
Note that if we have a $\sigma$-family $\hat{\mathcal{F}}^{\sigma}$ as in (i), (ii), (iii), then all $\mathcal{F}_{m}^{\sigma}$ are locally free of finite rank \cite[Prop.~III.9.1A(e)]{Har1}. The statement immediately follows from Propositions \ref{ch. 1, sect. 3, prop. 3}, \ref{ch. 1, sect. 3, prop. 4} and \ref{ch. 1, sect. 3, prop. T}.
\end{proof}

The general combinatorial description of the kind of families we are interested in easily follows.
\begin{theorem} \label{ch. 1, sect. 3, thm. 1}
Let $X$ be a nonsingular toric variety with fan $\Delta$ in a lattice $N$ of rank $r$. Let $\tau \in \Delta$, then $V(\tau)$ is covered by those $U_{\sigma}$ where $\sigma \in \Delta$ has dimension $r$ and $\tau \prec \sigma$. Denote these cones by $\sigma_{1}, \ldots, \sigma_{l}$. For each $i = 1, \ldots, l$, let $\left(\rho^{(i)}_{1}, \ldots, \rho^{(i)}_{s} \right) \subset \left(\rho^{(i)}_{1}, \ldots, \rho^{(i)}_{r}\right)$ be the rays of $\tau$ respectively $\sigma_{i}$. Let $S$ be a $k$-scheme of finite type and $\vec{\chi} \in \mathcal{X}^{\tau}$. The category of equivariant $S$-flat families $\mathcal{F}$ with fibres $\mathcal{F}_{x}$ pure equivariant sheaves with support $V(\tau) \times k(x)$ and characteristic function $\vec{\chi}$ is equivalent to the category $\mathcal{C}^{\tau}_{\vec{\chi}}(S)$, which can be described as follows. An object $\hat{\mathcal{F}}^{\Delta}$ of $\mathcal{C}^{\tau}_{\vec{\chi}}(S)$ consists of the following data: 
\begin{enumerate}
	\item [$\mathrm{(i)}$] For each $i = 1,\ldots, l$ we have a $\sigma_{i}$-family $\hat{\mathcal{F}}^{\sigma_{i}}$ over $S$ as described in Proposition \ref{ch. 1, sect. 3, prop. 5}.
	\item [$\mathrm{(ii)}$] Let $i,j = 1, \ldots, l$. Let $\left\{\rho^{(i)}_{i_{1}}, \ldots, \rho^{(i)}_{i_{p}}\right\} \subset \left\{\rho^{(i)}_{1}, \ldots, \rho^{(i)}_{r}\right\}$ resp.~$\left\{\rho^{(j)}_{j_{1}}, \ldots, \rho^{(j)}_{j_{p}}\right\} \subset \left\{\rho^{(j)}_{1}, \ldots, \rho^{(j)}_{r}\right\}$ be the rays of $\sigma_{i} \cap \sigma_{j}$ in $\sigma_{i}$ respectively $\sigma_{j}$, labeled in such a way that $\rho^{(i)}_{i_{k}} = \rho^{(j)}_{j_{k}}$ for all $k = 1, \ldots, p$. Now let $\lambda^{(i)}_{1}, \ldots, \lambda^{(i)}_{r} \in \mathbb{Z} \cup \{\infty\}$, $\lambda^{(j)}_{1}, \ldots, \lambda^{(j)}_{r} \in \mathbb{Z} \cup \{\infty\}$ be such that $\lambda^{(i)}_{i_{k}} = \lambda^{(j)}_{j_{k}} \in \mathbb{Z}$ for all $k = 1, \ldots, p$ and $\lambda^{(i)}_{n} = \lambda^{(j)}_{n} = \infty$ otherwise. Then
  \begin{align} 
 \begin{split}
	&\mathcal{F}^{\sigma_{i}} \left( \sum_{k=1}^{r} \lambda^{(i)}_{k} m\left(\rho_{k}^{(i)}\right) \right) = \mathcal{F}^{\sigma_{j}} \left( \sum_{k=1}^{r} \lambda^{(j)}_{k} m\left(\rho_{k}^{(j)}\right) \right), \\
	&\chi_{n}^{\sigma_{i}} \left( \sum_{k=1}^{r} \lambda^{(i)}_{k} m\left(\rho_{k}^{(i)}\right) \right) = \chi_{n}^{\sigma_{j}} \left( \sum_{k=1}^{r} \lambda^{(j)}_{k} m\left(\rho_{k}^{(j)}\right) \right), \ \forall n = 1, \ldots, r. \nonumber
	\end{split}
	\end{align}
\end{enumerate}
The morphisms of $\mathcal{C}^{\tau}_{\vec{\chi}}(S)$ are described as follows. If $\hat{\mathcal{F}}^{\Delta}$, $\hat{\mathcal{G}}^{\Delta}$ are two objects, then a morphism $\hat{\phi}^{\Delta} : \hat{\mathcal{F}}^{\Delta} \longrightarrow \hat{\mathcal{G}}^{\Delta}$ is a collection of morphisms of $\sigma$-families $\{\hat{\phi}^{\sigma_{i}} : \hat{\mathcal{F}}^{\sigma_{i}} \longrightarrow \hat{\mathcal{G}}^{\sigma_{i}}\}_{i = 1, \ldots, l}$ over $S$ such that for all $i,j$ as in $\mathrm{(ii)}$ one has
\begin{equation}
\phi^{\sigma_{i}}\left( \sum_{k=1}^{r} \lambda^{(i)}_{k} m\left(\rho_{k}^{(i)}\right) \right) = \phi^{\sigma_{j}}\left( \sum_{k=1}^{r} \lambda^{(j)}_{k} m\left(\rho_{k}^{(j)}\right) \right). \nonumber
\end{equation}
\end{theorem}
\begin{proof}
This theorem follows from combining Proposition \ref{ch. 1, sect. 3, prop. 5} and an obvious analogue of Proposition \ref{ch. 1, sect. 2, prop. 5}. 
\end{proof}

\noindent Note that in the context of the above theorem, we can define the following moduli functor
\begin{align} 
\begin{split}
\mathfrak{C}^{\tau}_{\vec{\chi}} : (Sch/k)^{o} &\longrightarrow Sets, \\
S &\mapsto \mathfrak{C}^{\tau}_{\vec{\chi}}(S) = \mathcal{C}^{\tau}_{\vec{\chi}}(S), \\
(f : S' \longrightarrow S) &\mapsto \mathfrak{C}^{\tau}_{\vec{\chi}}(f)=f^{*} : \mathcal{C}^{\tau}_{\vec{\chi}}(S) \longrightarrow \mathcal{C}^{\tau}_{\vec{\chi}}(S'). \nonumber
\end{split}
\end{align}
For later purposes, we need to define another moduli functor. If $S$ is a $k$-scheme of finite type, then we define $\mathcal{C}^{\tau, fr}_{\vec{\chi}}(S)$ to be the full\footnote{Note that we do insist on the full subcategory, meaning we keep the notion of morphism from the category $\mathcal{C}^{\tau}_{\vec{\chi}}(S)$. This allows us to mod out by isomorphisms and relate to GIT later (see next subsection).} subcategory of $\mathcal{C}^{\tau}_{\vec{\chi}}(S)$ consisting of those objects $\hat{\mathcal{F}}^{\Delta}$ with each limiting sheaf $\mathcal{F}^{\sigma_{i}}(\Lambda_{1}, \ldots, \Lambda_{s}, \infty, \ldots, \infty)$ equal to a sheaf of the form $\mathcal{O}_{S}^{\oplus n(\Lambda_{1}, \ldots, \Lambda_{s})}$ for some $n(\Lambda_{1}, \ldots, \Lambda_{s}) \in \mathbb{Z}_{\geq 0}$. We would like to think of the objects of this full subcategory as framed objects. This gives rise to a moduli functor $\mathfrak{C}^{\tau, fr}_{\vec{\chi}} : (Sch/k)^{o} \longrightarrow Sets$.

\subsection{GIT Quotients}

Our goal is to find $k$-schemes of finite type corepresenting the moduli functors $\underline{\mathcal{M}}_{\vec{\chi}}^{\tau, ss}$ and  $\underline{\mathcal{M}}_{\vec{\chi}}^{\tau, s}$. We will achieve this using GIT. 

Let $X$ be a nonsingular toric variety, use notation as in Theorem \ref{ch. 1, sect. 2, thm. 1} and fix $\vec{\chi} \in \mathcal{X}^{\tau}$. The integers $A_{1}^{(i)} \leq B_{1}^{(i)}, \ldots, A_{s}^{(i)} \leq B_{s}^{(i)}, A_{s+1}^{(i)}, \ldots, A_{r}^{(i)}$ for $i = 1, \ldots, l$ in Theorem \ref{ch. 1, sect. 2, thm. 1} of any pure equivariant sheaf $\mathcal{E}$ on $X$ with support $V(\tau)$ and characteristic function $\vec{\chi}$ are uniquely determined by $\vec{\chi}$ (if we choose them in a maximal respectively minimal way). Note that we have $A_{k}^{(1)} = \cdots = A_{k}^{(l)} =: A_{k}$ and $B_{k}^{(1)} = \cdots = B_{k}^{(l)} =: B_{k}$ for all $k = 1, \ldots, s$, because of the gluing conditions in Theorem \ref{ch. 1, sect. 2, thm. 1}. For all $A_{1} \leq \Lambda_{1} \leq B_{1}$, $\ldots$, $A_{s} \leq \Lambda_{s} \leq B_{s}$ we define 
\begin{equation}
n(\Lambda_{1}, \ldots, \Lambda_{s}):= \lim_{\lambda_{s+1}, \ldots, \lambda_{r} \rightarrow \infty} \chi^{\sigma_{i}}\left(\sum_{k=1}^{s} \Lambda_{k} m\left( \rho_{k}^{(i)} \right) + \sum_{k=s+1}^{r} \lambda_{k} m\left( \rho_{k}^{(i)} \right) \right), \nonumber
\end{equation}
which is independent of $i = 1, \ldots, l$, because of the gluing conditions in Theorem \ref{ch. 1, sect. 2, thm. 1}. For all other values of $\Lambda_{1}, \ldots, \Lambda_{s} \in \mathbb{Z}$, we define $n(\Lambda_{1}, \ldots, \Lambda_{s}) = 0$. In general, denote by $\mathrm{Gr}(m,n)$ the Grassmannian of $m$-dimensional subspaces of $k^{n}$ and by $\mathrm{Mat}(m,n)$ the affine space of $m \times n$ matrices with coefficients in $k$. Define the following ambient nonsingular quasi-projective variety 
\begin{align} \label{ch. 1, eqn3}
\begin{split}
\mathcal{A} = &\prod_{\scriptsize{\begin{array}{c} A_{1} \leq \Lambda_{1} \leq B_{1} \\ \cdots \\ A_{s} \leq \Lambda_{s} \leq B_{s} \end{array}}} \prod_{i=1}^{l} \prod_{m \in M} \mathrm{Gr} \left( \chi^{\sigma_{i}}(m), n(\Lambda_{1}, \ldots, \Lambda_{s}) \right) \\
&\times \prod_{\scriptsize{\begin{array}{c} A_{1} \leq \Lambda_{1} \leq B_{1} \\ \cdots \\ A_{s} \leq \Lambda_{s} \leq B_{s} \end{array}}} \mathrm{Mat}(n(\Lambda_{1}, \ldots, \Lambda_{s}),n(\Lambda_{1}+1,\Lambda_{2} \ldots, \Lambda_{s})) \times \cdots  
\end{split} \displaybreak \\
\begin{split} \nonumber
&\times \prod_{\scriptsize{\begin{array}{c} A_{1} \leq \Lambda_{1} \leq B_{1} \\ \cdots \\ A_{s} \leq \Lambda_{s} \leq B_{s} \end{array}}} \mathrm{Mat}(n(\Lambda_{1}, \ldots, \Lambda_{s}),n(\Lambda_{1}, \ldots, \Lambda_{s-1}, \Lambda_{s}+1)).
\end{split}
\end{align}
There is a natural closed subscheme $\mathcal{N}^{\tau}_{\vec{\chi}}$ of $\mathcal{A}$ with closed points precisely the framed pure $\Delta$-families with support $V(\tau)$ and characteristic function $\vec{\chi}$. This closed subscheme is cut out by requiring the various subspaces of any $k^{n(\Lambda_{1}, \ldots, \Lambda_{s})}$ to form a multi-filtration and by requiring the matrices between the limiting vector spaces $k^{n(\Lambda_{1}, \ldots, \Lambda_{s})}$ to commute and be compatible with the multi-filtrations (see Theorem \ref{ch. 1, sect. 2, thm. 1}). Using the standard atlases of Grassmannians, it is not difficult to see that these conditions cut out a closed subscheme. Define the reductive algebraic group
\begin{equation}
G = \Big\{ M \in \prod_{\scriptsize{\begin{array}{c} A_{1} \leq \Lambda_{1} \leq B_{1} \\ \cdots \\ A_{s} \leq \Lambda_{s} \leq B_{s} \end{array}}} \mathrm{GL}(n(\Lambda_{1}, \ldots, \Lambda_{s}),k) \ | \ \det(M) = 1 \Big\}. \label{ch. 1, eqnG}
\end{equation}
There is a natural regular action of $G$ on $\mathcal{A}$ leaving $\mathcal{N}^{\tau}_{\vec{\chi}}$ invariant. Two closed points of $\mathcal{N}^{\tau}_{\vec{\chi}}$ correspond to isomorphic elements if and only if they are in the same $G$-orbit. For any choice of $G$-equivariant line bundle $\mathcal{L} \in \mathrm{Pic}^{G}(\mathcal{N}_{\vec{\chi}})$, we get the notion of GIT (semi)stable elements of $\mathcal{N}_{\vec{\chi}}$ with respect to $\mathcal{L}$ \cite[Sect.~1.4]{MFK}. We denote the $G$-invariant open subset of GIT semistable respectively stable elements by $\mathcal{N}_{\vec{\chi}}^{\tau, ss}$ respectively $\mathcal{N}_{\vec{\chi}}^{\tau, s}$. We call a pure equivariant sheaf $\mathcal{E}$ on $X$ with support $V(\tau)$ and characteristic function $\vec{\chi}$ GIT semistable, respectively GIT stable, if its corresponding framed pure $\Delta$-family $\hat{E}^{\Delta}$ is GIT semistable, respectively GIT stable. Using \cite[Thm.~1.10]{MFK}, we obtain that there exists a categorical quotient $\pi : \mathcal{N}_{\vec{\chi}}^{\tau, ss} \longrightarrow \mathcal{N}_{\vec{\chi}}^{\tau, ss} /\!\!/ G$, where $\mathcal{N}_{\vec{\chi}}^{\tau, ss} /\!\!/ G$ is a quasi-projective scheme of finite type over $k$ which we denote by $\mathcal{M}_{\vec{\chi}}^{\tau, ss}$. Moreover, there exists an open subset $U \subset \mathcal{M}_{\vec{\chi}}^{\tau, ss}$, such that $\pi^{-1}(U) = \mathcal{N}_{\vec{\chi}}^{\tau, s}$ and $\varpi = \pi|_{\mathcal{N}_{\vec{\chi}}^{\tau, s}} : \mathcal{N}_{\vec{\chi}}^{\tau, s} \longrightarrow U = \mathcal{N}_{\vec{\chi}}^{\tau, s}/G$ is a geometric quotient. We define $\mathcal{M}_{\vec{\chi}}^{\tau, s} = \mathcal{N}_{\vec{\chi}}^{\tau, s} / G$. The fibres of closed points of $\varpi$ are precisely the $G$-orbits of closed points of $\mathcal{N}_{\vec{\chi}}^{\tau, s}$, or equivalently, the equivariant isomorphism classes of GIT stable pure equivariant sheaves on $X$ with support $V(\tau)$ and characteristic function $\vec{\chi}$. It seems natural to think of $\mathcal{M}_{\vec{\chi}}^{\tau, ss}$ and $\mathcal{M}_{\vec{\chi}}^{\tau, s}$ as moduli spaces. Before making this more precise, we would like to make the problem more geometric. Assume in addition $X$ is projective and fix an ample line bundle $\mathcal{O}_{X}(1)$ on $X$. The natural notion of stability for coherent sheaves on $X$ is Gieseker stability, which depends on the choice of $\mathcal{O}_{X}(1)$ \cite[Sect.~1.2]{HL}. 
\begin{assumption} \label{ch. 1, sect. 3, ass. 1}
Let $X$ be a nonsingular projective toric variety defined by a fan $\Delta$. Let $\mathcal{O}_{X}(1)$ be an ample line bundle on $X$ and $\tau \in \Delta$. Then for any $\vec{\chi} \in \mathcal{X}^{\tau}$, let $\mathcal{L}^{\tau}_{\vec{\chi}} \in \mathrm{Pic}^{G}(\mathcal{N}^{\tau}_{\vec{\chi}})$ be an equivariant line bundle such that any pure equivariant sheaf $\mathcal{E}$ on $X$ with support $V(\tau)$ and characteristic function $\vec{\chi}$ is GIT semistable respectively GIT stable w.r.t.~$\mathcal{L}^{\tau}_{\vec{\chi}}$ if and only if $\mathcal{E}$ is Gieseker semistable respectively Gieseker stable.
\end{assumption}

\noindent We will refer to equivariant line bundles as in this assumption as equivariant line bundles matching Gieseker and GIT stability. So far, the author cannot prove the existence of such equivariant line bundles in full generality. However, in subsection 3.5, we will construct such (ample) equivariant line bundles for the case $\tau = 0$, i.e.~torsion free equivariant sheaves (Theorem \ref{ch. 1, sect. 3, thm. 4}). As a by-product, for reflexive equivariant sheaves, we can always construct particularly simple ample equivariant line bundles matching GIT stability and $\mu$-stability (subsection 4.4). For pure equivariant sheaves of lower dimension, the existence of equivariant line bundles matching Gieseker and GIT stability can be proved in specific examples. Note that in the classical construction of moduli spaces of Gieseker (semi)stable sheaves, one also needs to match GIT stability of the underlying GIT problem with Gieseker stability (see \cite[Thm.~4.3.3]{HL}). 

We can now prove the following results regarding representability and corepresentability.
\begin{proposition} \label{ch. 1, sect. 3, prop. 6}
Let $X$ be a nonsingular projective toric variety defined by a fan $\Delta$. Let $\mathcal{O}_{X}(1)$ be an ample line bundle on $X$, $\tau \in \Delta$ and $\vec{\chi} \in \mathcal{X}^{\tau}$. Then $\mathfrak{C}^{\tau, fr}_{\vec{\chi}}$ is represented by $\mathcal{N}^{\tau}_{\vec{\chi}}$. Assume we have an equivariant line bundle $\mathcal{L}_{\vec{\chi}}^{\tau}$ matching Gieseker and GIT stability. Let $\mathfrak{C}^{\tau, ss, fr}_{\vec{\chi}}$ respectively $\mathfrak{C}^{\tau, s, fr}_{\vec{\chi}}$ be the moduli subfunctors of $\mathfrak{C}^{\tau, fr}_{\vec{\chi}}$ with Gieseker semistable respectively geometrically Gieseker stable fibres. Then $\mathfrak{C}^{\tau, ss, fr}_{\vec{\chi}}$ is represented\footnote{In this setting, it is understood we use a choice of $\mathcal{L}_{\vec{\chi}}^{\tau}$ as in Assumption \ref{ch. 1, sect. 3, ass. 1} to define our notion of GIT stability and hence $\mathcal{N}^{\tau, ss}_{\vec{\chi}}$, $\mathcal{N}^{\tau, s}_{\vec{\chi}}$, $\mathcal{M}^{\tau, ss}_{\vec{\chi}}$, $\mathcal{M}^{\tau, s}_{\vec{\chi}}$.} by $\mathcal{N}^{\tau, ss}_{\vec{\chi}}$ and $\mathfrak{C}^{\tau, s, fr}_{\vec{\chi}}$ is represented by $\mathcal{N}^{\tau, s}_{\vec{\chi}}$.  
\end{proposition}
\begin{proof}
Recall that for $V$ a $k$-vector space of dimension $n$ and $0 \leq m \leq n$, one has a moduli functor of Grassmannians $\mathcal{G}\it{r}(m,V) : (Sch/k)^{o} \longrightarrow Sets$ (e.g.~\cite[Exm.~2.2.2]{HL}), where $\mathcal{G}\it{r}(m,V)(S)$ consists of quasi-coherent subsheaves $\mathcal{E} \subset V \otimes \mathcal{O}_{S}$ with $S$-flat cokernel of rank $n-m$ and $\mathcal{G}\it{r}(m,V)(f) = f^{*}$ is pull-back. Let $\mathcal{U}$ be the sheaf of sections of the tautological bundle $U \longrightarrow \mathrm{Gr}(m,V)$, then it is not difficult to see that $\mathcal{U}$ is a universal family. Consequently, $\mathcal{G}\it{r}(m,V)$ is represented by $\mathrm{Gr}(m,V)$. Likewise, for $m,n$ arbitrary nonnegative integers, one has a moduli functor of matrices $\mathcal{M}\it{at}(m,n) : (Sch/k)^{o} \longrightarrow Sets$, where $\mathcal{M}\it{at}(m,n)(S)$ consists of all morphisms $\phi : \mathcal{O}_{S}^{\oplus n} \longrightarrow \mathcal{O}_{S}^{\oplus m}$ and $\mathcal{M}\it{at}(m,n)(f) = f^{*}$ is pull-back. Let $(x_{ij})$ be a matrix of coordinates on $\mathrm{Mat}(m,n)$. Then $(x_{ij})$ induces a morphism $\xi : \mathcal{O}_{\mathrm{Mat}(m,n)}^{\oplus n} \longrightarrow \mathcal{O}_{\mathrm{Mat}(m,n)}^{\oplus m}$. Again, it is easy to see that $\xi$ is a universal family. Consequently, $\mathcal{M}\it{at}(m,n)$ is represented by $\mathrm{Mat}(m,n)$. Now consider $\mathcal{N}_{\vec{\chi}}^{\tau}$ as a closed subscheme of $\mathcal{A}$, where $\mathcal{A}$ is defined in equation (\ref{ch. 1, eqn3}). Since $\mathcal{A}$ is a product of various components $\mathrm{Gr}(\chi^{\sigma_{i}}(m), n(\Lambda_{1}, \ldots, \Lambda_{s}))$ and $\mathrm{Mat}(n(\Lambda_{1}, \ldots, \Lambda_{s}), n(\Lambda_{1}, \ldots, \Lambda_{i}+1 , \ldots, \Lambda_{s}))$, it is easy to see that the variety $\mathcal{A}$ represents the moduli functor formed from taking the Cartesian product of the moduli functors corresponding to these factors. The corresponding universal object is just a tuple with entries the various universal objects of the components $\mathrm{Gr}(\chi^{\sigma_{i}}(m), n(\Lambda_{1}, \ldots, \Lambda_{s}))$ and $\mathrm{Mat}(n(\Lambda_{1}, \ldots, \Lambda_{s}), n(\Lambda_{1}, \ldots, \Lambda_{i}+1 , \ldots, \Lambda_{s}))$ pulled-back to $\mathcal{A}$. 
Pulling this universal family back along the closed immersion $\mathcal{N}_{\vec{\chi}}^{\tau} \hookrightarrow \mathcal{A}$ gives a universal family of $\mathfrak{C}_{\vec{\chi}}^{\tau, fr}$. Consequently, $\mathfrak{C}_{\vec{\chi}}^{\tau, fr}$ is represented by $\mathcal{N}_{\vec{\chi}}^{\tau}$. Since $\mathcal{N}_{\vec{\chi}}^{\tau, s} \subset \mathcal{N}_{\vec{\chi}}^{\tau, ss} \subset \mathcal{N}_{\vec{\chi}}^{\tau}$ are open subschemes defined by properties on the fibres, the rest is easy.
\end{proof}
\begin{theorem} \label{ch. 1, sect. 3, thm. 2}
Let $X$ be a nonsingular projective toric variety defined by a fan $\Delta$. Let $\mathcal{O}_{X}(1)$ be an ample line bundle on $X$, $\tau \in \Delta$ and $\vec{\chi} \in \mathcal{X}^{\tau}$. Assume we have an equivariant line bundle matching Gieseker and GIT stability. Then $\underline{\mathcal{M}}_{\vec{\chi}}^{\tau, ss}$ is corepresented by the quasi-projective $k$-scheme of finite type $\mathcal{M}^{\tau, ss}_{\vec{\chi}}$. Moreover, there is an open subset $\mathcal{M}^{\tau, s}_{\vec{\chi}} \subset \mathcal{M}^{\tau, ss}_{\vec{\chi}}$ such that $\underline{\mathcal{M}}_{\vec{\chi}}^{\tau, s}$ is corepresented by $\mathcal{M}_{\vec{\chi}}^{\tau, s}$ and $\mathcal{M}^{\tau, s}_{\vec{\chi}}$ is a coarse moduli space. 
\end{theorem}
\begin{proof}
Define the moduli functor 
\begin{align} 
\begin{split}
&\mathcal{G} : (Sch/k)^{o} \longrightarrow Sets, \\
&\mathcal{G}(S) = \Big\{ \Phi \in \prod_{\scriptsize{\begin{array}{c} A_{1} \leq \Lambda_{1} \leq B_{1} \\ \cdots \\ A_{s} \leq \Lambda_{s} \leq B_{s} \end{array}}} \mathrm{Aut}\left(\mathcal{O}_{S}^{\oplus n(\Lambda_{1}, \ldots, \Lambda_{s})} \right) \ | \ \det(\Phi) = 1 \Big\}, \\
&\mathcal{G}(f) = f^{*}. \nonumber
\end{split}
\end{align}
It is easy to see that $\mathcal{G}$ is naturally represented by $G$ defined in equation (\ref{ch. 1, eqnG}). For any $S \in Sch/k$ we have a natural action of $\mathcal{G}(S)$ on $\mathfrak{C}_{\vec{\chi}}^{\tau, ss, fr}(S)$ and a natural action of $\mathrm{Hom}(S,G)$ on $\mathrm{Hom}(S,\mathcal{N}_{\vec{\chi}}^{\tau, ss})$. Since we have canonical isomorphisms $\mathcal{G} \cong \mathrm{Hom}(-,G)$ and $\mathfrak{C}_{\vec{\chi}}^{\tau, ss, fr} \cong \mathrm{Hom}(-,\mathcal{N}_{\vec{\chi}}^{\tau, ss})$ (Proposition \ref{ch. 1, sect. 3, prop. 6}), we get an isomorphism of functors
\begin{equation}
\mathfrak{C}_{\vec{\chi}}^{\tau, ss, fr} / \mathcal{G} \cong \mathrm{Hom}(-,\mathcal{N}_{\vec{\chi}}^{\tau, ss}) / \mathrm{Hom}(-,G). \nonumber
\end{equation}
Since $\mathcal{M}_{\vec{\chi}}^{\tau, ss} = \mathcal{N}_{\vec{\chi}}^{\tau, ss} /\!\!/G$ is a categorical quotient \cite[Def.~4.2.1]{HL}, we conclude that $\mathcal{M}_{\vec{\chi}}^{\tau, ss}$ corepresents $\mathrm{Hom}(-,\mathcal{N}_{\vec{\chi}}^{\tau, ss}) / \mathrm{Hom}(-,G)$ and therefore $\mathfrak{C}_{\vec{\chi}}^{\tau, ss, fr} / \mathcal{G}$. We also have canonical natural transformations\footnote{Here the symbol $\cong$ in the quotients refers to taking isomorphism classes in the corresponding categories.}
\begin{equation}
\mathfrak{C}_{\vec{\chi}}^{\tau, ss, fr} / \mathcal{G} = \left(\mathfrak{C}_{\vec{\chi}}^{\tau, ss, fr} / \cong \right) \Longrightarrow \left( \mathfrak{C}_{\vec{\chi}}^{\tau, ss} / \cong \right) \stackrel{\cong}{\Longrightarrow} \underline{\mathcal{M}}_{\vec{\chi}}^{\prime \tau, ss} \Longrightarrow \underline{\mathcal{M}}_{\vec{\chi}}^{\tau, ss}, \nonumber
\end{equation}
where the first natural transformation is injective over all $S \in Sch/k$ and we use Theorem \ref{ch. 1, sect. 3, thm. 1} to obtain the isomorphism $\left( \mathfrak{C}_{\vec{\chi}}^{\tau, ss} / \cong \right) \cong \underline{\mathcal{M}}_{\vec{\chi}}^{\prime \tau, ss}$. The moduli functors $\underline{\mathcal{M}}_{\vec{\chi}}^{\prime \tau, ss}$, $\underline{\mathcal{M}}_{\vec{\chi}}^{\tau, ss}$ have been introduced in subsection 3.1. 

We will show that $\mathcal{M}_{\vec{\chi}}^{\tau, ss}$ also corepresents $\left( \mathfrak{C}_{\vec{\chi}}^{\tau, ss} / \cong \right) \cong \underline{\mathcal{M}}_{\vec{\chi}}^{\prime \tau, ss}$ and $\underline{\mathcal{M}}_{\vec{\chi}}^{\tau, ss}$. This can be done by using open affine covers on which locally free sheaves respectively equivariant invertible sheaves trivialise. More precisely, we know $\mathcal{M}_{\vec{\chi}}^{\tau, ss}$ corepresents $\left( \mathfrak{C}_{\vec{\chi}}^{\tau,ss,fr} / \cong \right)$ 
\begin{equation}
\Phi : \left( \mathfrak{C}_{\vec{\chi}}^{\tau,ss,fr} / \cong \right) \longrightarrow \mathrm{Hom}(-,\mathcal{M}_{\vec{\chi}}^{\tau,ss}). \nonumber
\end{equation}
For a fixed $k$-scheme $S$ of finite type, define a morphism
\begin{equation}
\tilde{\Phi}_{S} : \left( \mathfrak{C}_{\vec{\chi}}^{\tau,ss} / \cong \right)(S) \longrightarrow \mathrm{Hom}(S,\mathcal{M}_{\vec{\chi}}^{\tau,ss}), \nonumber  
\end{equation}
as follows. Let $[\hat{\mathcal{F}}^{\Delta}] \in \left( \mathfrak{C}_{\vec{\chi}}^{\tau,ss} / \cong \right)(S)$ and let $\{\iota_{\alpha} : U_{\alpha} \hookrightarrow S\}_{\alpha \in I}$ be an open affine cover of $S$ on which the limiting locally free sheaves of $\hat{\mathcal{F}}^{\Delta}$ (i.e.~the $\mathcal{F}^{\sigma_{i}}(\Lambda_{1}, \ldots, \Lambda_{s}, \infty, \ldots, \infty)$) trivialise. Then $[\iota_{\alpha}^{*} \hat{\mathcal{F}}^{\Delta}] \in \left( \mathfrak{C}_{\vec{\chi}}^{\tau,ss,fr} / \cong \right)(U_{\alpha})$ and therefore we get a morphism $F_{\alpha} = \Phi_{U_{\alpha}}([\iota_{\alpha}^{*}\hat{\mathcal{F}}^{\Delta}]) : U_{\alpha} \longrightarrow \mathcal{M}_{\vec{\chi}}^{\tau,ss}$ for all $\alpha \in I$. From the fact that $\Phi$ is a natural transformation, it is easy to see that $\{F_{\alpha}\}_{\alpha \in I}$ glues to a morphism $F : S \longrightarrow \mathcal{M}_{\vec{\chi}}^{\tau,ss}$ independent of the choice of open affine cover. This defines $\tilde{\Phi}_{S}([\hat{\mathcal{F}}^{\Delta}])$. One readily verifies this defines a natural transformation $\tilde{\Phi}$ fitting in the commutative diagram
\begin{displaymath}
\xymatrix
{
\left( \mathfrak{C}_{\vec{\chi}}^{\tau,ss,fr} / \cong \right) \ar@{=>}[r] \ar@{=>}[d]_{\Phi} & \left( \mathfrak{C}_{\vec{\chi}}^{\tau,ss} / \cong \right) \ar@{=>}[dl]^{\tilde{\Phi}} \\
\mathrm{Hom}(-,\mathcal{M}_{\vec{\chi}}^{\tau,ss}). &
}
\end{displaymath} 
The fact that $\mathcal{M}_{\vec{\chi}}^{\tau,ss}$ corepresents $\left( \mathfrak{C}_{\vec{\chi}}^{\tau,ss,fr} / \cong \right)$ implies that it corepresents $\left( \mathfrak{C}_{\vec{\chi}}^{\tau,ss} / \cong \right) \cong \underline{\mathcal{M}}_{\vec{\chi}}^{\prime \tau,ss}$ too. Similarly, but easier, one proves $\mathcal{M}_{\vec{\chi}}^{\tau,ss}$ corepresents $\underline{\mathcal{M}}_{\vec{\chi}}^{\tau,ss}$. 

The proof up to now also holds in the case ``Gieseker stable''. By saying $\mathcal{M}_{\vec{\chi}}^{\tau,s}$ is a coarse moduli space, we mean $\mathcal{M}_{\vec{\chi}}^{\tau,s}$ corepresents $\underline{\mathcal{M}}_{\vec{\chi}}^{\tau,s}$ and $\underline{\mathcal{M}}_{\vec{\chi}}^{\tau,s}(k) \longrightarrow \mathrm{Hom}(k,\mathcal{M}_{\vec{\chi}}^{\tau,s})$ is bijective for any algebraically closed field $k$ of characteristic $0$. This is clearly the case since the closed points of $\mathcal{M}_{\vec{\chi}}^{\tau,s}$ are precisely the equivariant isomorphism classes of Gieseker stable equivariant sheaves on $X$ with support $V(\tau)$ and characteristic function $\vec{\chi}$. 
\end{proof}

We end this subsection by discussing how the theory developed in section 3 so far generalises to the case of possibly reducible support. Again, no essentially new ideas will occur, only the notation will become more cumbersome. Let $X$ be a nonsingular toric variety defined by a fan $\Delta$. Let $\tau_{1}, \ldots, \tau_{a} \in \Delta$ be some cones of some dimension $s$. Let $\sigma_{1}, \ldots, \sigma_{l} \in \Delta$ be all cones of maximal dimension having a cone $\tau_{\alpha}$ as a face. In subsection 2.2, we discussed how to describe pure equivariant sheaves on $X$ with support $V(\tau_{1}) \cup \cdots \cup V(\tau_{a})$. We define characteristic functions of such sheaves as in Definition \ref{ch. 1, sect. 3, def. 1}, we denote the set of all such characteristic functions by $\mathcal{X}^{\tau_{1}, \ldots, \tau_{a}}$ and we define the moduli functors $\underline{\mathcal{M}}_{\vec{\chi}}^{\tau_{1}, \ldots, \tau_{a},ss}$, $\underline{\mathcal{M}}_{\vec{\chi}}^{\tau_{1}, \ldots, \tau_{a},s}$ as in subsection 3.1. The obvious analogue of Theorem \ref{ch. 1, sect. 3, thm. 1} holds. The only new condition discussed in subsection 2.2 is condition (\ref{ch. 1, eqnnew}) in Proposition \ref{ch. 1, sect. 2, prop. 6}. This is an open condition on matrix coefficients so it can be easily incorporated. We can define a $k$-scheme of finite type $\mathcal{N}_{\vec{\chi}}^{\tau_{1}, \ldots, \tau_{a}}$ and a reductive algebraic group $G$ acting regularly on it as earlier in this subsection. Performing the GIT quotients w.r.t.~an equivariant line bundle $\mathcal{L} \in \mathrm{Pic}(\mathcal{N}_{\vec{\chi}}^{\tau_{1}, \ldots, \tau_{a}})$ gives rise to a categorical quotient $\mathcal{M}_{\vec{\chi}}^{\tau_{1}, \ldots, \tau_{a},ss}$ and a geometric quotient $\mathcal{M}_{\vec{\chi}}^{\tau_{1}, \ldots, \tau_{a},s}$ (both are quasi-projective schemes of finite type over $k$). It is straightforward to prove the following result. 
\begin{theorem} \label{ch. 1, sect. 3, thm. 3}
Let $X$ be a nonsingular projective toric variety defined by a fan $\Delta$. Let $\mathcal{O}_{X}(1)$ be an ample line bundle on $X$, $\tau_{1}, \ldots, \tau_{a} \in \Delta$ some cones of dimension $s$ and $\vec{\chi} \in \mathcal{X}^{\tau_{1}, \ldots, \tau_{a}}$. Assume we have an equivariant line bundle matching Gieseker and GIT stability. Then $\underline{\mathcal{M}}_{\vec{\chi}}^{\tau_{1}, \ldots, \tau_{a}, ss}$ is corepresented by the quasi-projective $k$-scheme of finite type $\mathcal{M}^{\tau_{1}, \ldots, \tau_{a}, ss}_{\vec{\chi}}$. Moreover, there is an open subset $\mathcal{M}^{\tau_{1}, \ldots, \tau_{a}, s}_{\vec{\chi}} \subset \mathcal{M}^{\tau_{1}, \ldots, \tau_{a}, ss}_{\vec{\chi}}$ such that $\underline{\mathcal{M}}_{\vec{\chi}}^{\tau_{1}, \ldots, \tau_{a}, s}$ is corepresented by $\mathcal{M}_{\vec{\chi}}^{\tau_{1}, \ldots, \tau_{a}, s}$ and $\mathcal{M}^{\tau_{1}, \ldots, \tau_{a}, s}_{\vec{\chi}}$ is a coarse moduli space. 
\end{theorem}

\noindent It is important to note that the moduli spaces of Theorems \ref{ch. 1, sect. 3, thm. 2} and \ref{ch. 1, sect. 3, thm. 3} are explicit and combinatorial in nature and their construction is very different from the construction of general moduli spaces of Gieseker (semi)stable sheaves, which makes use of Quot schemes and requires boundedness results \cite[Ch.~1--4]{HL}.

\subsection{Chern Characters of Equivariant Sheaves on Toric Varieties}

The Hilbert polynomial of a pure equivariant sheaf on a nonsingular projective toric variety with ample line bundle is entirely determined by the characteristic function of that sheaf. We will prove this by a short general argument in the following proposition.
\begin{proposition} \label{ch. 1, sect. 3, prop. 7}
Let $X$ be a nonsingular projective toric variety defined by a fan $\Delta$. Let $\mathcal{O}_{X}(1)$ be an ample line bundle on $X$, let $\tau_{1}, \ldots, \tau_{a} \in \Delta$ be some cones of dimension $s$ and let $\vec{\chi} \in \mathcal{X}^{\tau_{1}, \ldots, \tau_{a}}$. Then the Hilbert polynomial of any pure equivariant sheaf on $X$ with characteristic function $\vec{\chi}$ is the same. We refer to this polynomial as the Hilbert polynomial associated to $\vec{\chi}$. We denote the collection of all characteristic functions in $\mathcal{X}^{\tau_{1}, \ldots, \tau_{a}}$ having the same associated Hilbert polynomial $P$ by $\mathcal{X}_{P}^{\tau_{1}, \ldots, \tau_{a}}$.
\end{proposition}
\begin{proof}
Assume the fan $\Delta$ lies in a lattice $N$ of rank $r$. Let $\mathcal{E}$ be a pure equivariant sheaf on $X$ with characteristic function $\vec{\chi}$. The Hilbert polynomial of $\mathcal{E}$ is the unique polynomial $P_{\mathcal{E}}(t) \in \mathbb{Q}[t]$ satisfying 
\begin{equation}
P_{\mathcal{E}}(t) = \chi(\mathcal{E} \otimes \mathcal{O}_{X}(t)) = \sum_{i=0}^{r}(-1)^{i} \mathrm{dim}(H^{i}(X,\mathcal{E} \otimes \mathcal{O}_{X}(t))), \nonumber
\end{equation}
for all $t \in \mathbb{Z}$. Clearly, for a fixed $t \in \mathbb{Z}$, $\chi(\mathcal{E} \otimes \mathcal{O}_{X}(t))$ only depends on the equivariant isomorphism class $[\mathcal{E}]$ hence only on the isomorphism class $[\hat{E}^{\Delta}]$, where $\hat{E}^{\Delta}$ is the pure $\Delta$-family corresponding to $\mathcal{E}$. Note that $\chi(\mathcal{E} \otimes \mathcal{O}_{X}(t))$ does not vary if we vary the module structure of $\mathcal{E}$ \cite[Prop.~III.2.6]{Har1}. The module structure of $\mathcal{E}$ is encoded in the $k$-linear maps $\chi_{1}^{\sigma_{i}}(\vec{\lambda}), \ldots, \chi_{r}^{\sigma_{i}}(\vec{\lambda})$, where $i = 1, \ldots, l$ and $\vec{\lambda} \in \mathbb{Z}^{r} \cong M$. Here $\sigma_{1}, \ldots, \sigma_{l}$ are all cones of maximal dimension having a $\tau_{\alpha}$ ($\alpha = 1, \ldots, a$) as a face. Therefore, $\chi(\mathcal{E} \otimes \mathcal{O}_{X}(t))$ can only depend on the dimensions of the weight spaces of $\hat{E}^{\Delta}$, i.e.~only on $\vec{\chi}$.
\end{proof}

The fact that the Hilbert polynomial of a pure equivariant sheaf on a nonsingular projective toric variety with ample line bundle is entirely determined by the characteristic function of that sheaf can be made more specific by using a formula due to Klyachko. Klyachko gives an explicit formula for the Chern character of a torsion free equivariant sheaf on a nonsingular projective toric variety \cite[Sect.~1.2, 1.3]{Kly4}. We will now discuss Klyachko's Formula. The reader has to be aware of the fact that we follow Perling's convention of ascending directions for the maps of $\sigma$-families, as opposed to Klyachko's convention of descending directions. This results in some minus signs compared to Klyachko's formulae.
\begin{definition} \label{ch. 1, sect. 3, def. 3}
Let $\{E(\lambda_{1}, \ldots, \lambda_{r})\}_{(\lambda_{1}, \ldots, \lambda_{r}) \in \mathbb{Z}^{r}}$ be a collection of finite-dimensional $k$-vector spaces. For each $i = 1, \ldots, r$, we define a $\mathbb{Z}$-linear operator $\Delta_{i}$ on the free abelian group generated by the vector spaces $\{E(\lambda_{1}, \ldots, \lambda_{r})\}_{(\lambda_{1}, \ldots, \lambda_{r}) \in \mathbb{Z}^{r}}$ determined by
\begin{equation*}
\Delta_{i}E(\lambda_{1}, \ldots, \lambda_{r}) = E(\lambda_{1}, \ldots, \lambda_{r}) - E(\lambda_{1}, \ldots, \lambda_{i-1}, \lambda_{i}-1,\lambda_{i+1}, \ldots, \lambda_{r}),
\end{equation*}
for any $\lambda_{1}, \ldots, \lambda_{r} \in \mathbb{Z}$. This allows us to define $[E](\lambda_{1}, \ldots, \lambda_{r}) = \Delta_{1} \cdots \Delta_{r}E(\lambda_{1}, \ldots, \lambda_{r})$ for any $\lambda_{1}, \ldots, \lambda_{r} \in \mathbb{Z}$. One can then define dimension $\mathrm{dim}$ as a $\mathbb{Z}$-linear operator on the free abelian group generated by the vector spaces $\{E(\lambda_{1}, \ldots, \lambda_{r})\}_{(\lambda_{1}, \ldots, \lambda_{r}) \in \mathbb{Z}^{r}}$ in the obvious way. It now makes sense to consider $\mathrm{dim}([E](\lambda_{1}, \ldots, \lambda_{r}))$ for any $\lambda_{1}, \ldots, \lambda_{r} \in \mathbb{Z}$. For example
\begin{align*}
\mathrm{dim}([E](\lambda)) &= \mathrm{dim}(E(\lambda)) - \mathrm{dim}(E(\lambda-1)), \\ 
\mathrm{dim}([E](\lambda_{1},\lambda_{2})) &= \mathrm{dim}(E(\lambda_{1},\lambda_{2})) - \mathrm{dim}(E(\lambda_{1}-1,\lambda_{2})) - \mathrm{dim}(E(\lambda_{1},\lambda_{2}-1)) \\
& \ \ \ + \mathrm{dim}(E(\lambda_{1}-1,\lambda_{2}-1)),
\end{align*} 
for any $\lambda, \lambda_{1}, \lambda_{2} \in \mathbb{Z}$. \hfill $\oslash$
\end{definition}
\begin{proposition}[Klyachko's Formula] \label{ch. 1, sect. 3, prop. K}
Let $X$ be a nonsingular projective toric variety with fan $\Delta$ in a lattice $N$ of rank $r$. Let $\sigma_{1}, \ldots, \sigma_{l} \in \Delta$ be the cones of dimension $r$ and for each $i = 1, \ldots, l$, let $\left(\rho^{(i)}_{1}, \ldots, \rho^{(i)}_{r} \right)$ be the rays of $\sigma_{i}$. 
Then for any equivariant coherent sheaf $\mathcal{E}$ on $X$ 
and corresponding $\Delta$-family $\hat{E}^{\Delta}$, we have
\begin{equation*}
\mathrm{ch}(\mathcal{E}) = \sum_{\sigma \in \Delta, \ \vec{\lambda} \in \mathbb{Z}^{\mathrm{dim}(\sigma)}} (-1)^{\mathrm{codim}(\sigma)} \mathrm{dim}([E^{\sigma}](\vec{\lambda})) \ \mathrm{exp}\left( - \sum_{\rho \in \sigma(1)} \langle \vec{\lambda}, n(\rho)\rangle V(\rho) \right).
\end{equation*} 
\end{proposition}

\noindent In this proposition, $\sigma(1)$ means the collection of rays of $\sigma$. Likewise, we denote the collection of all rays of $\Delta$ by $\Delta(1)$. Any cone $\sigma \in \Delta$ is a face of a cone $\sigma_{i}$ of dimension $r$. Assume $\sigma$ has dimension $t$. Let $\hat{E}^{\sigma}$ denote the $\sigma$-family corresponding to the equivariant coherent sheaf $\mathcal{E}|_{U_{\sigma}}$. 
Let $\left(\rho^{(i)}_{1}, \ldots, \rho^{(i)}_{r} \right)$ be the rays of $\sigma_{i}$, with first integral lattice points $\left(n\left(\rho^{(i)}_{1}\right), \ldots, n\left(\rho^{(i)}_{r}\right) \right)$, and let without loss of generality $\left(\rho^{(i)}_{1}, \ldots, \rho^{(i)}_{t} \right) \subset \left(\rho^{(i)}_{1}, \ldots, \rho^{(i)}_{r} \right)$ be the rays of $\sigma$. Then $E^{\sigma}(\lambda_{1}, \ldots, \lambda_{t}) = E^{\sigma_{i}}(\lambda_{1}, \ldots, \lambda_{t}, \infty, \ldots, \infty)$ for all $\lambda_{1}, \ldots, \lambda_{t} \in \mathbb{Z}$ by Proposition \ref{ch. 1, sect. 2, prop. 5}.  A proof of Klyachko's Formula in the case of equivariant vector bundles can be found in \cite{Kly2} and also \cite{KS1}. Hence it holds for equivariant coherent sheaves in general as can be seen as follows. Any equivariant coherent sheaf $\mathcal{E}$ on $X$ admits a finite locally free equivariant resolution (\cite[Prop.~5.1.28]{CG})
\begin{equation*}
0 \longrightarrow \mathcal{F}_{n} \longrightarrow \cdots \longrightarrow \mathcal{F}_{0} \longrightarrow \mathcal{E} \longrightarrow 0.
\end{equation*}
Each equivariant locally free sheaf $\mathcal{F}_{i}$ has a corresponding $\Delta$-family $\hat{F}^{\Delta}_{i}$ which satisfies Klyachko's Formula. Hence
\begin{align*}
\mathrm{ch}(\mathcal{E}) &= \sum_{i=0}^{n} (-1)^{i} \mathrm{ch}(\mathcal{F}_{i}) \\
&=  \sum_{\sigma \in \Delta, \ \vec{\lambda} \in \mathbb{Z}^{\mathrm{dim}(\sigma)}} (-1)^{\mathrm{codim}(\sigma)}  \mathrm{exp}\left( - \sum_{\rho \in \sigma(1)} \langle \vec{\lambda}, n(\rho)\rangle V(\rho) \right) \sum_{i=0}^{n}  (-1)^{i} \mathrm{dim}([F_{i}^{\sigma}](\vec{\lambda})). 
\end{align*}
In terms of $\Delta$-families, the resolution gives an exact sequence
\begin{equation*}
0 \longrightarrow F_{n}^{\sigma}(\vec{\lambda}) \longrightarrow \cdots \longrightarrow F_{0}^{\sigma}(\vec{\lambda}) \longrightarrow E^{\sigma}(\vec{\lambda}) \longrightarrow 0,
\end{equation*}
for each $\sigma \in \Delta$ and $\vec{\lambda} \in \mathbb{Z}^{\mathrm{dim}(\sigma)}$. Hence
\begin{equation*}
\sum_{i=0}^{n}  (-1)^{i} \mathrm{dim}(F_{i}^{\sigma}(\vec{\lambda})) = \mathrm{dim}(E^{\sigma}(\vec{\lambda})),
\end{equation*}
for each $\sigma \in \Delta$ and $\vec{\lambda} \in \mathbb{Z}^{\mathrm{dim}(\sigma)}$ From this, the formula easily follows.
Note that Proposition \ref{ch. 1, sect. 3, prop. 7} now follows from Klyachko's Formula (Proposition \ref{ch. 1, sect. 3, prop. K}) and the Hirzebruch--Riemann--Roch Theorem \cite[Thm.~A.4.1]{Har1}. 

We end this subsection by proving a combinatorial result we will use in the next subsection. As a nice aside, applying this combinatorial result for $s=1$ to the above proposition, we recover a simple formula for the first Chern class due to Klyachko \cite[Sect.~1.2, 1.3]{Kly4}. 
\begin{proposition} \label{ch. 1, sect. 3, prop. C}
Let $\Delta$ be a simplicial fan in a lattice $N$ of rank $r$ with support $|\Delta| = N \otimes_{\mathbb{Z}} \mathbb{R}$. Let $\tau \in \Delta$ be a cone of dimension $s$. Then
\begin{equation*}
(-1)^{r-s} \sum_{a=0}^{r-s} (-1)^{a} \# \{ \sigma \in \Delta \ | \ \tau \prec \sigma \ \mathrm{and} \ \mathrm{dim}(\sigma) = a+s  \} = 1. 
\end{equation*}
\end{proposition}
\begin{proof}
Choose a basis for $N \otimes_{\mathbb{Z}} \mathbb{R}$ such that the first $s$ basis vectors generate $\tau$ and let $N \otimes_{\mathbb{Z}} \mathbb{R}$ be endowed with the standard inner product. Let $x$ be in the relative interior of $\tau$ and fix $\epsilon > 0$. Define a normal space $N_{x}\tau \cong \mathbb{R}^{r-s}$ to $\tau$ at $x$ and a sphere $S^{r-s-1} \subset N_{x}\tau$ using the standard inner product
\begin{equation*}
N_{x}\tau = \{ x + v \ | \ v \perp \tau \}, \ S^{r-s-1} = \{ x + v \ | \ v \perp \tau, \ |v| = \epsilon \}. 
\end{equation*} 
By definition, the union of all cones of $\Delta$ is $N \otimes_{\mathbb{Z}} \mathbb{R}$. Choosing $\epsilon > 0$ sufficiently small
\begin{equation*}
\{\sigma \cap S^{r-s-1} \ | \ \sigma \in \Delta, \ \tau \prec \sigma, \ \mathrm{dim}(\sigma) > s \},
\end{equation*}
forms a triangulation of $S^{r-s-1}$. Therefore 
\begin{equation*}
\sum_{a=1}^{r-s} (-1)^{a-1} \# \{ \sigma \in \Delta \ | \ \tau \prec \sigma \ \mathrm{and} \ \mathrm{dim}(\sigma) = a+s  \} = e(S^{r-s-1}),
\end{equation*}
where $e(S^{r-s-1})$ is the Euler characteristic of $S^{r-s-1}$ \cite[Sect.~22]{Mun}, which satisfies $e(S^{r-s-1}) = 0$ when $r-s$ is even and $e(S^{r-s-1}) = 2$ when $r-s$ is odd \cite[Thm.~31.8]{Mun}. 
\end{proof}
\begin{corollary} \label{ch. 1, sect. 3, cor. K}
Let $X$ be a nonsingular projective toric variety with fan $\Delta$ in a lattice $N$ of rank $r$. Let $\sigma_{1}, \ldots, \sigma_{l} \in \Delta$ be the cones of dimension $r$ and for each $i = 1, \ldots, l$, let $\left(\rho^{(i)}_{1}, \ldots, \rho^{(i)}_{r} \right)$ be the rays of $\sigma_{i}$. Then for any equivariant coherent sheaf $\mathcal{E}$ on $X$ with corresponding $\Delta$-family $\hat{E}^{\Delta}$, we have
\begin{equation*}
c_{1}(\mathcal{E}) = - \sum_{\rho \in \Delta(1), \ \lambda \in \mathbb{Z}} \lambda \ \mathrm{dim}([E^{\rho}](\lambda)) V(\rho).
\end{equation*}
\end{corollary}
\begin{proof}
Using Klyachko's Formula (Proposition \ref{ch. 1, sect. 3, prop. K}), we obtain
\begin{align*}
c_{1}(\mathcal{E}) &= - \sum_{\sigma \in \Delta, \ \vec{\lambda} \in \mathbb{Z}^{\mathrm{dim}(\sigma)}} \sum_{\rho \in \sigma(1)} (-1)^{\mathrm{codim}(\sigma)} \mathrm{dim}([E^{\sigma}](\vec{\lambda})) \langle \vec{\lambda}, n(\rho) \rangle V(\rho) \\
&= - \sum_{\rho \in \Delta(1)} \sum_{\lambda \in \mathbb{Z}} \sum_{\rho \prec \sigma \in \Delta} (-1)^{\mathrm{codim}(\sigma)} \lambda \  \mathrm{dim}([E^{\rho}](\lambda)) V(\rho).
\end{align*}
The corollary follows from applying Proposition \ref{ch. 1, sect. 3, prop. C} with $\tau = \rho \in \Delta(1)$ and $s = 1$.
\end{proof}

\subsection{Matching Stability}

In this subsection, we will prove the existence of ample equivariant line bundles matching Gieseker and GIT stability for torsion free equivariant sheaves on nonsingular projective toric varieties (Theorem \ref{ch. 1, sect. 3, thm. 4}). Along the way, we derive a number of important preparatory results as well as some results which are interesting on their own. 

As we have seen in Proposition \ref{ch. 1, sect. 2, prop. 2}, for a $G$-equivariant coherent sheaf, it is enough to test purity just for $G$-equivariant coherent subsheaves. It is natural to ask whether an analogous property holds for Gieseker stability. 
\begin{proposition} \label{ch. 1, sect. 3, prop. 8}
Let $X$ be a projective variety with ample line bundle $\mathcal{O}_{X}(1)$. Let $G$ be an affine algebraic group acting regularly on $X$. Let $\mathcal{E}$ be a pure $G$-equivariant sheaf on $X$. Then $\mathcal{E}$ is Gieseker semistable if and only if $p_{\mathcal{F}} \leq p_{\mathcal{E}}$ for any proper $G$-equivariant coherent subsheaf $\mathcal{F}$. Now assume $G = T$ is an algebraic torus. Then $\mathcal{E}$ is Gieseker stable if and only if $p_{\mathcal{F}} < p_{\mathcal{E}}$ for any proper equivariant coherent subsheaf $\mathcal{F}$. 
\end{proposition}
\begin{proof}
The statement on Gieseker semistability is clear by noting that the Harder--Narasimhan filtration of $\mathcal{E}$ is $G$-equivariant. For the definition of the Harder--Narasimhan filtration see \cite[Def.~1.3.2]{HL}. Now assume $G = T$ is an algebraic torus and for any proper equivariant coherent subsheaf $\mathcal{F} \subset \mathcal{E}$ one has $p_{\mathcal{F}} < p_{\mathcal{E}}$. We have to prove $\mathcal{E}$ is Gieseker stable. Since $\mathcal{E}$ is clearly Gieseker semistable, it contains a unique nontrivial maximal Gieseker polystable subsheaf $\mathcal{S}$ with the same reduced Hilbert polynomial as $\mathcal{E}$. The sheaf $\mathcal{S}$ is called the socle of $\mathcal{E}$ \cite[Lem.~1.5.5]{HL}. Because of uniqueness, $\mathcal{S}$ is an equivariant coherent subsheaf hence $\mathcal{E} = \mathcal{S}$. Therefore, there are $n \in \mathbb{Z}_{>0}$ mutually non-isomorphic Gieseker stable sheaves $\mathcal{E}_{1}, \ldots, \mathcal{E}_{n}$, positive integers $m_{1}, \ldots, m_{n}$ and an isomorphism of coherent sheaves $\theta : \bigoplus_{i=1}^{n} \mathcal{E}_{i}^{\oplus m_{i}} \stackrel{\cong}{\longrightarrow} \mathcal{E}$. Clearly, $p_{\mathcal{E}_{1}} = \cdots = p_{\mathcal{E}_{n}} = p_{\mathcal{E}}$. We claim that each $\mathcal{E}_{i}$ is isomorphic to an equivariant coherent subsheaf of $\mathcal{E}$. This would prove the proposition. We proceed in two steps. First we show $\mathcal{E}_{i}$ a priori admits an equivariant structure for each $i =1, \ldots, n$. Subsequently, we use representation theory of the algebraic torus $T$. Denote by $\Phi$ the equivariant structure on $\mathcal{E}$, by $\sigma : T \times X \longrightarrow X$ the regular action of $T$ on $X$, by $p_{2} : T \times X \longrightarrow X$ projection and by $T_{cl}$ the set of closed points of $T$. 

We claim each $\mathcal{E}_{i}$ admits an equivariant structure. By Proposition \ref{ch. 1, sect. 4, prop. 3} of subsection 4.2, it is enough to prove $\mathcal{E}_{i}$ is invariant, i.e.~$\sigma^{*} \mathcal{E}_{i} \cong p_{2}^{*} \mathcal{E}_{i}$. By Propositions \ref{ch. 1, sect. 4, prop. 2}, \ref{tom} of subsection 4.1, this is equivalent to $t^{*}\mathcal{E}_{i} \cong \mathcal{E}_{i}$ for all $t \in T_{cl}$. Note that we use $G = T$ is an algebraic torus. We now prove $t^{*}\mathcal{E}_{i} \cong \mathcal{E}_{i}$ for any $i = 1, \ldots, n$, $t \in T_{cl}$. Since each $\mathcal{E}_{i}$ is indecomposable, the Krull--Schmidt property of the category of coherent sheaves on $X$ \cite[Thm.~2]{Ati} implies for any $i = 1, \ldots, n$ and $t \in T_{cl}$ there is an isomorphism $t^{*}\mathcal{E}_{i} \cong \mathcal{E}_{j}$ for some $j = 1, \ldots, n$. Note that for $i,j = 1, \ldots, n$ we have \cite[Prop.~1.2.7]{HL}
\begin{equation*}
\mathrm{Hom}(\mathcal{E}_{i},\mathcal{E}_{j}) = \left\{\begin{array}{cc} k & \mathrm{if} \ i=j  \\ 0 & \mathrm{otherwise.} \end{array} \right.
\end{equation*}
Fix $i = 1, \ldots, n$ and define $\Gamma_{j} = \{ t \in T_{cl} \ | \ t^{*}\mathcal{E}_{i} \cong \mathcal{E}_{j} \}$ for each $j = 1, \ldots, n$. Each $\Gamma_{j}$ can be written as 
\begin{equation*}
\Gamma_{j} = \{ t \in T_{cl} \ | \ \mathrm{dim}(\mathrm{Hom}(t^{*}\mathcal{E}_{i},\mathcal{E}_{j})) \geq 1 \},
\end{equation*} 
by using the fact that any morphism between Gieseker stable sheaves with the same reduced Hilbert polynomial is zero or an isomorphism \cite[Prop.~1.2.7]{HL}. We deduce each $\Gamma_{j}$ is a closed subset by a semicontinuity argument. But each $\Gamma_{j}$ is also open, because its complement is the disjoint union $\coprod_{k \neq j} \Gamma_{k}$. Connectedness of $T_{cl}$ implies $T_{cl} = \Gamma_{i}$, since $1 \in \Gamma_{i}$. Therefore, we obtain an equivariant structure $\Psi^{(i)}$ on $\mathcal{E}_{i}$ for each $i = 1, \ldots, n$. 

Fix $i = 1, \ldots, n$. Using $\theta$, we obtain $\mathrm{Hom}(\mathcal{E}_{i},\mathcal{E}) \cong k^{\oplus m_{i}}$ and any nonzero element of $\mathrm{Hom}(\mathcal{E}_{i},\mathcal{E})$ is injective. The equivariant structures $\Phi$, $\Psi^{(i)}$ give us a regular representation of $T_{cl}$
\begin{align*}
&T_{cl} \times \mathrm{Hom}(\mathcal{E}_{i},\mathcal{E}) \longrightarrow \mathrm{Hom}(\mathcal{E}_{i},\mathcal{E}), \\ 
&t \cdot f = \Phi_{1}^{-1} \circ \Phi_{t^{-1}} \circ (t^{-1})^{*}(f) \circ \Psi_{t^{-1}}^{(i)-1} \circ \Psi_{1}^{(i)},
\end{align*}
where we define $\Phi_{t} = i_{t}^{*}\Phi$, $\Psi_{t}^{(i)} = i_{t}^{*} \Psi^{(i)}$ using the natural inclusion $i_{t} : X \hookrightarrow T \times X$ induced by $t \hookrightarrow T$ for any $t \in T_{cl}$. Now use that we are dealing with an algebraic torus $T$ to deduce there are 1-dimensional $k$-vector spaces $V_{1}^{(i)} = k \cdot v_{1}^{(i)}$, $\ldots$, $V_{m_{i}}^{(i)} = k \cdot v_{m_{i}}^{(i)}$ and characters $\chi_{1}^{(i)}, \ldots, \chi_{m_{i}}^{(i)} \in X(T)$ such that (Complete Reducibility Theorem \cite[Thm.~2.30]{Per1})
\begin{align*}
\mathrm{Hom}(\mathcal{E}_{i},\mathcal{E}) &= V_{1}^{(i)} \oplus \cdots \oplus V_{m_{i}}^{(i)}, \\ 
t \cdot v_{a}^{(i)} &= \chi_{a}^{(i)}(t) \cdot v_{a}^{(i)}, \ \forall t \in T_{cl} \ \forall a = 1, \ldots, m_{i}.
\end{align*}
Redefine $\tilde{\chi}_{a}^{(i)}(-) = \chi_{a}^{(i)}((-)^{-1}) \in X(T)$, $\tilde{v}_{a}^{(i)} = \Phi_{1} \circ v_{a}^{(i)} \circ \Psi_{1}^{(i)-1} \in \mathrm{Hom}(\mathcal{E}_{i},\mathcal{E}) \setminus 0$ and define $\tilde{\Psi}_{a}^{(i)}$ to be the equivariant structure on $\tilde{v}_{a}^{(i)}(\mathcal{E}_{i})$ induced by $\Psi^{(i)}$ for all $a = 1, \ldots, m_{i}$. We deduce
\begin{equation*}
\mathcal{E} \cong \bigoplus_{i=1}^{n} \left( \tilde{v}_{1}^{(i)}(\mathcal{E}_{i}) \oplus \cdots \oplus \tilde{v}_{m_{i}}^{(i)}(\mathcal{E}_{i}) \right), 
\end{equation*}
and the equivariant structure $\Phi$ induces an equivariant structure on each $\tilde{v}_{a}^{(i)}(\mathcal{E}_{i})$, denoted by $\Phi |_{\tilde{v}_{a}^{(i)}(\mathcal{E}_{i})}$ such that
\begin{equation*}
\Phi |_{\tilde{v}_{a}^{(i)}(\mathcal{E}_{i})} = \mathcal{O}(\tilde{\chi}_{a}^{(i)}) \otimes \tilde{\Psi}_{a}^{(i)}, \ \forall i = 1, \ldots, n \ \forall a = 1, \ldots, m_{i},
\end{equation*}
where $\mathcal{O}(\tilde{\chi}_{a}^{(i)})$ is the equivariant structure induced by the character $\tilde{\chi}_{a}^{(i)}$.
\end{proof}

The following proposition relates $\mu$-stability and GIT stability of torsion free equivariant sheaves on nonsingular projective toric varieties. Although we do not need this proposition for the proof of Theorem \ref{ch. 1, sect. 3, thm. 4}, which matches Gieseker and GIT stability for torsion free equivariant sheaves on nonsingular projective toric varieties in general, the proof is instructive. Moreover, the ample equivariant line bundles $\mathcal{L}_{\vec{\chi}}^{0,\mu}$ constructed in the proof are of a particularly simple form as opposed to the more complicated ample equivariant line bundles $\mathcal{L}_{\vec{\chi}}^{0}$ matching Gieseker and GIT stability of Theorem \ref{ch. 1, sect. 3, thm. 4}. Furthermore, the reasoning in the proof will be used later to construct particularly simple ample equivariant line bundles matching $\mu$-stability and GIT stability for reflexive equivariant sheaves on nonsingular projective toric varieties (subsection 4.4). Recall that a torsion free sheaf $\mathcal{E}$ on a nonsingular projective variety $X$ with ample line bundle $\mathcal{O}_{X}(1)$ is $\mu$-semistable, resp.~$\mu$-stable, if $\mu_{\mathcal{F}} \leq \mu_{\mathcal{E}}$, resp.~$\mu_{\mathcal{F}} < \mu_{\mathcal{E}}$, for any coherent subsheaf $\mathcal{F} \subset \mathcal{E}$ with $0 < \mathrm{rk}(\mathcal{F}) < \mathrm{rk}(\mathcal{E})$ \cite[Def.~1.2.12]{HL}. Denoting the Hilbert polynomial of $\mathcal{E}$ by $P_{\mathcal{E}}(t) = \sum_{i=0}^{n} \frac{\alpha_{i}(\mathcal{E})}{i!} t^{i}$, where $n = \mathrm{dim}(X)$, the rank of $\mathcal{E}$ is defined to be $\mathrm{rk}(\mathcal{E}) = \frac{\alpha_{n}(\mathcal{E})}{\alpha_{n}(\mathcal{O}_{X})}$, the degree of $\mathcal{E}$ is defined to be $\mathrm{deg}(\mathcal{E}) = \alpha_{n-1}(\mathcal{E}) - \alpha_{n-1}(\mathcal{O}_{X}) \cdot \mathrm{rk}(\mathcal{E})$ and the slope of $\mathcal{E}$ is defined to be $\mu_{\mathcal{E}} = \frac{\mathrm{deg}(\mathcal{E})}{\mathrm{rk}(\mathcal{E})}$ \cite[Def.~1.2.2, 1.2.11]{HL}. 
\begin{proposition} \label{ch. 1, sect. 3, prop. match}
Let $X$ be a nonsingular projective toric variety and let $\mathcal{O}_{X}(1)$ be an ample line bundle on $X$. Then for any $\vec{\chi} \in \mathcal{X}^{0}$, there is an ample equivariant line bundle $\mathcal{L}_{\vec{\chi}}^{0,\mu} \in \mathrm{Pic}^{G}(\mathcal{N}_{\vec{\chi}}^{0})$, such that any torsion free equivariant sheaf $\mathcal{E}$ on $X$ with characteristic function $\vec{\chi}$ satisfies
\begin{equation*}
\mathcal{E} \ is \ \mu\textrm{-}stable \Longrightarrow \mathcal{E} \ is \ properly \ GIT \ stable \ w.r.t.~\mathcal{L}_{\vec{\chi}}^{0,\mu} \Longrightarrow \mathcal{E} \ is \ \mu\textrm{-}semistable\footnote{Since we want to apply \cite[Thm.~11.1]{Dol} and the notion of GIT stable points in \cite{Dol} corresponds to the notion of \emph{properly} GIT stable points in \cite{MFK} (compare \cite[Sect.~8.1]{Dol} and \cite[Def.~1.8]{MFK}), we match ``Gieseker stable'' with ``\emph{properly} GIT stable''. Note that the results of section 3 so far continue to hold analogously in this setting.}.
\end{equation*}
\end{proposition}
\begin{proof}
We note that if $\mathcal{E}$ is a torsion free equivariant sheaf on $X$, then $\mathcal{E}$ is $\mu$-semistable if and only if for any equivariant coherent subsheaf $\mathcal{F} \subset \mathcal{E}$ with $0 < \mathrm{rk}(\mathcal{F}) < \mathrm{rk}(\mathcal{E})$ we have $\mu(\mathcal{F}) \leq \mu(\mathcal{E})$. This can be seen by noting that the Harder--Narasimhan filtration of $\mathcal{E}$ is equivariant. As an aside: note that we do \emph{not} prove $\mathcal{E}$ is $\mu$-stable if and only if for any equivariant coherent subsheaf $\mathcal{F} \subset \mathcal{E}$ with $0 < \mathrm{rk}(\mathcal{F}) < \mathrm{rk}(\mathcal{E})$ we have $\mu(\mathcal{F}) < \mu(\mathcal{E})$. We will prove this in the case $\mathcal{E}$ is reflexive in Proposition \ref{ch. 1, sect. 4, prop. 8} of subsection 4.4. For $\mathcal{E}$ torsion free equivariant, the problem is that if $\mu(\mathcal{F}) < \mu(\mathcal{E})$ for any equivariant coherent subsheaf $\mathcal{F} \subset \mathcal{E}$ with $0 < \mathrm{rk}(\mathcal{F}) < \mathrm{rk}(\mathcal{E})$, then indeed $\mathcal{E}$ is $\mu$-semistable and has a $\mu$-Jordan--H\"older filtration \cite[Sect.~1.6]{HL}, but the graded object $gr^{JH}(\mathcal{E})$ is only unique in codimension 1 \cite[Sect.~1.6]{HL}. Consequently, in the case of $\mu$-stability, we cannot mimic the proof of Proposition \ref{ch. 1, sect. 3, prop. 8}, which uses the socle of $\mathcal{E}$ and its uniqueness. 

Let $X$ be defined by a fan $\Delta$ in a lattice $N$ of rank $r$. Let $\mathcal{E}$ be a torsion free equivariant sheaf on $X$ with characteristic function $\vec{\chi}$ and corresponding framed torsion free $\Delta$-family $\hat{E}^{\Delta}$. Assume $\mathcal{E}$ has rank $M$ (we can assume $M \geq 2$ otherwise the proposition is trivial). Then $\hat{E}^{\Delta}$ consists of multi-filtrations $\{E^{\sigma_{i}}(\vec{\lambda})\}_{\vec{\lambda} \in \mathbb{Z}^{r}}$ of $k^{\oplus M}$, for each $i=1, \ldots, l$, such that each multi-filtration reaches $k^{\oplus M}$ (see Theorem \ref{ch. 1, sect. 2, thm. 1} and use the notation of this theorem). Moreover, for each $i = 1, \ldots, l$, there are integers $A_{1}^{(i)}, \ldots, A_{r}^{(i)}$ such that $E^{\sigma_{i}}(\vec{\lambda}) = 0$ unless $A_{1}^{(i)} \leq \lambda_{1}$, $\ldots$, $A_{r}^{(i)} \leq \lambda_{r}$ (let $A_{1}^{(i)}, \ldots, A_{r}^{(i)}$ be maximally chosen with this property). These multi-filtrations satisfy certain gluing conditions (see Theorem \ref{ch. 1, sect. 2, thm. 1}). Let $(\rho_{1}, \ldots, \rho_{N})$ be all rays and let $A_{1}, \ldots, A_{N}$ be the corresponding integers among the $A_{j}^{(i)}$ (this makes sense because of the gluing conditions). Fix $j = 1, \ldots, N$ and let $\sigma_{i}$ be some cone of maximal dimension having $\rho_{j}$ as a ray. Let $\left( \rho_{1}^{(i)}, \ldots, \rho_{r}^{(i)} \right)$ be the rays of $\sigma_{i}$ and let $\rho_{k}^{(i)} = \rho_{j}$. Consider the filtration  
\begin{equation}
\{\beta_{\lambda} \}_{\lambda \in \mathbb{Z}} = \left\{ \lim_{\stackrel{\lambda_{k} \rightarrow \lambda}{\lambda_{1}, \ldots, \lambda_{k-1},\lambda_{k+1} \ldots, \lambda_{r} \rightarrow \infty}} E^{\sigma_{i}}\left( \sum_{\alpha=1}^{r} \lambda_{\alpha} m(\rho_{\alpha}) \right) \right\}_{\lambda \in \mathbb{Z}}, \nonumber
\end{equation}
of $k^{\oplus M}$. Define integers $\Delta_{j}(1), \Delta_{j}(2), \ldots, \Delta_{j}(M-1) \in \mathbb{Z}_{\geq 0}$ and elements $p_{j}(1) \in \mathrm{Gr}(1,M)$, $p_{j}(2) \in \mathrm{Gr}(2,M)$, $\ldots$, $p_{j}(M-1) \in \mathrm{Gr}(M-1,M)$, such that the filtration changes value as follows
\begin{equation}
\beta_{\lambda} = \left\{\begin{array}{cc}  0 & \mathrm{if \ } \lambda < A_{j} \\ p_{j}(1) \in \mathrm{Gr}(1,M) & \mathrm{if \ } A_{j} \leq \lambda < A_{j} + \Delta_{j}(1) \\ p_{j}(2) \in \mathrm{Gr}(2,M) & \mathrm{if \ } A_{j} + \Delta_{j}(1) \leq \lambda < A_{j} + \Delta_{j}(1) + \Delta_{j}(2) \\ \ldots & \ldots \\ k^{\oplus M} & \mathrm{if \ } A_{j} + \Delta_{j}(1) + \Delta_{j}(2) + \cdots + \Delta_{j}(M-1) \leq \lambda. \end{array} \right. \nonumber
\end{equation}
Note that $\Delta_{j}(k) = 0$ is allowed. These definitions are independent of the cone $\sigma_{i}$ chosen, because of the gluing conditions. Denote the toric divisor $V(\rho_{j})$ corresponding to the ray $\rho_{j}$ by $D_{j}$. Using the formula for first Chern class of Corollary \ref{ch. 1, sect. 3, cor. K}, we easily obtain
\begin{equation}
\mathrm{ch}(\mathcal{E}) = M - \sum_{j=1}^{N}\left( MA_{j} + (M-1)\Delta_{j}(1) + (M-2)\Delta_{j}(2) + \cdots + \Delta_{j}(M-1) \right) D_{j} + O(2), \nonumber
\end{equation} 
where $O(2)$ means terms of degree $\geq 2$ in the Chow ring $A(X) \otimes_{\mathbb{Z}} \mathbb{Q}$. 
The Todd class of the tangent bundle $\mathcal{T}_{X}$ of $X$ is easily seen to be \cite[Sect.~5.3]{Ful}
\begin{equation}
\mathrm{td}(\mathcal{T}_{X}) = 1 + \frac{1}{2} \sum_{j=1}^{N} D_{j} + O(2). \nonumber
\end{equation}
Now $c_{1}(\mathcal{O}_{X}(1)) = \sum_{j = 1}^{N} \alpha_{j} D_{j}$ for some integers $\alpha_{j}$, since $D_{1}, \ldots, D_{N}$ generate $A^{1}(X)$ \cite[Sect.~5.2]{Ful}. Consequently,
\begin{equation}
\mathrm{ch}(\mathcal{O}_{X}(t)) = \frac{1}{r!} \left(\sum_{j=1}^{N} \alpha_{j} D_{j}\right)^{r} t^{r} + \frac{1}{(r-1)!} \left(\sum_{j=1}^{N} \alpha_{j} D_{j}\right)^{r-1} t^{r-1} + \cdots. \nonumber 
\end{equation}
From the Hirzebruch--Riemann--Roch Theorem \cite[Thm.~A.4.1]{Har1}, we obtain the two leading terms of the Hilbert polynomial of $\mathcal{E}$
\begin{align}
\begin{split}
P_{\mathcal{E}}(t) = &\frac{M}{r!} \ \mathrm{deg} \left\{ \left(\sum_{j=1}^{N} \alpha_{j} D_{j}\right)^{r} \right\}_{r} t^{r} + \frac{1}{(r-1)!} \ \mathrm{deg} \left\{ \frac{M}{2} \left( \sum_{j=1}^{N} \alpha_{j} D_{j} \right)^{r-1}\left( \sum_{j=1}^{N} D_{j} \right) \right. \\
&\left. - \left( \sum_{j=1}^{N} \alpha_{j} D_{j} \right)^{r-1} \left( \sum_{j=1}^{N}\left[MA_{j} + \sum_{k=1}^{M-1}(M-k)\Delta_{j}(k) \right] D_{j} \right) \right\}_{r} t^{r-1} + \cdots, \nonumber
\end{split}
\end{align}
where $\cdots$ means terms of degree $< r-1$ in $t$, $\{ - \}_{r}$ denotes the component of degree $r$ in $A(X) \otimes_{\mathbb{Z}} \mathbb{Q}$ and $\mathrm{deg} : A^{r}(X) \otimes_{\mathbb{Z}} \mathbb{Q} \longrightarrow \mathbb{Q}$ is the degree map \cite[App.~A]{Har1}. Let $0 \neq W \subset k^{\oplus M}$ be an $m$-dimensional subspace and let $\hat{F}^{\Delta} = \hat{E}^{\Delta} \cap W \subset \hat{E}^{\Delta}$ be the corresponding torsion free $\Delta$-family. Let $\mathcal{F}_{W} \subset \mathcal{E}$ be the corresponding equivariant coherent subsheaf. Analogous to the previous reasoning, one computes the two leading terms of the Hilbert polynomial of $\mathcal{F}_{W}$ to be
\begin{align}
\begin{split}
&P_{\mathcal{F}_{W}}(t) = \frac{m}{r!} \ \mathrm{deg}\left\{ \left(\sum_{j=1}^{N} \alpha_{j} D_{j}\right)^{r} \right\}_{r} t^{r} +\frac{1}{(r-1)!} \ \mathrm{deg} \left\{ \frac{m}{2} \left( \sum_{j=1}^{N} \alpha_{j} D_{j} \right)^{r-1} \left( \sum_{j=1}^{N} D_{j} \right) \right. \\
&\left. - \left(\sum_{j=1}^{N} \alpha_{j} D_{j} \right)^{r-1} \left( \sum_{j=1}^{N} \left[m A_{j} + \sum_{k=1}^{M-1} \left(m -  \mathrm{dim}(p_{j}(k) \cap W) \right)  \Delta_{j}(k) \right] D_{j} \right) \right\}_{r} t^{r-1} + \cdots, \nonumber
\end{split}
\end{align}
where the term on the second line can be straightforwardly derived by using induction on $M$. Before we continue, we make two remarks regarding positivity. Firstly, the leading coefficient of any Hilbert polynomial is positive, so $\mathrm{deg}\left\{ \left(\sum_{j=1}^{N} \alpha_{j} D_{j}\right)^{r} \right\}_{r} > 0$. Secondly, using the definition of degree of a coherent sheaf \cite[Def.~1.2.11]{HL} and the Nakai--Moishezon Criterion \cite[Thm.~A.5.1]{Har1}, we deduce that for any $j = 1, \ldots, N$
\begin{equation}
\mathrm{deg}(D_{j}) := \mathrm{deg}(\mathcal{O}_{X}(D_{j})) = \mathrm{deg} \left\{\left( \sum_{k=1}^{N} \alpha_{k} D_{k} \right)^{r-1} D_{j} \right\}_{r} > 0. \nonumber
\end{equation}
Combining our results so far and using the definition of slope of a coherent sheaf \cite[Def.~1.2.11]{HL}, we obtain that $\mathcal{E}$ is $\mu$-semistable if and only if for any subspace $0 \neq W \subsetneq k^{\oplus M}$ we have 
\begin{equation}
\sum_{j=1}^{N} \sum_{k=1}^{M-1} \Delta_{j}(k) \ \mathrm{deg}(D_{j}) \ \mathrm{dim}(p_{j}(k) \cap W) \leq \frac{\mathrm{dim}(W)}{M} \sum_{j=1}^{N} \sum_{k=1}^{M-1} \Delta_{j}(k) \ \mathrm{deg}(D_{j}) \ \mathrm{dim}(p_{j}(k)). \nonumber
\end{equation}  

We are now ready to prove the proposition. Let $\vec{\chi} \in \mathcal{X}^{0}$ be arbitrary. From $\vec{\chi}$ we read off the integers $A_{j}$, the rank $M$ (we assume $M\geq 2$ otherwise the proposition is trivial) and the non-negative integers $\Delta_{j}(k) \in \mathbb{Z}$. Without loss of generality, we can assume not all $\Delta_{j}(k) = 0$ (otherwise there are no $\mu$-stable torsion free equivariant sheaves on $X$ with characteristic function $\vec{\chi}$ and the proposition is trivial). In subsection 3.3, we defined the closed subscheme\footnote{To be be precise: for those $\Delta_{j}(k)$ which are zero, the corresponding term $\mathrm{Gr}(k,M)$ in the product of $\mathcal{A}'$ should be left out.}
\begin{equation}
\mathcal{N}_{\vec{\chi}}^{0} \subset \mathcal{A}' = \prod_{j=1}^{N} \prod_{k=1}^{M-1} \mathrm{Gr}(k,M) \times \prod_{\alpha=1}^{a} \mathrm{Gr}(n_{\alpha}, M). \nonumber
\end{equation} 
Here $a \in \mathbb{Z}_{\geq 0}$ and $0< n_{1}, \ldots, n_{a} < M$ are some integers. A closed point of $\mathcal{N}_{\vec{\chi}}^{0}$ is of the form $(p_{j}(k);q_{\alpha})_{j = 1, \ldots, N, k = 1, \ldots, M-1, \alpha = 1, \ldots, a}$, where there are certain compatibilities among the $p_{j}(k)$, $q_{\alpha}$ dictated by the shape of $\vec{\chi}$. An equivariant line bundle $\mathcal{L}_{\vec{\chi}}^{0, \mu \prime}$ on $\mathcal{A}'$ (up to equivariant isomorphism) is of the form $(\kappa_{jk}; \kappa_{\alpha})_{j = 1, \ldots, N, k = 1, \ldots, M-1, \alpha = 1, \ldots, a}$, where $\kappa_{jk}, \kappa_{\alpha}$ can be any integers \cite[Lem.~11.1]{Dol}. Such an equivariant line bundle is ample if and only if all $\kappa_{jk}, \kappa_{\alpha} > 0$ \cite[Sect.~11.1]{Dol}. The notion of GIT stability determined by such an ample equivariant line bundle is made explicit in \cite[Thm.~11.1]{Dol}. Suppose $a = 0$. Choose $\kappa_{jk} = \Delta_{j}(k) \ \mathrm{deg}(D_{j})$ for all $j,k$ and $\mathcal{L}_{\vec{\chi}}^{0, \mu} = \mathcal{L}_{\vec{\chi}}^{0, \mu \prime}|_{\mathcal{N}_{\vec{\chi}}^{0}}$. The proposition now follow easily from \cite[Thm.~11.1]{Dol} and \cite[Thm.~1.19]{MFK}. Now assume $a >0$. Let $R$ be a positive integer satisfying $0 < \sum_{\alpha = 1}^{a} \frac{n_{\alpha}}{R} < \frac{1}{M^{2}}$. Choose $\kappa_{jk} = \Delta_{j}(k) \ \mathrm{deg}(D_{j}) R$ and $\kappa_{\alpha} = 1$ for all $j,k,\alpha$. Note that any $\mu$-stable torsion free equivariant sheaf on $X$ with characteristic function $\vec{\chi}$ and corresponding framed torsion free $\Delta$-family defined by $(p_{j}(k);q_{\alpha})_{j = 1, \ldots, N, k = 1, \ldots, M-1, \alpha = 1, \ldots, a}$ satisfies
\begin{equation}
\frac{\mathrm{dim}(W)}{M} \sum_{j=1}^{N} \sum_{k=1}^{M-1} \Delta_{j}(k) \ \mathrm{deg}(D_{j}) \ \mathrm{dim}(p_{j}(k)) - \sum_{j=1}^{N} \sum_{k=1}^{M-1} \Delta_{j}(k) \ \mathrm{deg}(D_{j}) \ \mathrm{dim}(p_{j}(k) \cap W) \geq \frac{1}{M} \nonumber
\end{equation}
for any subspace $0 \neq W \subsetneq k^{\oplus M}$. Using \cite[Thm.~11.1]{Dol} and \cite[Thm.~1.19]{MFK} finishes the proof\footnote{Note that in this proof, we are not allowed to deduce ``$\mathcal{E}$ is properly GIT stable w.r.t.~$\mathcal{L}_{\vec{\chi}}^{0,\mu}$ $\Longrightarrow$ $\mathcal{E}$ is $\mu$-stable''.}.
\end{proof}

Here is our main result of this subsection, which explicitly matches Gieseker and GIT stability for torsion free sheaves in full generality.
\begin{theorem} \label{ch. 1, sect. 3, thm. 4}
Let $X$ be a nonsingular projective toric variety defined by a fan $\Delta$. Let $\mathcal{O}_{X}(1)$ be an ample line bundle on $X$. Then for any $\vec{\chi} \in \mathcal{X}^{0}$, there is an ample equivariant line bundle $\mathcal{L}^{0}_{\vec{\chi}} \in \mathrm{Pic}^{G}(\mathcal{N}^{0}_{\vec{\chi}})$ such that any torsion free equivariant sheaf $\mathcal{E}$ on $X$ with characteristic function $\vec{\chi}$ is GIT semistable resp.~properly GIT stable w.r.t.~$\mathcal{L}^{0}_{\vec{\chi}}$ if and only if $\mathcal{E}$ is Gieseker semistable resp.~Gieseker stable.
\end{theorem}
\begin{proof}
Let the fan $\Delta$ lie in a lattice $N \cong \mathbb{Z}^{r}$. Fix $\vec{\chi} \in \mathcal{X}^{0}$ and let $\mathcal{E}$ be a torsion free equivariant sheaf on $X$ with characteristic function $\vec{\chi}$. Denote the corresponding framed torsion free $\Delta$-family by $\hat{E}^{\Delta}$. Let $\Delta(1)$ be the collection of rays of $\Delta$, then the corresponding divisors $\{V(\rho)\}_{\rho \in \Delta(1)}$ generate $A^{1}(X)$ (\cite[Sect.~5.2]{Ful}), therefore we can write
\begin{equation*}
\mathcal{O}_{X}(1) \cong \mathcal{O}\left(\sum_{\rho \in \Delta(1)} \alpha_{\rho} V(\rho)\right).
\end{equation*} 
Using the Hirzebruch--Riemann--Roch Theorem \cite[Thm.~A.4.1]{Har1} and Klyachko's Formula (Proposition \ref{ch. 1, sect. 3, prop. K}), we obtain the Hilbert polynomial of $\mathcal{E}$
\begin{align*}
P_{\mathcal{E}}(t) &= \sum_{\sigma \in \Delta, \ \vec{\lambda} \in \mathbb{Z}^{\mathrm{dim}(\sigma)}} \Phi_{\sigma, \vec{\lambda}}(t) \ \mathrm{dim}([E^{\sigma}](\vec{\lambda})), \\
\Phi_{\sigma, \vec{\lambda}}(t) &= (-1)^{\mathrm{codim}(\sigma)} \mathrm{deg} \left\{ e^{-\sum_{\rho \in \sigma(1)} \langle \vec{\lambda}, n(\rho) \rangle V(\rho) + \sum_{\rho \in \Delta(1)} t \alpha_{\rho} V(\rho) } \mathrm{td}(\mathcal{T}_{X})  \right\}_{r},
\end{align*}
where $\mathrm{td}(\mathcal{T}_{X})$ is the Todd class of the tangent bundle $\mathcal{T}_{X}$ of $X$, $\{-\}_{r}$ projects to the degree $r$ component in the Chow ring $A(X) \otimes_{\mathbb{Z}} \mathbb{Q}$ and $\mathrm{deg} : A^{r}(X) \otimes_{\mathbb{Z}} \mathbb{Q} \longrightarrow \mathbb{Q}$ is the degree map. Let $\sigma_{1}, \ldots, \sigma_{l}$ be the cones of dimension $r$ and for each $\sigma_{i}$ denote by $(\rho_{1}^{(i)}, \ldots, \rho_{r}^{(i)})$ the rays of $\sigma_{i}$. For a fixed $j = 0, \ldots, r$, we denote the $\binom{r}{j}$ faces of $\sigma_{i}$ of dimension $j$ by $\sigma_{ijk}$. In other words, we choose a bijection between integers $k \in \{1, \ldots, \binom{r}{j}\}$ and $j$-tuples $i_{1} < \cdots < i_{j} \in \{1, \ldots, r\}$ and the cone $\sigma_{ijk}$ is by definition generated by the rays $(\rho_{i_{1}}^{(i)}, \ldots, \rho_{i_{j}}^{(i)})$. This allows us to rewrite 
\begin{equation*}
P_{\mathcal{E}}(t) = \sum_{i=1}^{l} \sum_{j=0}^{r} \sum_{k=1}^{\binom{r}{j}} \sum_{\vec{\lambda} \in \mathbb{Z}^{\mathrm{dim}(\sigma_{ijk})}} \Phi_{\sigma_{ijk}, \vec{\lambda}}(t) \ \mathrm{dim}([E^{\sigma_{ijk}}](\vec{\lambda})).
\end{equation*}
This sum is wrong as it stands. For fixed $(i,j,k)$, there can be distinct $(i',j',k')$ such that $\sigma_{ijk} = \sigma_{i'j'k'}$. In this case, we call $(i,j,k)$ and $(i',j',k')$ equivalent. This induces an equivalence relation on the set of such triples. We now choose exactly one representative $(i',j',k')$ of each equivalence class $[(i,j,k)]$, for which the sequence of polynomials $\{\Phi_{\sigma_{i'j'k'}, \vec{\lambda}}(t)\}_{\vec{\lambda} \in \mathbb{Z}^{\mathrm{dim}(\sigma_{i'j'k'})}}$ is possibly nonzero. This defines the above sum. Fix a component $\chi^{\sigma_{i}}$ of the characteristic function $\vec{\chi} = (\chi^{\sigma_{1}}, \ldots, \chi^{\sigma_{l}})$. There exists a box 
\begin{equation*}
\mathcal{B}^{\sigma_{i}} = (-\infty,C_{1}^{(i)}] \times \cdots \times (-\infty,C_{r}^{(i)}],
\end{equation*}
with the following properties. Write the box as 
\begin{align*}
\mathcal{B}^{\sigma_{i}} = &\mathcal{B}_{r}^{\sigma_{i}} \sqcup \mathcal{B}^{\sigma_{i}}_{r-1} \sqcup \cdots \sqcup \mathcal{B}_{0}^{\sigma_{i}}, \\
\mathcal{B}_{r}^{\sigma_{i}} = &(-\infty,C_{1}^{(i)}) \times \cdots \times (-\infty,C_{r}^{(i)}), \\
\mathcal{B}^{\sigma_{i}}_{r-1} = &(\{C_{1}^{(i)}\} \times (-\infty,C_{2}^{(i)}) \times \cdots \times (-\infty,C_{r}^{(i)})) \sqcup \cdots \\
&\sqcup ((-\infty,C_{1}^{(i)}) \times \cdots \times (-\infty,C_{r-1}^{(i)}) \times \{C_{r}^{(i)}\}), \\
&\!\!\!\!\! \!\! \cdots \\
\mathcal{B}^{\sigma_{i}}_{0} = &\{C_{1}^{(i)}\} \times \cdots \times \{C_{r}^{(i)}\},
\end{align*}
where $\mathcal{B}^{\sigma_{i}}_{j}$ is the disjoint union of $\binom{r}{j}$ sets $\mathcal{B}^{\sigma_{i}}_{j}(k)$. Here the labeling is such that the $j$ components of $\mathcal{B}^{\sigma_{i}}_{j}(k)$ in which we have an open interval correspond precisely to the $j$ rays $(\rho_{i_{1}}^{(i)}, \ldots, \rho_{i_{j}}^{(i)})$ that generate the cone $\sigma_{ijk}$, where $i_{1}< \cdots < i_{j} \in \{1, \ldots, r\}$ corresponds to $k$. Then for any $j = 0, \ldots, r$, $k=1, \ldots, \binom{r}{j}$ (corresponding to $i_{1}< \cdots < i_{j} \in \{1, \ldots, r\}$) and any polynomial $f(\lambda_{i_{1}}, \ldots, \lambda_{i_{j}}) \in k[\lambda_{i_{1}}, \ldots, \lambda_{i_{j}}]$  
\begin{align} 
\begin{split} \label{ch. 1, eqnrew}
&\sum_{(\lambda_{i_{1}}, \ldots, \lambda_{i_{j}}) \in \mathbb{Z}^{\mathrm{dim}(\sigma_{ijk})}} f(\lambda_{i_{1}}, \ldots, \lambda_{i_{j}}) \ \mathrm{dim}([E^{\sigma_{ijk}}](\lambda_{i_{1}}, \ldots, \lambda_{i_{j}})) \\
&= \sum_{(\lambda_{i_{1}}, \ldots, \lambda_{i_{j}}) \in \overline{\mathcal{B}^{\sigma_{i}}_{j}(k)}} f(\lambda_{i_{1}}, \ldots, \lambda_{i_{j}}) \ \mathrm{dim}([E^{\sigma_{ijk}}](\lambda_{i_{1}}, \ldots, \lambda_{i_{j}})), 
\end{split}
\end{align}
and if we let $a_{1} < \cdots < a_{r-j} \in \{1, \ldots, r\} \setminus \{i_{1}, \ldots, i_{j}\}$, then for all $\lambda_{i_{1}}, \ldots, \lambda_{i_{j}} \in \mathbb{Z}$
\begin{equation*}
\lim_{\lambda_{a_{1}} \rightarrow C_{a_{1}}^{(i)}, \ldots, \lambda_{a_{r-j}} \rightarrow C_{a_{r-j}}^{(i)}} E^{\sigma_{ijk}} \left(\sum_{i=1}^{r} \lambda_{i} m(\rho^{(i)}) \right) = \lim_{\lambda_{a_{1}} \rightarrow \infty, \ldots, \lambda_{a_{r-j}} \rightarrow \infty} E^{\sigma_{ijk}} \left(\sum_{i=1}^{r} \lambda_{i} m(\rho^{(i)}) \right). 
\end{equation*} 
The bar in equation (\ref{ch. 1, eqnrew}) denotes closure in $\mathbb{R}^{r}$. Let $f(\lambda_{i_{1}}, \ldots, \lambda_{i_{j}}) \in k[\lambda_{i_{1}}, \ldots, \lambda_{i_{j}}]$ and assume without loss of generality $(i_{1}, \ldots, i_{j}) = (1, \ldots, j)$. Then equation (\ref{ch. 1, eqnrew}) can be rewritten as 
\begin{align} 
&\sum_{\lambda_{1} = - \infty}^{C_{1}^{(i)}-1} \cdots \sum_{\lambda_{j} = - \infty}^{C_{j}^{(i)}-1} \left[ f(\lambda_{1}, \ldots, \lambda_{j}) - f(\lambda_{1}+1, \lambda_{2}, \ldots, \lambda_{j}) - \cdots - f(\lambda_{1}, \ldots, \lambda_{j-1}, \lambda_{j}+1) \right. \nonumber \\
&\left. + \cdots +(-1)^{j}f(\lambda_{1}+1, \ldots, \lambda_{j}+1) \right] \mathrm{dim}(E^{\sigma_{i}}(\lambda_{1}, \ldots, \lambda_{j}, \infty, \ldots, \infty)) \nonumber \\
&+ \sum_{\lambda_{2} = - \infty}^{C_{2}^{(i)}-1} \cdots \sum_{\lambda_{j} = - \infty}^{C_{j}^{(i)}-1} \left[ f(C_{1}^{(i)}, \lambda_{2}, \ldots, \lambda_{j}) - f(C_{1}^{(i)}, \lambda_{2}+1, \lambda_{3}, \ldots, \lambda_{j}) - \cdots \right. \label{star} \\
&\left. - f(C_{1}^{(i)}, \lambda_{2}, \ldots, \lambda_{j-1}, \lambda_{j}+1) + \cdots +(-1)^{j-1}f(C_{1}^{(i)}, \lambda_{2}+1, \ldots, \lambda_{j}+1) \right] \nonumber \\
&\cdot \mathrm{dim}(E^{\sigma_{i}}(\infty, \lambda_{2}, \ldots, \lambda_{j}, \infty, \ldots, \infty)) \nonumber \\
&+ \cdots \nonumber \displaybreak \\
&+\sum_{\lambda_{1} = - \infty}^{C_{1}^{(i)}-1} \cdots \sum_{\lambda_{j-1} = - \infty}^{C_{j-1}^{(i)}-1} \left[ f(\lambda_{1}, \ldots, \lambda_{j-1}, C_{j}^{(i)}) - f(\lambda_{1}+1, \lambda_{2}, \ldots, \lambda_{j-1}, C_{j}^{(i)}) - \cdots \right. \nonumber \\
&\left. - f(\lambda_{1}, \ldots, \lambda_{j-2}, \lambda_{j-1}+1, C_{j}^{(i)}) + \cdots +(-1)^{j-1}f(\lambda_{1}+1, \ldots, \lambda_{j-1}+1,C_{j}^{(i)}) \right] \nonumber \\
&\cdot \mathrm{dim}(E^{\sigma_{i}}(\lambda_{1}, \ldots, \lambda_{j-1},\infty, \ldots, \infty)) \nonumber \\
&+ \cdots \nonumber \\
&+f(C_{1}^{(i)}, \ldots, C_{j}^{(i)}) \ \mathrm{dim}(E^{\sigma_{i}}(\infty, \ldots, \infty)). \nonumber
\end{align} 
Using these manipulations only, we can write
\begin{equation*} 
P_{\mathcal{E}}(t) = \sum_{i=1}^{l} \sum_{j=0}^{r} \sum_{k=1}^{\binom{r}{j}} \sum_{\vec{\lambda} \in \mathcal{B}^{\sigma_{i}}_{j}(k)} \Psi_{\sigma_{ijk}, \vec{\lambda}}(t) \ \mathrm{dim}(E^{\sigma_{i}}(\vec{\lambda})).
\end{equation*}
We rewrite this expression one more time. For each $i=1, \ldots, l$, $j=0, \ldots, r$, $k = 1, \ldots, \binom{r}{j}$ we define $\Psi_{\sigma_{ijk}, \vec{\lambda}}(t) = 0$ in the case $\vec{\lambda} \notin \mathcal{B}^{\sigma_{i}}_{j}(k)$. Then we can write
\begin{align*} 
P_{\mathcal{E}}(t) = &\sum_{i=1}^{l} \sum_{j=0}^{r} \sum_{k=1}^{\binom{r}{j}} \sum_{\vec{\lambda} \in \mathcal{B}^{\sigma_{i}}_{j}(k)} \Psi_{\sigma_{ijk}, \vec{\lambda}}(t) \ \mathrm{dim}(E^{\sigma_{i}}(\vec{\lambda})) \\
= &\sum_{{\scriptsize{\begin{array}{c} \mathrm{equivalence \ classes} \\ \left[(i',j',k')\right] \end{array}}}} \sum_{\vec{\lambda} \in \mathbb{Z}^{\mathrm{dim}(\sigma_{i'j'k'})}} \underbrace{\left(\sum_{(i,j,k) \in [(i',j',k')]}  \Psi_{\sigma_{ijk}, \vec{\lambda}}(t) \right)}_{\Xi_{\sigma_{i'j'k'},\vec{\lambda}}(t)} \mathrm{dim}(E^{\sigma_{i'j'k'}}(\vec{\lambda})).
\end{align*}
Note that summing over all equivalence classes $[(i',j',k')]$ corresponds to summing over all cones of $\Delta$ and if $(i,j,k) \in [(i',j',k')]$, then $j=j'$. Recall that for fixed $i=1, \ldots, r$, there are integers $A_{1}^{(i)}, \ldots, A_{r}^{(i)}$ such that $E^{\sigma_{i}}(\lambda_{1}, \ldots, \lambda_{r}) = 0$ unless $\lambda_{1} \geq A_{1}^{(i)}$, $\ldots$, $\lambda_{r} \geq A_{r}^{(i)}$ (Proposition \ref{ch. 1, sect. 2, prop. 4}). Therefore, by construction, there are only finitely many $\Xi_{\sigma_{i'j'k'},\vec{\lambda}}(t)$ that are possibly non-zero. Consequently, for a fixed equivalence class $[(i',j',k')]$, there is a finite subset $\mathcal{R}([(i',j',k')]) \subset \mathbb{Z}^{\mathrm{dim}(\sigma_{i'j'k'})}$ such that 
\begin{equation} \label{ch. 1, eqnxi}
P_{\mathcal{E}}(t) = \sum_{{\scriptsize{\begin{array}{c} \mathrm{equivalence \ classes} \\ \left[(i',j',k')\right] \end{array}}}} \sum_{\vec{\lambda} \in \mathcal{R}([(i',j',k')])} \Xi_{\sigma_{i'j'k'},\vec{\lambda}}(t) \ \mathrm{dim}(E^{\sigma_{i'j'k'}}(\vec{\lambda})).
\end{equation}
Note that the polynomials $\Phi_{\sigma, \vec{\lambda}}(t)$, the boxes $\mathcal{B}^{\sigma_{i}}_{j}(k)$ and hence the polynomials $\Psi_{\sigma_{ijk}, \vec{\lambda}}(t)$, $\Xi_{\sigma_{i'j'k'}, \vec{\lambda}}(t)$ and the regions $\mathcal{R}([i',j',k'])$ can be chosen such that they only depend on the characteristic function $\vec{\chi}$. Therefore, equation (\ref{ch. 1, eqnxi}) holds for any torsion free equivariant sheaf $\mathcal{E}$ on $X$ with characteristic function $\vec{\chi}$ and corresponding framed torsion free $\Delta$-family $\hat{E}^{\Delta}$. Now let $\mathcal{E}$ be a torsion free equivariant sheaf on $X$ with characteristic function $\vec{\chi}$ and corresponding framed torsion free $\Delta$-family $\hat{E}^{\Delta}$. Assume the rank of $\mathcal{E}$ is $M$. Let $0 \neq W \subsetneq k^{\oplus M} = E^{\sigma_{i}}(\infty, \ldots, \infty)$ be a linear subspace. Consider the torsion free $\Delta$-family $\hat{F}^{\Delta}_{W} = \hat{E}^{\Delta} \cap W$ and denote the corresponding torsion free equivariant sheaf by $\mathcal{F}_{W}$. It is not difficult to see that 
\begin{equation*} 
P_{\mathcal{F}_{W}}(t) = \sum_{{\scriptsize{\begin{array}{c} \mathrm{equivalence \ classes} \\ \left[(i',j',k')\right] \end{array}}}} \sum_{\vec{\lambda} \in \mathcal{R}([(i',j',k')])} \Xi_{\sigma_{i'j'k'},\vec{\lambda}}(t) \ \mathrm{dim}(E^{\sigma_{i'j'k'}}(\vec{\lambda}) \cap W).
\end{equation*}
Using Proposition \ref{ch. 1, sect. 3, prop. 8}, we see that $\mathcal{E}$ is Gieseker semistable if and only if for any linear subspace $0 \neq W \subsetneq k^{\oplus M}$ and $t \gg 0$
\begin{align*}
&\frac{1}{\mathrm{dim}(W)} \sum_{{\scriptsize{\begin{array}{c} \mathrm{equivalence \ classes} \\ \left[(i',j',k')\right] \end{array}}}} \sum_{\vec{\lambda} \in \mathcal{R}([(i',j',k')])} \Xi_{\sigma_{i'j'k'},\vec{\lambda}}(t) \ \mathrm{dim}(E^{\sigma_{i'j'k'}}(\vec{\lambda}) \cap W) \\
&\leq \frac{1}{M} \sum_{{\scriptsize{\begin{array}{c} \mathrm{equivalence \ classes} \\ \left[(i',j',k')\right] \end{array}}}} \sum_{\vec{\lambda} \in \mathcal{R}([(i',j',k')])} \Xi_{\sigma_{i'j'k'},\vec{\lambda}}(t) \ \mathrm{dim}(E^{\sigma_{i'j'k'}}(\vec{\lambda})).
\end{align*}
Moreover, $\mathcal{E}$ is Gieseker stable if and only if the same holds with strict inequality. 

We now have to study the polynomials $\Xi_{\sigma_{i'j'k'},\vec{\lambda}}(t)$ in more detail. Fix $i'=1, \ldots, l$, $j' = 1, \ldots, r$ and $k' = 1, \ldots, \binom{r}{j'}$. We will now show that for any $\vec{\lambda} \in \mathcal{R}([(i',j',k')])$, we have $\Xi_{\sigma_{i'j'k'}, \vec{\lambda}}(R) \in \mathbb{Z}_{> 0}$ for integers $R \gg 0$. From the fact that the polynomials $\Phi_{\sigma,\vec{\lambda}}(t)$ are integer-valued for $t \in \mathbb{Z}$ it easily follows that the polynomials $\Xi_{\sigma_{i'j'k'}, \vec{\lambda}}(t)$ are integer-valued for $t \in \mathbb{Z}$. Let $\mathcal{E}$ be an arbitrary torsion free equivariant sheaf on $X$ with characteristic function $\vec{\chi}$ and corresponding framed torsion free $\Delta$-family $\hat{E}^{\Delta}$. Consider the face $\sigma_{i'j'k'} \prec \sigma_{i'}$ and assume without loss of generality that it is spanned by the rays $(\rho_{1}^{(i')}, \ldots, \rho_{j'}^{(i')}) \subset (\rho_{1}^{(i')}, \ldots, \rho_{r}^{(i')})$. Consider the expression 
\begin{equation*}
\Xi_{\sigma_{i'j'k'}, \vec{\lambda}}(t) \ \mathrm{dim}(E^{\sigma_{i'}}(\lambda_{1}, \ldots, \lambda_{j'}, \infty, \ldots, \infty)),
\end{equation*}
for fixed $\vec{\lambda} \in \mathcal{R}([(i',j',k')])$. We first claim $\Xi_{\sigma_{i'j'k'}, \vec{\lambda}}(t)$ is a polynomial in $t$ of degree at most $r-j'$. To see this, consider expression (\ref{star}) for $i=i'$ and $j \geq j'$. Then for any monomial $f$ of degree $< j'$, expression (\ref{star}) does not contribute to $\Xi_{\sigma_{i'j'k'}, \vec{\lambda}}(t)$. We now want to show $\Xi_{\sigma_{i'j'k'}, \vec{\lambda}}(t)$ is of degree $r-j'$ in $t$ with positive leading coefficient. Fix $i,j$ as before and consider expression (\ref{star}) for any monomial $f$ of degree $j'$. We only get a contribution to the leading term of $\Xi_{\sigma_{i'j'k'}, \vec{\lambda}}(t)$ for $f(\lambda_{1}, \ldots, \lambda_{j}) = \lambda_{1} \cdots \lambda_{j'}$. From this, it is easy to see that the leading term of $\Xi_{\sigma_{i'j'k'}, \vec{\lambda}}(t)$ is 
\begin{align*}
&\sum_{a=0}^{r-j'} \frac{(-1)^{r-(j'+a)}}{(r-j')!} \ \# \{ \sigma \in \Delta \ | \ \sigma_{i'j'k'} \prec \sigma, \ \mathrm{dim}(\sigma) = a+j'  \} \\
&\cdot \left( H^{r-j'} \cdot V\left( \rho^{(i')}_{1} \right) \cdots V\left( \rho^{(i')}_{j'} \right)\right) t^{r-j'} = \frac{1}{(r-j')!} \left( H^{r-j'} \cdot V\left( \rho^{(i')}_{1} \right) \cdots V\left( \rho^{(i')}_{j'} \right)\right) t^{r-j'},
\end{align*}
where we use Proposition \ref{ch. 1, sect. 3, prop. C}. Let $\nu$ be the cone generated by the rays $\rho^{(i')}_{1}$, $\ldots$, $\rho^{(i')}_{j'}$, then $V\left( \rho^{(i')}_{1} \right) \cap \cdots \cap V\left( \rho^{(i')}_{j'} \right) = V(\nu)$ is a nonsingular closed subvariety of $X$ of dimension $r-j'$ \cite[Sect.~3.1]{Ful}. We deduce that $H^{r-j'} \cdot V\left( \rho^{(i')}_{1} \right) \cdots V\left( \rho^{(i')}_{j'} \right) = H^{r-j'} \cdot V\left( \nu \right) > 0$ by the Nakai--Moishezon Criterion \cite[Thm.~A.5.1]{Har1}. 

Let $[(i',j',k')]$ be an equivalence class and let $k'=1, \ldots, \binom{r}{j'}$ correspond to $i_{1} < \cdots < i_{j'} \in \{1, \ldots, r\}$. Assume $a_{1} < \cdots < a_{r-j'} \in \{1, \ldots, r\} \setminus \{i_{1}, \ldots, i_{j'}\}$. Define $\chi^{\sigma_{i',j',k'}}(\vec{\lambda}) = \lim_{\lambda_{a_{1}} \rightarrow \infty, \ldots, \lambda_{a_{r-j'}} \rightarrow \infty} \chi^{\sigma_{i'}}(\vec{\lambda})$ for all $\vec{\lambda} = (\lambda_{i_{1}}, \ldots, \lambda_{i_{j'}}) \in \mathbb{Z}^{\mathrm{dim}(\sigma_{i'j'k'})}$. Consider the product of Grassmannians
\begin{equation} \label{ch. 1, eqngrass}
\prod_{{\scriptsize{\begin{array}{c} \mathrm{equivalence \ classes} \\ \left[(i',j',k')\right] \end{array}}}} \prod_{\vec{\lambda} \in \mathcal{R}([i',j',k'])} \mathrm{Gr}(\chi^{\sigma_{i'j'k'}}(\vec{\lambda}),M).
\end{equation}
Referring to \cite[Sect.~11.1]{Dol}, the equivariant line bundles (up to equivariant isomorphism) on the product of Grassmannians (\ref{ch. 1, eqngrass}) correspond to arbitrary sequences of integers $\{k_{[(i',j',k')],\vec{\lambda}}\}$, where $([(i',j',k')],\vec{\lambda}) \in \coprod_{{\scriptsize{\begin{array}{c} \mathrm{equivalence \ classes} \\ \left[(i',j',k')\right] \end{array}}}} \mathcal{R}([(i',j',k')])$. Such an equivariant line bundle $\{k_{[(i',j',k')],\vec{\lambda}}\}$ is ample if and only if all $k_{[(i',j',k')],\vec{\lambda}}>0$ \cite[Sect.~11.1]{Dol}. We conclude that by choosing an integer $R \gg 0$, the sequence $\{\Xi_{\sigma_{i'j'k'}, \vec{\lambda}}(R)\}$ forms an ample equivariant line bundle on the product of Grassmannians (\ref{ch. 1, eqngrass}). The notion of GIT stability determined by such an ample equivariant line bundle is made explicit in \cite[Thm.~11.1]{Dol}. By definition, $\mathcal{N}_{\vec{\chi}}^{0}$ is a closed subscheme of the product of Grassmannians (\ref{ch. 1, eqngrass}). Using \cite[Thm.~11.1]{Dol}, we see that a the closed point $\hat{E}^{\Delta}$ of $\mathcal{N}_{\vec{\chi}}^{0}$ is GIT semistable w.r.t.~$\{\Xi_{\sigma_{i'j'k'}, \vec{\lambda}}(R)\}$ if and only if for any linear subspace $0 \neq W \subsetneq k^{\oplus M}$
\begin{align*}
&\frac{1}{\mathrm{dim}(W)} \sum_{{\scriptsize{\begin{array}{c} \mathrm{equivalence \ classes} \\ \left[(i',j',k')\right] \end{array}}}} \sum_{\vec{\lambda} \in \mathcal{R}([(i',j',k')])} \Xi_{\sigma_{i'j'k'},\vec{\lambda}}(R) \ \mathrm{dim}(E^{\sigma_{i'j'k'}}(\vec{\lambda}) \cap W) \\
&\leq \frac{1}{M} \sum_{{\scriptsize{\begin{array}{c} \mathrm{equivalence \ classes} \\ \left[(i',j',k')\right] \end{array}}}} \sum_{\vec{\lambda} \in \mathcal{R}([(i',j',k')])} \Xi_{\sigma_{i'j'k'},\vec{\lambda}}(R) \ \mathrm{dim}(E^{\sigma_{i'j'k'}}(\vec{\lambda})).
\end{align*}
Moreover, $\hat{E}^{\Delta}$ is properly GIT stable w.r.t.~$\{\Xi_{\sigma_{i'j'k'}, \vec{\lambda}}(R)\}$ if and only if the same holds with strict inequality. By choosing $R$ sufficiently large, we conclude any torsion free equivariant sheaf $\mathcal{E}$ on $X$ with characteristic function $\vec{\chi}$ and framed torsion free $\Delta$-family $\hat{E}^{\Delta}$ is Gieseker semistable resp.~Gieseker stable if and only if $\hat{E}^{\Delta}$ is GIT semistable resp.~properly GIT stable w.r.t.~$\{\Xi_{\sigma_{i'j'k'}, \vec{\lambda}}(R)\}$. Pulling back the ample equivariant line bundle $\{\Xi_{\sigma_{i'j'k'}, \vec{\lambda}}(R)\}$ to $\mathcal{N}_{\vec{\chi}}^{0}$ defines the desired ample equivariant line bundle $\mathcal{L}_{\vec{\chi}}^{0}$ and finishes the proof by \cite[Thm.~1.19]{MFK}. 
\end{proof}

\section{Fixed Point Loci of Moduli Spaces of Sheaves on Toric Varieties}
 
In this section, we study how the explicit moduli spaces of pure equivariant sheaves of the previous section relate to fixed point loci of moduli spaces of all Gieseker stable sheaves on nonsingular projective toric varieties. We start by studying the torus action on moduli spaces of Gieseker semistable sheaves on projective toric varieties. Subsequently, we study relations between equivariant and invariant simple sheaves. We prove a theorem expressing fixed point loci of moduli spaces of all Gieseker stable sheaves on an arbitrary nonsingular projective toric variety in terms of the explicit moduli spaces of pure equivariant sheaves of the previous section in the case one can match Gieseker and GIT stability. Since this match can always be achieved for torsion free equivariant sheaves, we obtain the theorem discussed in the introduction (Theorem \ref{ch. 1, sect. 1, thm. 1}). After discussing some examples appearing in a sequel (\cite{Koo}), where we specialise to $X$ a nonsingular complete toric surface over $\mathbb{C}$, we prove how the fixed point locus of any moduli space of $\mu$-stable sheaves on a nonsingular projective toric variety can be expressed combinatorially.

\subsection{Torus Actions on Moduli Spaces of Sheaves on Toric Varieties}

Let us briefly recall some notions concerning moduli spaces of Gieseker (semi)stable sheaves in general as discussed in \cite[Ch.~4]{HL}. Let $X$ be a connected projective $k$-scheme, $\mathcal{O}_{X}(1)$ an ample line bundle and $P$ a choice of Hilbert polynomial. Let $S$ be a $k$-scheme of finite type. Any two $S$-flat families $\mathcal{F}_{1}, \mathcal{F}_{2}$ are said to be equivalent if there exists a line bundle $L \in \mathrm{Pic}(S)$ and an isomorphism $\mathcal{F}_{1} \cong \mathcal{F}_{2} \otimes p_{S}^{*}L$, where $p_{S} : X \times S \longrightarrow S$ is projection. Let $\underline{\mathcal{M}}_{P}^{ss}(S)$ be the collection of equivalence classes of Gieseker semistable $S$-flat families with Hilbert polynomial $P$. Likewise, let $\underline{\mathcal{M}}_{P}^{s}(S)$ be the collection of equivalence classes of geometrically Gieseker stable $S$-flat families with Hilbert polynomial $P$. These give rise to moduli functors $\underline{\mathcal{M}}_{P}^{ss}$, $\underline{\mathcal{M}}_{P}^{s}$. One can prove that there is a projective $k$-scheme of finite type $\mathcal{M}_{P}^{ss}$ corepresenting $\underline{\mathcal{M}}_{P}^{ss}$ \cite[Thm.~4.3.4]{HL}. Moreover, there is an open subset $\mathcal{M}_{P}^{s}$ of $\mathcal{M}_{P}^{ss}$ corepresenting $\underline{\mathcal{M}}_{P}^{s}$. The closed points of $\mathcal{M}_{P}^{ss}$ are in bijection with $S$-equivalence classes of Gieseker semistable sheaves on $X$ with Hilbert polynomial $P$. The closed points of $\mathcal{M}_{P}^{s}$ are in bijection with isomorphism classes of Gieseker stable sheaves on $X$ with Hilbert polynomial $P$ (hence $\mathcal{M}_{P}^{s}$ is a coarse moduli space). If $X$ is a toric variety with torus $T$, then we can define a natural regular action of $T$ on $\mathcal{M}_{P}^{ss}$, $\mathcal{M}_{P}^{s}$ as is expressed by the following proposition. 
\begin{proposition} \label{ch. 1, sect. 4, prop. 1}
Let $X$ be a projective toric variety with torus $T$. Let $\mathcal{O}_{X}(1)$ be an ample line bundle on $X$ and $P$ a choice of Hilbert polynomial. Choose an equivariant structure on $\mathcal{O}_{X}(1)$. Then there is a natural induced regular action $\sigma : T \times \mathcal{M}_{P}^{ss} \longrightarrow \mathcal{M}_{P}^{ss}$ defined using the equivariant structure, which restricts to $\mathcal{M}_{P}^{s}$ and on closed points is given by
\begin{align} 
\begin{split}
\sigma : T_{cl} \times \mathcal{M}^{ss}_{P,cl} &\longrightarrow \mathcal{M}^{ss}_{P,cl}, \\
\sigma(t,[\mathcal{E}]) &\mapsto [t^{*}\mathcal{E}]. \nonumber
\end{split}
\end{align}
\end{proposition}
\begin{proof}
Denote the action of the torus by $\sigma : T \times X \longrightarrow X$. Let $m$ be an integer such that any Gieseker semistable sheaf on $X$ with Hilbert polynomial $P$ is $m$-regular \cite[Sect.~4.3]{HL}. Let $V = k^{\oplus P(m)}$ and $\mathcal{H} = V \otimes_{k} \mathcal{O}_{X}(-m)$. We start by noting that any line bundle on $X$ admits a $T$-equivariant structure, since $\mathrm{Pic}(T) = 0$ (e.g.~\cite[Thm.~7.2]{Dol}). Let $\Phi$ be a $T$-equivariant structure on $\mathcal{O}_{X}(1)$. The $T$-equivariant structure $\Phi$ induces a $T$-equivariant structure on $\mathcal{O}_{X}(-m)$ and therefore it induces a $T$-equivariant structure $\Phi_{\mathcal{H}} : \sigma^{*}\mathcal{H} \longrightarrow p_{2}^{*}\mathcal{H}$ (where we let $T$ act trivially on $V$). Let $\mathcal{Q} = \underline{\mathrm{Quot}}_{X/k}(\mathcal{H},P)$ be the Quot functor and let $Q = \mathrm{Quot}_{X/k}(\mathcal{H},P)$ be the Quot scheme. Here $Q$ is a projective $k$-scheme representing $\mathcal{Q}$ \cite[Sect.~2.2]{HL}
\begin{equation} \nonumber
\Xi : \mathcal{Q} \stackrel{\cong}{\Longrightarrow} \underline{Q},
\end{equation}
where for any $k$-scheme $S$ we denote the contravariant functor $\mathrm{Hom}(-,S)$ by $\underline{S}$. Now let $[\mathcal{H}_{Q} \stackrel{u}{\longrightarrow} \mathcal{U}]$ be the universal family. Here $\mathcal{H}_{Q}$ is the pull-back of $\mathcal{H}$ along projection $X \times Q \longrightarrow X$, $\mathcal{U}$ is a $Q$-flat coherent sheaf on $X \times Q$ with Hilbert polynomial $P$ and $u$ is a surjective morphism. Let $p_{12} : T \times X \times Q \longrightarrow T \times X$ be projection, then it is easy to see that precomposing $(\sigma \times 1_{Q})^{*}u$ with $p_{12}^{*}\Phi_{\mathcal{H}}^{-1}$ gives an element of $\mathcal{Q}(T \times Q)$. Applying $\Xi_{T \times Q}$ gives a morphism $\sigma : T \times Q \longrightarrow Q$, our candidate regular action. Note that $\sigma$ depends on the choice of $T$-equivariant structure on $\mathcal{H}$. For any closed point $t \in T$, let $i_{t} : X \hookrightarrow T \times X$ be the inclusion induced by $t \hookrightarrow T$ and consider $\Phi_{\mathcal{H},t} = i_{t}^{*}\Phi_{\mathcal{H}} : t^{*}\mathcal{H} \stackrel{\cong}{\longrightarrow} \mathcal{H}$. Let $p = [\mathcal{H} \stackrel{\rho}{\longrightarrow} \mathcal{F}]$ be a closed point of $Q$. Using the properties of the universal family and the definition of $\sigma$, it is easy to see that the closed point corresponding to $\sigma(t,p) = t \cdot p$ is given by
\begin{equation}
[\mathcal{H} \stackrel{\Phi_{\mathcal{H},t}^{-1}}{\longrightarrow} t^{*}\mathcal{H} \stackrel{t^{*}\rho}{\longrightarrow} t^{*}\mathcal{F}]. \nonumber
\end{equation}
Using the properties of the universal family, a somewhat tedious yet straightforward exercise shows that $\sigma : T \times Q \longrightarrow Q$ satisfies the axioms of an action. Let $R \subset Q$ be the open subscheme with closed points those elements $[\mathcal{H} \stackrel{\rho}{\longrightarrow} \mathcal{F}] \in Q$, where $\mathcal{F}$ is Gieseker semistable and the induced map $V \longrightarrow H^{0}(X,\mathcal{F}(m))$ is an isomorphism. Since the $T$-equivariant structure on $\mathcal{H}$ comes from a $T$-equivariant structure on $\mathcal{O}_{X}(1)$ and a trivial action of $T$ on $V$, $\sigma$ restricts to a regular action on $R$. Let $G = \mathrm{PGL}(V)$, then there is a natural (right) action $\mathcal{Q} \times \underline{G} \Longrightarrow \mathcal{Q}$. This induces a regular (right) action of $G$ on $Q$, which restricts to $R$. The moduli space $M^{ss} = \mathcal{M}_{P}^{ss}$ can be formed as a categorical quotient $\pi : R \longrightarrow M^{ss}$ \cite[Sect.~4.3]{HL}. Consider the diagram
\begin{displaymath}
\xymatrix
{
T \times R \ar[r]^{\sigma} \ar[d]_{1_{T} \times \pi} & R \ar[d]^{\pi} \\
T \times M^{ss} & M^{ss}. 
}
\end{displaymath}
The morphism $\sigma : T \times R \longrightarrow R$ is $G$-equivariant (where we let $G$ act trivially on $T$). Again, this can be seen by using that the $T$-equivariant structure on $\mathcal{H}$ comes from a $T$-equivariant structure on $\mathcal{O}_{X}(1)$ and a trivial action of $T$ on $V$. Consequently, $\pi \circ \sigma$ is $G$-invariant. From the definition of a categorical quotient, we get an induced morphism $\sigma : T \times M^{ss} \longrightarrow M^{ss}$. Again, using the definition of a categorical quotient, we obtain that $\sigma : T \times M^{ss} \longrightarrow M^{ss}$ is a regular action of $T$ on $M^{ss}$ acting on closed points as stated in the proposition. Let $R^{s} \subset R$ be the open subscheme with closed points Gieseker stable sheaves and denote the corresponding geometric quotient by $\varpi : R^{s} \longrightarrow M^{s}$. It is clear the regular action $\sigma : T \times M^{ss} \longrightarrow M^{ss}$ will restrict to $M^{s}$. 
\end{proof}
\begin{proposition} \label{ch. 1, sect. 4, prop. 2}
Let $X$ be a projective toric variety with torus action $\sigma : T \times X \longrightarrow X$. Denote projection to the second factor by $p_{2} : T \times X \longrightarrow X$. Let $\mathcal{O}_{X}(1)$ be an ample line bundle on $X$ and $P$ a choice of Hilbert polynomial. Choose an equivariant structure on $\mathcal{O}_{X}(1)$. Then the closed points of the fixed point locus\footnote{We use the notion of fixed point locus as defined in \cite{Fog}.} of the natural induced regular action of $T$ on $\mathcal{M}_{P}^{s}$ (defined, using the equivariant structure, in Proposition \ref{ch. 1, sect. 4, prop. 1}) are
\begin{equation}
\left(\mathcal{M}_{P}^{s} \right)^{T}_{cl} = \left\{ [\mathcal{E}] \in \mathcal{M}_{P, cl}^{s} \ | \ \sigma^{*}\mathcal{E} \cong p_{2}^{*}\mathcal{E} \right\}. \nonumber
\end{equation}
\end{proposition}
\begin{proof}
From the definition of $\sigma : T \times \mathcal{M}_{P}^{s} \longrightarrow \mathcal{M}_{P}^{s}$, it is clear that the fixed point locus can be characterised as \cite[Rmk.~4]{Fog} 
\begin{equation}
\left(\mathcal{M}_{P}^{s} \right)^{T}_{cl} = \left\{ [\mathcal{E}] \in \mathcal{M}_{P,cl}^{s} \ | \ t^{*}\mathcal{E} \cong \mathcal{E} \ \forall t \in T_{cl} \right\}. \nonumber
\end{equation}
However, we claim that moreover
\begin{equation}
\left(\mathcal{M}_{P}^{s} \right)^{T}_{cl} = \left\{ [\mathcal{E}] \in \mathcal{M}_{P,cl}^{s} \ | \ \sigma^{*}\mathcal{E} \cong p_{2}^{*}\mathcal{E} \right\}. \nonumber
\end{equation}
The inclusion ``$\supset$'' is trivial. Conversely, let $\mathcal{E}$ be Gieseker stable sheaf on $X$ with Hilbert polynomial $P$ such that $t^{*}\mathcal{E} \cong \mathcal{E}$ for all closed points $t \in T$. Since $\mathcal{E}$ is simple\footnote{A simple sheaf $\mathcal{E}$ on a $k$-scheme $X$ of finite type is by definition a coherent sheaf $\mathcal{E}$ on $X$ such that $\mathrm{End}(\mathcal{E}) \cong k$.} \cite[Cor.~1.2.8]{HL}, the result follows from the following proposition\footnote{The following proposition and its proof are due to Tom Bridgeland. The author's original proof of Proposition \ref{ch. 1, sect. 4, prop. 2} was much more complicated and involved Luna's \'Etale Slice Theorem and descent for quasi-coherent sheaves.} applied to $\sigma^{*}\mathcal{E}$ and $p_{2}^{*}\mathcal{E}$.
\end{proof}
\begin{proposition} \label{tom}
Let $X$ be a projective $k$-scheme of finite type and $T$ an algebraic torus. Let $\mathcal{E}$, $\mathcal{F}$ be $T$-flat coherent sheaves on $T \times X$ such that and $\mathcal{E}_{t} \cong \mathcal{F}_{t}$ is simple for all closed points $t \in T$. Then $\mathcal{E} \cong \mathcal{F}$.
\end{proposition}
\begin{proof}
Denote projection to the first component by $p_{1} : T \times X \longrightarrow X$. For any closed point $t \in T$, we denote the inclusion by $i_{t} : t  \hookrightarrow T$. We consider the coherent sheaf $\mathcal{L} = p_{1*} \mathcal{H}{\it{om}}_{\mathcal{O}_{T \times X}}(\mathcal{E},\mathcal{F})$ and will prove it is a line bundle on $T$. There exists a coherent sheaf $\mathcal{N}$ on $T$ and an isomorphism  
\begin{equation*}
p_{1*} \mathcal{H}{\it{om}}_{\mathcal{O}_{T \times X}}(\mathcal{E},\mathcal{F} \otimes_{\mathcal{O}_{T \times X}} p_{1}^{*}(-)) \cong \mathcal{H}{\it{om}}_{\mathcal{O}_{T}}(\mathcal{N},-),
\end{equation*}
of functors $\mathrm{Qco}(T) \longrightarrow \mathrm{Qco}(T)$ \cite[Cor.~7.7.8]{EGA2}. We deduce $\mathcal{L} \cong \mathcal{N}^{\vee}$. However, the construction of $\mathcal{N}$ commutes with base change \cite[Rem.~7.7.9]{EGA2}, so 
\begin{equation*}
(i_{t}^{*}\mathcal{N})^{\vee} \cong \mathrm{Hom}_{X}(\mathcal{E}_{t},\mathcal{F}_{t}) \cong k,
\end{equation*}
for all closed points $t \in T$. Here we use that $\mathcal{E}_{t} \cong \mathcal{F}_{t}$ is simple for all closed points $t \in T$. Consequently, $\mathcal{N}$ and $\mathcal{L}$ are line bundles on $T$ \cite[Exc.~II.5.8]{Har1}. Since $\mathrm{Pic}(T)=0$, we deduce $\mathcal{L} \cong \mathcal{O}_{T}$. Now consider 
\begin{equation*}
H^{0}(T,\mathcal{L}) = \mathrm{Hom}_{T \times X}(\mathcal{E},\mathcal{F}).
\end{equation*}
Since $\mathcal{L} \cong \mathcal{O}_{T}$, there exists a nowhere vanishing section $f \in H^{0}(T,\mathcal{L})$. This section corresponds to a morphism $f : \mathcal{E} \longrightarrow \mathcal{F}$ having the property that $f_{t} : \mathcal{E}_{t} \longrightarrow \mathcal{F}_{t}$ is nonzero for any closed point $t \in T$. Similarly, $\mathcal{L}' = p_{1*} \mathcal{H}{\it{om}}_{\mathcal{O}_{T \times X}}(\mathcal{F},\mathcal{E}) \cong \mathcal{O}_{T}$ and we can take a nowhere vanishing section $f' \in H^{0}(T,\mathcal{L}')$ corresponding to a morphism $f' : \mathcal{F} \longrightarrow \mathcal{E}$. Now consider the composition $g = f' \circ f$. There is a canonical map 
\begin{equation*}
H^{0}(T,\mathcal{O}_{T}) \stackrel{\cong}{\longrightarrow} \mathrm{Hom}_{T \times X}(\mathcal{E},\mathcal{E}),
\end{equation*}
which is an isomorphism by the arguments above. It is easy to see that $g_{t} \neq 0$ for any closed point $t \in T$, from which we deduce that $g$ corresponds to $c \chi \in H^{0}(T,\mathcal{O}_{T})$ for some $c \in k^{*}$ and $\chi \in X(T)$ a character. Therefore $(c^{-1} \chi^{-1} f') \circ f = \mathrm{id}_{\mathcal{E}}$. Similarly, we get a right inverse for $f$, showing $f$ is an isomorphism.
\end{proof}

\noindent Note that if we are in the situation of Propositions \ref{ch. 1, sect. 4, prop. 1} and \ref{ch. 1, sect. 4, prop. 2}, the regular action of $T$ on $\mathcal{M}_{P}^{ss}$, $\mathcal{M}_{P}^{s}$ a priori depends on choice of equivariant structure on $\mathcal{O}_{X}(1)$. However, the set $\left(\mathcal{M}_{P}^{s}\right)^{T}_{cl}$ is independent of this choice and our future constructions will not depend on this choice either. Hence, whenever we are in the situation of these propositions, we assume we fix an arbitrary equivariant structure on $\mathcal{O}_{X}(1)$ and induced torus action on $\mathcal{M}_{P}^{ss}$ without further notice.

\subsection{Equivariant versus Invariant}

Let $G$ be an affine algebraic group acting regularly on a $k$-scheme $X$ of finite type. Denote the action by $\sigma : G \times X \longrightarrow X$ and projection by $p_{2} : G \times X \longrightarrow X$. From Proposition \ref{ch. 1, sect. 4, prop. 2}, we see that it is natural to define a $G$-invariant sheaf on $X$ to be a sheaf of $\mathcal{O}_{X}$-modules $\mathcal{E}$ on $X$ for which there is an isomorphism $\sigma^{*}\mathcal{E} \cong p_{2}^{*}\mathcal{E}$. Clearly, any $G$-equivariant sheaf on $X$ is $G$-invariant, but the converse is not true in general (for an example, see \cite[App.~A]{DOPR}). In the situation of Proposition \ref{ch. 1, sect. 4, prop. 2}, the isomorphism classes of Gieseker stable invariant sheaves on $X$ with Hilbert polynomial $P$ are in bijection with the closed points of $\left( \mathcal{M}_{P}^{s} \right)^{T}$. We have the following results. 
\begin{proposition} \label{ch. 1, sect. 4, prop. 3}
Let $G$ be a connected affine algebraic group acting regularly on a scheme $X$ of finite type over $k$. Let $\mathcal{E}$ be a simple sheaf on $X$. Then $\mathcal{E}$ is $G$-invariant if and only if $\mathcal{E}$ admits a $G$-equivariant structure.
\end{proposition}
\begin{proof}
Denote the action by $\sigma : G \times X \longrightarrow X$ and projection to the second component by $p_{2} : G \times X \longrightarrow X$. Assume $\mathcal{E}$ is a simple sheaf on $X$ and we have an isomorphism $\Phi : \sigma^{*}\mathcal{E} \longrightarrow p_{2}^{*}\mathcal{E}$. We would like $\Phi$ to satisfy the cocycle condition (see Definition \ref{ch. 1, sect. 2, def. 1}). In order to achieve this, we use an argument similar to the proof of \cite[Lem.~7.1]{Dol}. For any closed point $g \in G$, let $i_{g} : X \hookrightarrow G \times X$ be the inclusion induced by $g \hookrightarrow G$ and define $\Phi_{g} = i_{g}^{*}\Phi : g^{*}\mathcal{E} \longrightarrow \mathcal{E}$. By Proposition \ref{lemma}, it is enough to prove $\Phi_{hg} = \Phi_{g} \circ g^{*}\Phi_{h}$, for all closed points $g,h \in G$. By redefining $\Phi$, i.e.~replacing $\Phi$ by $p_{2}^{*}(\Phi_{1}^{-1}) \circ \Phi$, we might just as well assume $\Phi_{1} = 1$. Now define the morphism
\begin{align} 
\begin{split}
&F : G_{cl} \times G_{cl} \longrightarrow \mathrm{Aut}(\mathcal{E}) \cong k^{*}, \\
&F(g,h) = \Phi_{g} \circ g^{*}(\Phi_{h}) \circ \Phi_{hg}^{-1}, \nonumber
\end{split}
\end{align}
where $(-)_{cl}$ means taking the closed points. We know $F(g,1) = F(1,h) = 1$ and we have to prove $F(g,h) = 1$, for all closed points $g,h \in G$. Since $G_{cl}$ is an irreducible algebraic variety over an algebraically closed field $k$ and $F \in \mathcal{O}(G_{cl} \times G_{cl})^{*}$, we can use a theorem by Rosenlicht \cite[Rmk.~7.1]{Dol}, to conclude that $F(g,h) = F_{1}(g)F_{2}(h)$, where $F_{1}, F_{2} \in \mathcal{O}(G_{cl})^{*}$, for all closed points $g,h \in G$. The result now follows immediately.
\end{proof}
\begin{proposition} \label{ch. 1, sect. 4, prop. 4}
Let $G$ be an affine algebraic group acting regularly on a scheme $X$ of finite type over $k$. Let $\mathcal{E}$ be a simple $G$-equivariant sheaf on $X$. Then all $G$-equivariant structures on $\mathcal{E}$ are given by $\mathcal{E} \otimes \mathcal{O}_{X}(\chi)$, where $\mathcal{O}_{X}(\chi)$ is the structure sheaf of $X$ endowed with the $G$-equivariant structure induced by the character $\chi \in X(G)$.
\end{proposition}
\begin{proof}
Let $\Phi, \Psi : \sigma^{*}(\mathcal{E}) \longrightarrow p_{2}^{*}(\mathcal{E})$ be two $G$-equivariant structures on $\mathcal{E}$. Consider the automorphism $\Psi \circ \Phi^{-1} : p_{2}^{*}(\mathcal{E}) \longrightarrow p_{2}^{*}(\mathcal{E})$. For all closed points $g \in G$
\begin{equation} \nonumber
\Psi_{g} \circ \Phi_{g}^{-1} \in \mathrm{Aut}(\mathcal{E}) \cong k^{*}.
\end{equation}
We obtain a morphism of varieties $\chi : G_{cl} \longrightarrow k^{*}$ defined by $\chi(g) = \Psi_{g} \circ \Phi_{g}^{-1}$. In fact, from the fact that $\Phi, \Psi$ satisfy the cocycle condition (see Definition \ref{ch. 1, sect. 2, def. 1}), we see that $\chi$ is a character. The result follows from applying Proposition \ref{lemma}.
\end{proof}

This last proposition suggests we should study the effect of tensoring an equivariant sheaf on a toric variety by an equivariant line bundle. We start with a brief recapitulation of equivariant line bundles and reflexive equivariant sheaves on toric varieties. On a general normal variety $X$, a coherent sheaf $\mathcal{F}$ is said to be reflexive if the natural morphism $\mathcal{F} \longrightarrow \mathcal{F}^{\vee\vee}$ is an isomorphism, where $(-)^{\vee} = \mathcal{H}{\it{om}}(-,\mathcal{O}_{X})$ is the dual. Let $X$ be a nonsingular toric variety defined by a fan $\Delta$ in a lattice $N$ of rank $r$. Take $\tau = 0$ and let $\sigma_{1}, \ldots, \sigma_{l}$ be the cones of dimension $r$. For each $i = 1, \ldots, l$, let $\left(\rho^{(i)}_{1}, \ldots, \rho^{(i)}_{r}\right)$ be the rays of $\sigma_{i}$. The equivariant line bundles on $X$ are precisely the rank 1 reflexive equivariant sheaves on $X$. In general, reflexive equivariant sheaves on $X$ are certain torsion free equivariant sheaves on $X$ and they admit a particularly nice combinatorial description in terms of filtrations associated to the rays of $\Delta$. Denote the collection of rays by $\Delta(1)$. Let $E$ be a nonzero finite-dimensional $k$-vector space. For each ray $\rho \in \Delta(1)$ specify $k$-vector spaces
\begin{equation} \nonumber
\cdots \subset E^{\rho}(\lambda-1) \subset E^{\rho}(\lambda) \subset E^{\rho}(\lambda+1) \subset \cdots,
\end{equation}   
such that there is an integer $A_{\rho}$ with $E^{\rho}(\lambda) = 0$ if $\lambda < A_{\rho}$ and there is an integer $B_{\rho}$ such that $E^{\rho}(\lambda) = E$ if $\lambda \geq B_{\rho}$. There is an obvious notion of morphisms between such collections of filtrations $\{E^{\rho}(\lambda)\}_{\rho \in \Delta(1)}$. Suppose we are given such a collection of filtrations $\{E^{\rho}(\lambda)\}_{\rho \in \Delta(1)}$. From it we obtain a torsion free $\Delta$-family by defining 
\begin{equation} \label{ch. 1, eqn4}
E^{\sigma_{i}}(\lambda_{1}, \ldots, \lambda_{r}) = E^{\rho_{1}^{(i)}}(\lambda_{1}) \cap \cdots \cap E^{\rho_{r}^{(i)}}(\lambda_{r}),
\end{equation}
for all $i = 1, \ldots, l$, $\lambda_{1}, \ldots, \lambda_{r} \in \mathbb{Z}$. Denote the full subcategory of torsion free $\Delta$-families obtained in this way by $\mathcal{R}$. The equivalence of categories of Theorem \ref{ch. 1, sect. 2, thm. 1}, restricts to an equivalence between the the full subcategory of reflexive equivariant sheaves on $X$ and the full subcategory $\mathcal{R}$ (see \cite[Thm.~4.21]{Per1}). This equivalence further restricts to an equivalence between the category of equivariant line bundles on $X$ and the category of filtrations of $E = k$ associated to the rays of $\Delta$ as above. We obtain a canonical isomorphism $\mathrm{Pic}^{T}(X) \cong \mathbb{Z}^{\Delta(1)}$, where $\mathbb{Z}^{\Delta(1)} = \mathbb{Z}^{\# \Delta(1)}$. In particular, if $\Delta(1) = ( \rho_{1}, \ldots, \rho_{N} )$, then the integers $\vec{k} = (k_{1}, \ldots, k_{N}) \in \mathbb{Z}^{\Delta(1)}$ correspond to the filtrations $\{L_{\vec{k}}^{\rho}(\lambda)\}_{\rho \in \Delta(1)}$ defined by\footnote{Do not be confused by the unfortunate clash of notation: $k$ is the ground field and the $k_{j}$ are integers.}
\begin{equation} 
L_{\vec{k}}^{\rho_{j}}(\lambda) = \left\{ \begin{array}{cc} k & \mathrm{if} \ \lambda \geq -k_{j} \\ 0 & \mathrm{if} \ \lambda < - k_{j}, \end{array} \right. \nonumber
\end{equation}
for all $j = 1, \ldots, N$. Denote the corresponding equivariant line bundle by $\mathcal{L}_{\vec{k}}$. Note that the first Chern class of $\mathcal{L}_{\vec{k}}$ is given by $c_{1}(\mathcal{L}_{\vec{k}}) = \sum_{j} k_{j} V(\rho_{j})$ (Corollary \ref{ch. 1, sect. 3, cor. K}), so as a line bundle $\mathcal{L}_{\vec{k}} \cong \mathcal{O}_{X}(\sum_{j} k_{j} V(\rho_{j}))$. Finally, when we consider $\vec{k}$, $\{L_{\vec{k}}^{\rho}(\lambda)\}_{\rho \in \Delta(1)}$, $\mathcal{L}_{\vec{k}}$ as above, then the corresponding torsion free $\Delta$-family is given by (equation (\ref{ch. 1, eqn4}))
\begin{equation} 
L_{\vec{k}}^{\sigma_{i}}(\lambda_{1}, \ldots, \lambda_{r}) = \left\{ \begin{array}{cc} k & \mathrm{if} \ \lambda_{1} \geq -k_{1}^{(i)}, \ldots, \lambda_{r} \geq -k_{r}^{(i)} \\ 0 & \mathrm{otherwise}, \end{array} \right. \nonumber
\end{equation}
for all $i = 1, \ldots, l$, where $\sigma_{i}$ has rays $\left(\rho^{(i)}_{1}, \ldots, \rho^{(i)}_{r}\right)$ and we denote the corresponding integers where the filtrations $L_{\vec{k}}^{\rho^{(i)}_{1}}(\lambda), \ldots, L_{\vec{k}}^{\rho^{(i)}_{r}}(\lambda)$ jump by $-k_{1}^{(i)}, \ldots, -k_{r}^{(i)}$.
\begin{proposition} \label{ch. 1, sect. 4, prop. 5}
Let $X$ be a nonsingular toric variety with fan $\Delta$ in a lattice $N$ of rank $r$. Let $\tau_{1}, \ldots, \tau_{a}$ be some cones of $\Delta$ of dimension $s$. Let $\sigma_{1}, \ldots, \sigma_{l}$ be all cones of $\Delta$ of maximal dimension having a cone $\tau_{\alpha}$ as a face. For each $i = 1, \ldots, l$, let $\left(\rho^{(i)}_{1}, \ldots, \rho^{(i)}_{r}\right)$ be the rays of $\sigma_{i}$. Let $\mathcal{E}$ be a pure equivariant sheaf on $X$ with support $V(\tau_{1}) \cup \cdots \cup V(\tau_{a})$ and corresponding pure $\Delta$-family $\hat{E}^{\Delta}$. Consider the equivariant line bundle $\mathcal{L}_{\vec{k}}$ for some $\vec{k} \in \mathbb{Z}^{\Delta(1)}$. Then $\mathcal{F} = \mathcal{E} \otimes \mathcal{L}_{\vec{k}}$ is a pure equivariant sheaf on $X$ with support $V(\tau_{1}) \cup \cdots \cup V({\tau_{a}})$ and its pure $\Delta$-family $\hat{F}^{\Delta}$ is given by
\begin{align} 
\begin{split}
F^{\sigma_{i}}(\lambda_{1}, \ldots, \lambda_{r}) &= E^{\sigma_{i}}(\lambda_{1}+k_{1}^{(i)}, \ldots, \lambda_{r}+k_{r}^{(i)}), \ \forall i = 1, \ldots, l \\
\chi^{\sigma_{i}}_{F, n}(\lambda_{1}, \ldots, \lambda_{r}) &= \chi^{\sigma_{i}}_{E, n}(\lambda_{1}+k_{1}^{(i)}, \ldots, \lambda_{r}+k_{r}^{(i)}), \ \forall i = 1, \ldots, l, \ \forall n = 1, \ldots, r. \nonumber
\end{split}
\end{align}
\end{proposition} 
\begin{proof}
One can compute the $M$-grading of $\Gamma(U_{\sigma_{i}}, \mathcal{F}) \cong \Gamma(U_{\sigma_{i}}, \mathcal{E}) \otimes_{k[S_{\sigma_{i}}]} \Gamma(U_{\sigma_{i}}, \mathcal{L}_{\vec{k}})$ along the same lines as in the proof of Proposition \ref{ch. 1, sect. 2, prop. 5}. The result easily follows.
\end{proof}

\subsection{Combinatorial Description of the Fixed Point Loci $\left( \mathcal{M}_{P}^{s} \right)^{T}$}

We are now ready to prove the theorem stated in the introduction (Theorem \ref{ch. 1, sect. 1, thm. 1}). An analogous result to Theorem \ref{ch. 1, sect. 1, thm. 1} turns out to hold without any assumption on the Hilbert polynomial if we assume the existence of equivariant line bundles matching Gieseker and GIT stability. Therefore, we will first prove a general combinatorial expression for the fixed point locus of any moduli space of Gieseker stable sheaves on a nonsingular projective toric variety making this assumption (Theorem \ref{ch. 1, sect. 4, thm. 1}). Theorem \ref{ch. 1, sect. 1, thm. 1} then follows as a trivial corollary of this result by combining with Theorem \ref{ch. 1, sect. 3, thm. 4}. Let $X$ be a nonsingular projective toric variety defined by a fan $\Delta$ in a lattice $N$ of rank $r$. Let $\mathcal{O}_{X}(1)$ be an ample line bundle on $X$ and let $P$ be a choice of Hilbert polynomial. The degree $d$ of $P$ is the dimension $d$ of any coherent sheaf on $X$ with Hilbert polynomial $P$. Let $s = r - d$ and let $\tau_{1}, \ldots, \tau_{a}$ be all cones of $\Delta$ of dimension $s$. For any $i_{1} < \cdots < i_{\alpha} \in \{1, \ldots, a\}$, we have defined $\mathcal{X}^{\tau_{i_{1}}, \ldots, \tau_{i_{\alpha}}}_{P} \subset \mathcal{X}^{\tau_{i_{1}}, \ldots, \tau_{i_{\alpha}}}$ to be the subset of all characteristic functions with associated Hilbert polynomial $P$ (see Proposition \ref{ch. 1, sect. 3, prop. 7}). Assume that for any $\vec{\chi} \in \mathcal{X}^{\tau_{i_{1}}, \ldots, \tau_{i_{\alpha}}}_{P}$, we can pick an equivariant line bundle matching Gieseker and GIT stability (e.g.~for $P$ of degree $\mathrm{dim}(X)$ this can always be done by Theorem \ref{ch. 1, sect. 3, thm. 4}). For any $\vec{\chi} \in \mathcal{X}^{\tau_{i_{1}}, \ldots, \tau_{i_{\alpha}}}_{P}$ the obvious forgetful natural transformations $\underline{\mathcal{M}}_{\vec{\chi}}^{\tau_{i_{1}}, \ldots, \tau_{i_{\alpha}},ss} \Longrightarrow \underline{\mathcal{M}}_{P}^{ss}$, $\underline{\mathcal{M}}_{\vec{\chi}}^{\tau_{i_{1}}, \ldots, \tau_{i_{\alpha}},s} \Longrightarrow \underline{\mathcal{M}}_{P}^{s}$ induce morphisms $\mathcal{M}_{\vec{\chi}}^{\tau_{i_{1}}, \ldots, \tau_{i_{\alpha}},ss} \longrightarrow \mathcal{M}_{P}^{ss}$, $\mathcal{M}_{\vec{\chi}}^{\tau_{i_{1}}, \ldots, \tau_{i_{\alpha}},s} \longrightarrow \mathcal{M}_{P}^{s}$ (by Theorem \ref{ch. 1, sect. 3, thm. 3}). We obtain morphisms
\begin{align}
\coprod_{\alpha=1}^{a} \coprod_{i_{1} < \cdots < i_{\alpha} \in \{1, \ldots, a\}} \coprod_{\vec{\chi} \in \mathcal{X}_{P}^{\tau_{i_{1}}, \ldots, \tau_{i_{\alpha}}}} \mathcal{M}_{\vec{\chi}}^{\tau_{i_{1}}, \ldots, \tau_{i_{\alpha}},ss} &\longrightarrow \mathcal{M}_{P}^{ss}, \nonumber \\
\coprod_{\alpha=1}^{a} \coprod_{i_{1} < \cdots < i_{\alpha} \in \{1, \ldots, a\}} \coprod_{\vec{\chi} \in \mathcal{X}_{P}^{\tau_{i_{1}}, \ldots, \tau_{i_{\alpha}}}} \mathcal{M}_{\vec{\chi}}^{\tau_{i_{1}}, \ldots, \tau_{i_{\alpha}},s} &\longrightarrow \mathcal{M}_{P}^{s}, \nonumber
\end{align}
where the second morphism, on closed points, is just the map forgetting the equivariant structure. Consequently, we could expect the second morphism to factor through an isomorphism onto the fixed point locus $\left(\mathcal{M}_{P}^{s}\right)^{T}$. Indeed it maps to the fixed point locus on closed points 
\begin{equation}
\coprod_{\alpha=1}^{a} \coprod_{i_{1} < \cdots < i_{\alpha} \in \{1, \ldots, a\}} \coprod_{\vec{\chi} \in \mathcal{X}_{P}^{\tau_{i_{1}}, \ldots, \tau_{i_{\alpha}}}} \mathcal{M}_{\vec{\chi},cl}^{\tau_{i_{1}}, \ldots, \tau_{i_{\alpha}},s} \longrightarrow \left(\mathcal{M}_{P}^{s}\right)_{cl}^{T}, \nonumber 
\end{equation}
and this map is surjective (Propositions \ref{ch. 1, sect. 4, prop. 2} and \ref{ch. 1, sect. 4, prop. 3}). However, it is not injective (Proposition \ref{ch. 1, sect. 4, prop. 4}). Indeed, if $\mathcal{E}$ is an invariant simple sheaf on $X$, then it admits an equivariant structure (fix one) and all equivariant structures are given by $\mathcal{E} \otimes \mathcal{O}_{X}(\chi)$, $\chi \in M$. 
So before we can expect the morphism to factor through an isomorphism onto the fixed point locus $\left(\mathcal{M}_{P}^{s}\right)^{T}$, we first need to make a choice of equivariant structure for each $\mathcal{E}$. In view of Proposition \ref{ch. 1, sect. 4, prop. 5}, this might be achieved as follows. Let $\sigma_{1}, \ldots, \sigma_{l}$ be all cones of maximal dimension. Let $\alpha = 1, \ldots, a$ and $i_{1} < \cdots < i_{\alpha} \in \{1, \ldots, a\}$. Let $\sigma_{n}$ be a cone among $\sigma_{1}, \ldots, \sigma_{l}$ having at least one of $\tau_{i_{1}}, \ldots, \tau_{i_{\alpha}}$ as a face. For definiteness, we choose $\sigma_{n}$ the cone among $\sigma_{1}, \ldots, \sigma_{l}$ with this property and smallest index $n$ (though any cone will do). Let $\vec{\chi} \in \mathcal{X}_{P}^{\tau_{i_{1}}, \ldots, \tau_{i_{\alpha}}}$, then there are integers $A_{1}^{(n)}, \ldots, A_{r}^{(n)}$ such that $\chi^{\sigma_{n}}(\lambda_{1}, \ldots, \lambda_{r}) = 0$ unless $\lambda_{1} \geq A_{1}^{(n)}$, $\ldots$, $\lambda_{r} \geq A_{r}^{(n)}$ (see section 2). Assume $A_{1}^{(n)}, \ldots, A_{r}^{(n)}$ are chosen maximally with this property. We define $\vec{\chi}$ to be framed if $A_{1}^{(n)} = \cdots = A_{r}^{(n)} = 0$ (any other choice of integers will do too). We denote the subset of framed characteristic functions of $\mathcal{X}_{P}^{\tau_{i_{1}}, \ldots, \tau_{i_{\alpha}}}$ by $\left( \mathcal{X}_{P}^{\tau_{i_{1}}, \ldots, \tau_{i_{\alpha}}} \right)^{fr}$. We get a morphism 
\begin{equation}
F : \coprod_{\alpha=1}^{a} \coprod_{i_{1} < \cdots < i_{\alpha} \in \{1, \ldots, a\}} \coprod_{\vec{\chi} \in \left( \mathcal{X}_{P}^{\tau_{i_{1}}, \ldots, \tau_{i_{\alpha}}} \right)^{fr}} \mathcal{M}_{\vec{\chi}}^{\tau_{i_{1}}, \ldots, \tau_{i_{\alpha}},s} \longrightarrow \mathcal{M}_{P}^{s}. \label{ch. 1, eqnF}
\end{equation}

\noindent \emph{Claim.} The map induced by $F$ on closed points maps bijectively onto $\left(\mathcal{M}_{P}^{s}\right)_{cl}^{T}$. 

\noindent \emph{Proof of Claim.} This can be seen as follows. We need to characterise all equivariant line bundles $\mathcal{O}_{X}(\chi)$, $\chi \in X(T)$ (introduced in Proposition \ref{ch. 1, sect. 4, prop. 4}). By \cite[Cor.~7.1, Thm.~7.2]{Dol}, they are precisely the elements of the kernel of the forgetful map in the short exact sequence
\begin{equation} \nonumber
0 \longrightarrow M \longrightarrow \mathrm{Pic}^{T}(X) \longrightarrow \mathrm{Pic}(X) \longrightarrow 0.
\end{equation} 
Suppose $\mathcal{L}_{\vec{k}}$, $\vec{k} \in \mathbb{Z}^{\Delta(1)}$ is an equivariant line bundle (notation as in subsection 4.2). Then its underlying line bundle is trivial if and only if $\sum_{j=1}^{N} k_{j} V(\rho_{j}) = 0$ in the Chow group $A_{r-1}(X)$. The Chow group $A_{r-1}(X)$ is the free abelian group on $V(\rho_{1}), \ldots, V(\rho_{N})$ modulo the following relations \cite[Sect.~5.2]{Ful}
\begin{equation} \nonumber
\sum_{j=1}^{N} \langle u, n(\rho_{j}) \rangle V(\rho_{j}) = 0, \ \forall u \in M.
\end{equation}
Let $\sigma_{n}$ be any cone of maximal dimension and take $m\left(\rho_{1}^{(n)}\right), \ldots, m\left(\rho_{r}^{(n)}\right)$ as a basis for $M$. From this we see that for arbitrary $k_{1}^{(n)}, \ldots, k_{r}^{(n)} \in \mathbb{Z}$, there are unique other $k_{1}^{(i)}, \ldots, k_{r}^{(i)}$, for all $i =1, \ldots, l$, $i \neq n$ such that $\sum_{j=1}^{N} k_{j} V(\rho_{j}) = 0$. In particular, if $\mathcal{L}_{\vec{k}}$, $\vec{k} \in \mathbb{Z}^{\Delta(1)}$ is an equivariant line bundle with underlying line bundle trivial and $k_{1}^{(n)} = \cdots = k_{r}^{(n)} = 0$, then also $k_{1}^{(i)} = \cdots = k_{r}^{(i)} = 0$ for all $i = 1, \ldots, l$, $i \neq n$. Now note that for any two distinct sequences $i_{1} < \cdots < i_{\alpha}$, $j_{1} < \cdots < j_{\beta} \in \{1, \ldots, a\}$, we have $\mathcal{X}^{\tau_{i_{1}}, \ldots, \tau_{i_{\alpha}}} \cap \mathcal{X}^{\tau_{j_{1}}, \ldots, \tau_{j_{\beta}}} = \varnothing$. Using Propositions \ref{ch. 1, sect. 4, prop. 2}, \ref{ch. 1, sect. 4, prop. 3}, \ref{ch. 1, sect. 4, prop. 4} and \ref{ch. 1, sect. 4, prop. 5}, the claim follows. $\hfill \Box$ 

We note that the above claim crucially depends on Propositions \ref{ch. 1, sect. 4, prop. 2}, \ref{ch. 1, sect. 4, prop. 3}, \ref{ch. 1, sect. 4, prop. 4}, which are about \emph{simple} sheaves. This is one of the main reasons we have to focus attention on \emph{Gieseker stable} sheaves only in this section. The above claim provides good evidence that the morphism $F$ of equation (\ref{ch. 1, eqnF}) indeed factors through an isomorphism onto the fixed point locus $\left(\mathcal{M}_{P}^{s}\right)^{T}$. We will prove this using the following two technical results. 
\begin{proposition} \label{ch. 1, sect. 4, prop. 6}
Let $X,Y$ be schemes of finite type and separated over $k$. Let $f : X \longrightarrow Y$ be a morphism and $\iota : Y' \hookrightarrow Y$ a closed immersion. Assume for any local artinian $k$-algebra $A$ with residue field $k$, the map $f \circ -$ factors bijectively 
\begin{displaymath}
\xymatrix
{
\mathrm{Hom}(A,X) \ar@{-->}_{\cong}[dr] \ar[r]^{f \circ -} & \mathrm{Hom}(A,Y) \\
 & \mathrm{Hom}(A,Y'). \ar@^{(->}[u]_{\iota \circ -} 
}
\end{displaymath}
Then $f$ factors isomorphically and uniquely onto $Y'$
\begin{displaymath}
\xymatrix
{
X \ar[r]^{f} \ar@{-->}[dr]_{\cong} & Y \\
& Y'. \ar@^{(->}[u]_{\iota}
}
\end{displaymath}
\end{proposition}
\begin{proof}
First we prove the proposition while assuming $f \circ -$ factors (not necessarily as a bijection) and conclude $f$ factors through $Y'$ (not necessarily as an isomorphism). It is clear that if $f$ factors, then it factors uniquely, because $\iota$ is a closed immersion. By taking an appropriate open affine cover we get the following diagram of finitely generated $k$-algebras
\begin{displaymath}
\xymatrix
{
R &  \ar[l]_{f^{\#}} S \ar[d]^{\iota^{\#}} \\
& S/I, 
}
\end{displaymath}
where $I \subset S$ is some ideal. It is enough to prove $f^{\#}(I) = 0$. Suppose this is not the case. Then there is some $0 \neq s \in I$ such that $f^{\#}(s) = r \neq 0$. There exists a maximal ideal $\mathfrak{m} \subset R$ such that $r$ is not mapped to zero by the localisation map $R \longrightarrow R_{\mathfrak{m}}$ (use \cite[Exc.~4.10]{AM}). Moreover, there is an integer $n > 0$ such that the canonical map $R_{\mathfrak{m}} \longrightarrow R_{\mathfrak{m}} / (\mathfrak{m}R_{\mathfrak{m}})^{n}$ maps $r/1 \in R_{\mathfrak{m}}$ to a nonzero element (this follows from a corollary of Krull's Theorem \cite[Cor.~10.19]{AM}). Now $R_{\mathfrak{m}} / (\mathfrak{m}R_{\mathfrak{m}})^{n}$ is a local artinian $k$-algebra with residue field $k$ \cite[Prop.~8.6]{AM}. We obtain a $k$-algebra homomorphism $R \longrightarrow R_{\mathfrak{m}} / (\mathfrak{m}R_{\mathfrak{m}})^{n}$ such that precomposition with $f^{\#}$ maps $s$ to a nonzero element. But by assumption, this composition has to factor through $S/I$ and since $s \in I$ this is a contradiction. 

The second part of the proof consists of proving the statement of the proposition in the case $Y' = Y$. This is established in \cite[Lem.~3.2]{Fog}. 
\end{proof}
\begin{proposition} \label{ch. 1, sect. 4, prop. 7}
Let $X$ be a nonsingular projective toric variety defined by a fan $\Delta$ in a lattice of rank $r$. Let $\mathcal{O}_{X}(1)$ be an ample line bundle on $X$, let $P$ be a choice of Hilbert polynomial of degree $d$ and let $\tau_{1}, \ldots, \tau_{a}$ be all cones of $\Delta$ of dimension $s = r - d$. Let $i_{1} < \cdots < i_{\alpha} \in \{1, \ldots, a\}$ and assume for any $\vec{\chi} \in \mathcal{X}^{\tau_{i_{1}}, \ldots, \tau_{i_{\alpha}}}_{P}$, we can pick an equivariant line bundle matching Gieseker and GIT stability. Then for any local artinian $k$-algebra $A$ with residue field $k$ the moduli functors and their moduli spaces induce bijections
\begin{align}
\begin{split}
\underline{\mathcal{M}}_{\vec{\chi}}^{\tau_{i_{1}}, \ldots, \tau_{i_{\alpha}}, s}(A) &\stackrel{\cong}{\longrightarrow} \mathrm{Hom}(A, \mathcal{M}_{\vec{\chi}}^{\tau_{i_{1}}, \ldots, \tau_{i_{\alpha}}, s}), \\
\underline{\mathcal{M}}_{P}^{s}(A) &\stackrel{\cong}{\longrightarrow} \mathrm{Hom}(A, \mathcal{M}_{P}^{s}). \nonumber
\end{split}
\end{align}
\end{proposition}
\begin{proof}
Let us prove the first bijection first. Denote the moduli functor $\underline{\mathcal{M}}_{\vec{\chi}}^{\tau_{i_{1}}, \ldots, \tau_{i_{\alpha}}, s}$ by $\mathcal{M}$. Recall that $M = \mathcal{M}_{\vec{\chi}}^{\tau_{i_{1}}, \ldots, \tau_{i_{\alpha}}, s}$ was formed by considering the regular action of the reductive algebraic group $G$ on $N = \mathcal{N}_{\vec{\chi}}^{\tau_{i_{1}}, \ldots, \tau_{i_{\alpha}}, s}$ (where GIT stability is defined by the $G$-equivariant line bundle $L = \mathcal{L}_{\vec{\chi}}^{\tau_{i_{1}}, \ldots, \tau_{i_{\alpha}}}$) and the induced geometric quotient $\varpi : N \longrightarrow M = N / G$ (see subsection 3.3). Here $G$ is the closed subgroup of elements of determinant 1 of an algebraic group of the form
\begin{equation}
H = \prod_{i=1}^{n} \mathrm{GL}(n_{i},k). \nonumber
\end{equation}
We would like to use a corollary of Luna's \'Etale Slice Theorem to conclude that $\varpi$ is a principal $G$-bundle \cite[Cor.~4.2.13]{HL}. Unfortunately, the stabiliser of a closed point of $N$ is the group $\mu_{p}$ of $p$th roots of unity, where $p = \sum_{i=1}^{n} n_{i}$, hence not trivial. Consider the diagonal closed subgroup $\mathbb{G}_{m} \lhd H$ and define $\tilde{G} = H / \mathbb{G}_{m}$. There is a natural regular action of the reductive algebraic group $\tilde{G}$ on $\mathcal{N}_{\vec{\chi}}^{\tau_{i_{1}}, \ldots, \tau_{i_{\alpha}}}$ giving rise to the same orbits as $G$. The natural morphism $G \longrightarrow \tilde{G}$ of algebraic groups gives rise to an isomorphism $G / \mu_{p} \cong \tilde{G}$ of algebraic groups. If we fix a $G$-equivariant line bundle $L$, then it is easy to see that $L^{\otimes p}$ admits a $\tilde{G}$-equivariant structure. For both choices, the sets of GIT semistable respectively stable points will be the same and the categorical and geometric quotients will be the same. In particular $M = N / G = N / \tilde{G}$. The stabiliser in $\tilde{G}$ of any GIT stable closed point of $N$ is trivial. Consequently, $\varpi : N \longrightarrow M$ is a principal $\tilde{G}$-bundle, i.e.~there is an \'etale surjective morphism $\pi : Y \longrightarrow M$ and a $\tilde{G}$-equivariant isomorphism  $\psi : \tilde{G} \times Y \longrightarrow N \times_{M} Y$, such that the following diagram commutes
\begin{displaymath}
\xymatrix
{
\tilde{G} \times Y \ar[r]^{\psi} \ar[d] & N \times_{M} Y \ar[dl] \\
Y. &
}
\end{displaymath}
Let $P = [\mathcal{E}] \in M$ be a closed point and let $Q \in Y$ be a closed point such that $\pi(Q) = P$. Let $A$ be any local artinian $k$-algebra with residue field $k$, let $\mathrm{Hom}(A,M)_{P}$ be the set of morphisms $A \longrightarrow M$ where the point is mapped to $P$ and let $\mathrm{Hom}(A,Y)_{Q}$ be the set of morphisms $A \longrightarrow Y$ where the point is mapped to $Q$. Using induction on the length\footnote{Note that for any local artinian $k$-algebra $(A',\mathfrak{m}')$ of length $l \geq 2$ with residue field $k$ there is a surjective local $k$-algebra homomorphism $\sigma : A' \longrightarrow A$, where $A$ is a local artinian $k$-algebra of length $<l$ with residue field $k$ and kernel $J$ a principal ideal satisfying $\mathfrak{m}'J = 0$. Such surjective morphisms are called small extensions \cite{Sch}.} of $A$ and using the definition of formally \'etale \cite[Def.~17.1.1]{EGA3}, it is easy to see that composition with $\pi$ gives a bijection
\begin{equation}
\mathrm{Hom}(A,Y)_{Q} \stackrel{\cong}{\longrightarrow} \mathrm{Hom}(A,M)_{P}. \nonumber
\end{equation}
As an aside, we note that this implies in particular that the Zariski tangent spaces at $P$ and $Q$ are isomorphic $T_{Q} Y \cong T_{P} M$, by taking $A$ the ring of dual numbers. We have $\mathcal{M}(A) \cong \mathrm{Hom}(A,N) / \mathrm{Hom}(A,G) = \mathrm{Hom}(A,N) / \mathrm{Hom}(A,\tilde{G})$. The first isomorphism follows from the definition of $\mathcal{M}$ (see proof of Theorem \ref{ch. 1, sect. 3, thm. 2}). The second equality can be deduced from the fact that the morphism $G \longrightarrow \tilde{G}$ is \'etale and surjective on closed points. Using these facts, together with the isomorphism $\psi$, we obtain a natural injection $\mathrm{Hom}(A,M)_{P} \hookrightarrow \mathcal{M}(A)$ such that the following diagram commutes
\begin{displaymath}
\xymatrix
{
\mathcal{M}(A) \ar[r] & \mathrm{Hom}(A,M) \\
& \mathrm{Hom}(A,M)_{P}. \ar@^{(-->}[ul] \ar@^{(->}[u]_{\subset}
}
\end{displaymath}
Let $\mathcal{M}(A)_{P}$ be the image of the injection $\mathrm{Hom}(A,M)_{P} \hookrightarrow \mathcal{M}(A)$. Consider the natural morphism $\iota : \mathrm{Spec}(k) \hookrightarrow \mathrm{Spec}(A)$. It is easy to see that $\mathcal{M}(A)_{P}$ is the set of equivariant isomorphism classes of $A$-flat equivariant coherent sheaves $\mathcal{F}$ on $X \times A$ such that there exists an equivariant isomorphism $(1_{X} \times \iota)^{*} \mathcal{F} \cong \mathcal{E}$. We obtain a natural bijection
\begin{equation}
\mathcal{M}(A)_{P} = \{ [\mathcal{F}] \in \mathcal{M}(A) \ | \ (1_{X} \times \iota)^{*} \mathcal{F} \cong \mathcal{E} \} \stackrel{\cong}{\longrightarrow} \mathrm{Hom}(A, \mathcal{M})_{P}. \nonumber
\end{equation}
Taking a union over all closed points $P$ gives the required bijection. The second bijection of the proposition can be proved entirely analogously. For the definition of the moduli functor and moduli space in this case, we refer to \cite[Ch.~4]{HL}. 
\end{proof}

We can now formulate and prove the following theorem.
\begin{theorem} \label{ch. 1, sect. 4, thm. 1}
Let $X$ be a nonsingular projective toric variety defined by a fan $\Delta$ in a lattice $N$ of rank $r$. Let $\mathcal{O}_{X}(1)$ be an ample line bundle on $X$, let $P$ be a choice of Hilbert polynomial of degree $d$ and let $\tau_{1}, \ldots, \tau_{a}$ be all cones of $\Delta$ of dimension $s = r - d$. Assume for any $i_{1} < \cdots < i_{\alpha} \in \{1, \ldots, a\}$ and $\vec{\chi} \in \mathcal{X}^{\tau_{i_{1}}, \ldots, \tau_{i_{\alpha}}}_{P}$, we can pick an equivariant line bundle matching Gieseker and GIT stability. Then there is a natural isomorphism of quasi-projective schemes of finite type over $k$ 
\begin{equation}
\coprod_{\alpha=1}^{a} \coprod_{i_{1} < \cdots < i_{\alpha} \in \{1, \ldots, a\}} \coprod_{\vec{\chi} \in \left( \mathcal{X}_{P}^{\tau_{i_{1}}, \ldots, \tau_{i_{\alpha}}} \right)^{fr}} \mathcal{M}_{\vec{\chi}}^{\tau_{i_{1}}, \ldots, \tau_{i_{\alpha}},s} \cong \left( \mathcal{M}_{P}^{s} \right)^{T}. \label{ch. 1, eqniso}
\end{equation}
\end{theorem}
\begin{proof}
Consider the morphism $F$ of equation (\ref{ch. 1, eqnF}). We start by noting that there are only finitely many characteristic functions $\vec{\chi}$ in the disjoint union of the left hand side of equation (\ref{ch. 1, eqniso}) for which $\mathcal{M}_{\vec{\chi}}^{\tau_{i_{1}}, \ldots, \tau_{i_{\alpha}},s} \neq \varnothing$. This follows from the fact that the morphism $F_{cl}$ on closed points is bijective and the disjoint union is over a countable set. As a consequence, the left hand side of (\ref{ch. 1, eqniso}) is a quasi-projective $k$-scheme of finite type over $k$ (see Theorem \ref{ch. 1, sect. 3, thm. 3}). We now want to apply Proposition \ref{ch. 1, sect. 4, prop. 6} to the morphism $F$ of equation (\ref{ch. 1, eqnF}) and the closed subscheme $\iota : \left( \mathcal{M}_{P}^{s} \right)^{T} \hookrightarrow \mathcal{M}_{P}^{s}$. We proceed by induction on the length of local artinian $k$-algebras with residue field $k$. For length 1 (i.e.~$A \cong k$), the hypothesis of Proposition \ref{ch. 1, sect. 4, prop. 6} is satisfied. This is the content of the claim at the beginning of this subsection. Assume we have proved the hypothesis of Proposition \ref{ch. 1, sect. 4, prop. 6} for all lengths $1, \ldots, l$ and let $A'$ be a local artinian $k$-algebra of length $l + 1$ with residue field $k$. Then it fits in a small extension $0 \longrightarrow J \longrightarrow A' \stackrel{\sigma}{\longrightarrow} A \longrightarrow 0$, where $A$ is a local artinian $k$-algebra of length $\leq l$ with residue field $k$. Using \cite[Thm.~2.3]{Fog}, one can show the image of $\mathrm{Hom}(A',\left( \mathcal{M}_{P}^{s} \right)^{T})$ in $\mathrm{Hom}(A',\mathcal{M}_{P}^{s})$ is $\mathrm{Hom}(A',\mathcal{M}_{P}^{s})^{T_{cl}}$. Define abbreviations 
\begin{align}
\begin{split}
&\underline{\mathcal{M}} = \coprod_{\alpha=1}^{a} \coprod_{i_{1} < \cdots < i_{\alpha} \in \{1, \ldots, a\}} \coprod_{\vec{\chi} \in \left( \mathcal{X}_{P}^{\tau_{i_{1}}, \ldots, \tau_{i_{\alpha}}} \right)^{fr}} \underline{\mathcal{M}}_{\vec{\chi}}^{\tau_{i_{1}}, \ldots, \tau_{i_{\alpha}},s}, \\
&\underline{\mathcal{N}} = \underline{\mathcal{M}}_{P}^{s}. \nonumber
\end{split}
\end{align}
Using Proposition \ref{ch. 1, sect. 4, prop. 7}, it is enough to prove that the map $\underline{\mathcal{M}}(A') \longrightarrow \underline{\mathcal{N}}(A')$ maps bijectively onto the fixed point locus $\underline{\mathcal{N}}(A')^{T_{cl}}$. (Note that $T_{cl}$ act naturally on the set $\underline{\mathcal{N}}(A')$. We will drop the subscript $cl$ referring to closed points from now on.) By the induction hypothesis, we know $\underline{\mathcal{M}}(A) \longrightarrow \underline{\mathcal{N}}(A)$ maps bijectively onto $\underline{\mathcal{N}}(A)^{T}$. Before we continue, we need to study the deformations and obstructions associated to the moduli functors $\underline{\mathcal{M}}$, $\underline{\mathcal{N}}$. 

In general, let $\mathcal{E}_{0}$ be a simple coherent sheaf on $X$ and $\mathcal{F}_{0}$ a simple equivariant coherent sheaf on $X$. Let $Artin/k$ be the category of local artinian $k$-algebras with residue field $k$. Consider the deformation functor $\mathcal{D}_{\mathcal{E}_{0}} : Artin/k \longrightarrow Sets$, where $\mathcal{D}_{\mathcal{E}_{0}}(A)$ is defined to be the set of isomorphism classes of $A$-flat coherent sheaves on $X \times A$ such that $\mathcal{F} \otimes_{k} A \cong \mathcal{E}_{0}$. Similarly, we define the deformation functor $\mathcal{D}_{\mathcal{F}_{0}}^{eq} : Artin/k \longrightarrow Sets$, where $\mathcal{D}_{\mathcal{F}_{0}}^{eq}(A)$ is defined to be the set of equivariant isomorphism classes of $A$-flat equivariant coherent sheaves on $X \times A$ such that $\mathcal{F} \otimes_{k} A \cong \mathcal{F}_{0}$ (equivariant isomorphism). In our setting, we have a small extension $0 \longrightarrow J \longrightarrow A' \stackrel{\sigma}{\longrightarrow} A \longrightarrow 0$. We now consider the maps $\mathcal{D}_{\mathcal{E}_{0}}(\sigma) : \mathcal{D}_{\mathcal{E}_{0}}(A') \longrightarrow \mathcal{D}_{\mathcal{E}_{0}}(A)$, $\mathcal{D}_{\mathcal{F}_{0}}^{eq}(\sigma) : \mathcal{D}_{\mathcal{F}_{0}}^{eq}(A') \longrightarrow \mathcal{D}_{\mathcal{F}_{0}}^{eq}(A)$. There is a natural map 
\begin{equation}
\mathfrak{o}(\sigma) : \mathcal{D}_{\mathcal{E}_{0}}(A) \longrightarrow \mathrm{Ext}^{2}(\mathcal{E}_{0}, \mathcal{E}_{0}) \otimes_{k} J, \nonumber
\end{equation}
called the obstruction map, such that $\mathfrak{o}(\sigma)^{-1}(0) = \mathrm{im}(\mathcal{D}_{\mathcal{E}_{0}}(\sigma))$. The construction of this map can be found in \cite[Sect.~2]{Art}. Moreover, for any $[\mathcal{F}] \in \mathrm{im}(\mathcal{D}_{\mathcal{E}_{0}}(\sigma))$, the fibre $\mathcal{D}_{\mathcal{E}_{0}}(\sigma)^{-1}([\mathcal{F}])$ is naturally an $\mathrm{Ext}^{1}(\mathcal{E}_{0}, \mathcal{E}_{0}) \otimes_{k} J$-torsor. This can be seen by noting that Proposition \ref{ch. 1, sect. 4, prop. 7} implies $\mathcal{D}_{\mathcal{E}_{0}}$ is pro-representable by the completion $\widehat{\mathcal{O}_{\mathcal{M}_{P}^{s}, [\mathcal{E}_{0}]}}$ of the noetherian local $k$-algebra $\mathcal{O}_{\mathcal{M}_{P}^{s}, [\mathcal{E}_{0}]}$ and using Schlessinger's Criterion \cite[Thm.~2.11]{Sch}. Entirely analogously, one can construct an obstruction map\footnote{By $\mathrm{Ext}^{\cdot}(-,-)$ we denote the Ext groups formed in the category $\mathrm{Qco}(X)$ of quasi-coherent sheaves on $X$. By $T\textrm{-}\mathrm{Ext}^{\cdot}(-,-)$ we denote the Ext groups formed in the category $\mathrm{Qco}^{T}(X)$ of $T$-equivariant quasi-coherent sheaves on $X$.}
\begin{equation}
\mathfrak{o}^{eq}(\sigma) : \mathcal{D}_{\mathcal{F}_{0}}(A) \longrightarrow T\textrm{-}\mathrm{Ext}^{2}(\mathcal{F}_{0}, \mathcal{F}_{0}) \otimes_{k} J, \nonumber
\end{equation}
also called the obstruction map, such that $\mathfrak{o}^{eq}(\sigma)^{-1}(0) = \mathrm{im}(\mathcal{D}_{\mathcal{F}_{0}}^{eq}(\sigma))$. Moreover, for any $[\mathcal{F}] \in \mathrm{im}(\mathcal{D}_{\mathcal{F}_{0}}^{eq}(\sigma))$, the fibre $\mathcal{D}_{\mathcal{F}_{0}}^{eq}(\sigma)^{-1}([\mathcal{F}])$ is naturally a $T\textrm{-}\mathrm{Ext}^{1}(\mathcal{F}_{0}, \mathcal{F}_{0}) \otimes_{k} J$-torsor. 

Rewriting the moduli functors in terms of deformation functors, we obtain
\begin{equation}
\underline{\mathcal{M}}(A') = \coprod_{[\mathcal{F}] \in \underline{\mathcal{M}}(A)} \mathcal{D}_{\mathcal{F} \otimes_{A} k}(\sigma)^{-1}([\mathcal{F}]), \ \underline{\mathcal{N}}(A') = \coprod_{[\mathcal{F}] \in \underline{\mathcal{N}}(A)} \mathcal{D}_{\mathcal{F} \otimes_{A} k}^{eq}(\sigma)^{-1}([\mathcal{F}]). \nonumber
\end{equation}
The remarks on obstructions and deformations together with the induction hypothesis, show that it is enough to relate the $T$-equivariant Ext groups to the invariant part of the ordinary Ext groups. It is enough to prove that for any equivariant coherent sheaves $\mathcal{A}, \mathcal{B}$ on $X$ and for any $i \in \mathbb{Z}$ there is a canonical bijection
\begin{equation}
T\textrm{-}\mathrm{Ext}^{i}(\mathcal{A},\mathcal{B}) \stackrel{\cong}{\longrightarrow} \mathrm{Ext}^{i}(\mathcal{A},\mathcal{B})^{T} \subset \mathrm{Ext}^{i}(\mathcal{A},\mathcal{B}). \nonumber
\end{equation} 
This can be seen by using the following spectral sequence \cite[Thm.~5.6.3]{Toh}
\begin{equation}
II_{2}^{p,q}(\mathcal{B}) = H^{p}(T, \mathrm{Ext}^{q}(\mathcal{A}, \mathcal{B})) \Longrightarrow T\textrm{-}\mathrm{Ext}^{n}(\mathcal{A},\mathcal{B}), \nonumber 
\end{equation}
together with $H^{p}(T, \mathrm{Ext}^{q}(\mathcal{A}, \mathcal{B})) = 0$ for any $p >0$, $q \in \mathbb{Z}$ \cite[Lem.~4.3]{Jan}. Note that $H^{p}(T,-)$ denotes rational cohomology.
\end{proof}
\begin{corollary}[Theorem \ref{ch. 1, sect. 1, thm. 1}] \label{ch. 1, sect. 4, cor. 1}
Let $X$ be a nonsingular projective toric variety. Let $\mathcal{O}_{X}(1)$ be an ample line bundle on $X$ and let $P$ be a choice of Hilbert polynomial of degree $\mathrm{dim}(X)$. Then there is a canonical isomorphism\footnote{Note that in the context of Theorem \ref{ch. 1, sect. 3, thm. 4}, ``Gieseker stable'' is equivalent to ``\emph{properly} GIT stable''. Therefore, one should take \emph{properly} GIT stable points on the right hand side.}
\begin{equation} \nonumber
\left( \mathcal{M}_{P}^{s} \right)^{T} \cong \coprod_{\vec{\chi} \in \left( \mathcal{X}_{P}^{0} \right)^{fr}} \mathcal{M}_{\vec{\chi}}^{0,s}. 
\end{equation}
\end{corollary}
\begin{proof}
Immediate from Theorems \ref{ch. 1, sect. 4, thm. 1} and \ref{ch. 1, sect. 3, thm. 4}.
\end{proof}

The advantage of this result is that for any nonsingular projective toric variety $X$ with ample line bundle $\mathcal{O}_{X}(1)$ and Hilbert polynomial $P$ of degree $\mathrm{dim}(X)$, we now have a combinatorial description of $\left( \mathcal{M}_{P}^{s} \right)^{T}$ in terms of the explicit moduli spaces of torsion free equivariant sheaves of section 2. Explicit knowledge of $\left( \mathcal{M}_{P}^{s} \right)^{T}$ is useful for computing invariants associated to $\mathcal{M}_{P}^{s}$, e.g.~the Euler characteristic of $\mathcal{M}_{P}^{s}$, using localisation techniques. This will be exploited in a sequel (\cite{Koo}), where we take $X$ to be a nonsingular complete toric surface over $\mathbb{C}$. We will derive expressions for generating functions of Euler characteristics of moduli spaces of $\mu$-stable torsion free sheaves on $X$, keeping track of the dependence on choice of ample line bundle $\mathcal{O}_{X}(1)$. This will give rise to wall-crossing formulae. We will give the easiest two examples occurring in \cite{Koo}, without further discussion. In these examples, wall-crossing phenomena are absent.  

\begin{example} Let $X$ be a nonsingular complete toric surface over $\mathbb{C}$ and let $H$ be an ample divisor on $X$. Let $e(X)$ be the Euler characteristic of $X$. Denote by $M_{X}^{H}(r,c_{1},c_{2})$ the moduli space of $\mu$-stable torsion free sheaves on $X$ of rank $r \in H^{0}(X,\mathbb{Z}) \cong \mathbb{Z}$, first Chern class $c_{1} \in H^{2}(X,\mathbb{Z})$ and second Chern class $c_{2} \in H^{4}(X,\mathbb{Z}) \cong \mathbb{Z}$. Then for rank $r=1$ 
\begin{equation} 
\sum_{c_{2} \in \mathbb{Z}} e(M_{X}^{H}(1,c_{1},c_{2})) q^{c_{2}} = \frac{1}{\prod_{k=1}^{\infty}(1-q^{k})^{e(X)}}. \nonumber
\end{equation}
This result is known for general nonsingular projective surfaces $X$ over $\mathbb{C}$ by work of G\"ottsche, using very different techniques, i.e.~using his expression for the Poincar\'e polynomial of Hilbert schemes of points computed using the Weil Conjectures. \hfill $\oslash$
\end{example}

\begin{example} Using the notation of the previous example, let $X = \mathbb{P}^{2}$ and rank $r=2$. Then
\begin{align}
&\sum_{c_{2} \in \mathbb{Z}} e(M_{X}^{H}(2,1,c_{2})) q^{c_{2}} = \frac{1}{\prod_{k=1}^{\infty}(1-q^{k})^{6}} \sum_{m = 1}^{\infty} \sum_{n = 1}^{\infty} \frac{q^{mn}}{1-q^{m+n-1}} \nonumber \\ 
&= q + 9 q^{2} + 48 q^{3} + 203 q^{4} + 729 q^{5} + 2346 q^{6} + 6918 q^{7} + 19062 q^{8} + 49620 q^{9} + O(q^{10}). \nonumber
\end{align}
Another expression for the same generating function has been obtained by Yoshioka, who obtains an expression for the Poincar\'e polynomial using the Weil Conjectures. In \cite{Kly4}, Klyachko also computes this generating function expressing it in terms of Hurwitz class numbers. In fact, the current paper lays the foundations for many ideas appearing in \cite{Kly4} and generalises them to pure equivariant sheaves of any dimension on any nonsingular toric variety. The sequel \cite{Koo} can be seen as a systematic application to torsion free sheaves on nonsingular complete toric surfaces. \hfill $\oslash$
\end{example}

\subsection{Fixed Point Loci of Moduli Spaces of Reflexive Sheaves on Toric Varieties}

We end this section by discussing how our theory so far can be adapted to give combinatorial descriptions of fixed point loci of moduli spaces of $\mu$-stable reflexive sheaves on a nonsingular projective toric variety $X$ with ample line bundle $\mathcal{O}_{X}(1)$. We will start by describing how sections 2 and 3 analogously hold in the setting of reflexive equivariant sheaves on nonsingular toric varieties. In fact, we will construct a particularly simple ample equivariant line bundle in the GIT problem, which precisely recovers $\mu$-stability. Subsequently, we will quickly construct a general theory of moduli spaces of $\mu$-stable reflexive sheaves on any normal projective variety $X$ with ample line bundle $\mathcal{O}_{X}(1)$ in a form useful for our purposes. Combining the results, gives the desired combinatorial description of the fixed point loci. 

Let $X$ be a nonsingular toric variety defined by a fan $\Delta$ in a lattice $N$ of rank $r$. In subsection 4.2, we mentioned the combinatorial description of reflexive equivariant sheaves on $X$ originally due to Klyachko (see for instance \cite{Kly4}). As we discussed, the category of reflexive equivariant sheaves on $X$ is equivalent to the category $\mathcal{R}$ of filtrations $\{E^{\rho}(\lambda)\}_{\rho \in \Delta(1)}$ of finite-dimensional nonzero $k$-vector spaces. Let $\mathcal{X}^{r}$ be the collection of characteristic functions of reflexive equivariant sheaves on $X$, which is a subset of the collection $\mathcal{X}^{0}$ of characteristic functions of torsion free equivariant sheaves on $X$. Note that the characteristic function of a reflexive equivariant sheaf can also occur as the characteristic function of a torsion free equivariant sheaf that is not reflexive. Now assume $X$ is projective and $\mathcal{O}_{X}(1)$ is an ample line bundle on $X$. Let $\vec{\chi} \in \mathcal{X}^{r}$, then we can introduce natural moduli functors
\begin{align*} 
\underline{\mathcal{N}}_{\vec{\chi}}^{\mu ss} &: (Sch/k)^{o} \longrightarrow Sets, \\
\underline{\mathcal{N}}_{\vec{\chi}}^{\mu s} &: (Sch/k)^{o} \longrightarrow Sets, 
\end{align*}  
where $\underline{\mathcal{N}}_{\vec{\chi}}^{\mu ss}(S)$ consists of equivariant $S$-flat families $\mathcal{F}$ on $X \times S$ such that the fibres $\mathcal{F}_{x}$ are $\mu$-semistable reflexive equivariant sheaves on $X \times k(x)$ with characteristic function $\vec{\chi}$, where we identify two such families $\mathcal{F}_{1}, \mathcal{F}_{2}$ if there is a line bundle $L$ on $S$ (with trivial equivariant structure) and an equivariant isomorphism $\mathcal{F}_{1} \cong \mathcal{F}_{2} \otimes p_{S}^{*}L$. The definition of $\underline{\mathcal{N}}_{\vec{\chi}}^{\mu s}$ is analogous using geometric $\mu$-stability\footnote{On a projective $k$-scheme $X$ (for $k$ \emph{any} field, not necessarily algebraically closed of characteristic zero) with ample line bundle $\mathcal{O}_{X}(1)$, a torsion free sheaf $\mathcal{E}$ is called geometrically $\mu$-stable if $\mathcal{E} \otimes_{k} K$ is torsion free and $\mu$-stable on $X \times_{k} K$ for any field extension $K/k$. If $k$ is algebraically closed, a torsion free sheaf $\mathcal{E}$ on $X$ is $\mu$-stable if and only if geometrically $\mu$-stable \cite[Exm.~1.6.5, Thm.~1.6.6, Cor.~1.5.11]{HL}.}. Taking $\tau = 0$, Theorem \ref{ch. 1, sect. 3, thm. 1} tells us how to describe equivariant $S$-flat families with fibres torsion free equivariant sheaves with fixed characteristic function $\vec{\chi}$. Let $\mathcal{F}$ be such a family with corresponding object $\hat{\mathcal{F}}^{\Delta} \in \mathcal{C}_{\vec{\chi}}^{0}(S)$. We see that $\mathcal{F}$ has reflexive fibres if and only if for all $\sigma \in \Delta$ a cone of maximal dimension and $x \in S$ we have
\begin{equation*}
\mathcal{F}^{\sigma}(\lambda_{1}, \ldots, \lambda_{r})_{x} = \mathcal{F}^{\sigma}(\lambda_{1}, \infty, \ldots, \infty)_{x} \cap \cdots \cap \mathcal{F}^{\sigma}(\infty, \ldots, \infty, \lambda_{r})_{x},
\end{equation*}
for all $\lambda_{1}, \ldots, \lambda_{r} \in \mathbb{Z}$, or equivalently for all $\sigma \in \Delta$ a cone of maximal dimension and $x \in S$ we have
\begin{equation} \label{ch. 1, eqnrefl}
\mathrm{dim}_{k(x)} \left( \mathcal{F}^{\sigma}(\lambda_{1}, \infty, \ldots, \infty)_{x} \cap \cdots \cap \mathcal{F}^{\sigma}(\infty, \ldots, \infty, \lambda_{r})_{x} \right) = \chi^{\sigma}(\lambda_{1}, \ldots, \lambda_{r}),
\end{equation}
for all $\lambda_{1}, \ldots, \lambda_{r} \in \mathbb{Z}$. This gives rise to a subcategory $\mathcal{C}_{\vec{\chi}}^{r}(S) \subset \mathcal{C}_{\vec{\chi}}^{0}(S)$ and the category of equivariant $S$-flat families with fibres reflexive equivariant sheaves with characteristic function $\vec{\chi}$ is equivalent to $\mathcal{C}_{\vec{\chi}}^{r}(S)$. Using a framing, we get a moduli functor $\mathfrak{C}_{\vec{\chi}}^{r,fr}$ which is a subfunctor of the functor $\mathfrak{C}_{\vec{\chi}}^{0,fr}$ introduced in subsection 3.2. Now let $N$ be the number of rays of $\Delta$ and $M = \chi^{\sigma}(\infty, \ldots, \infty)$ the rank, where $\sigma$ can be chosen to be any cone of maximal dimension. Referring to Proposition \ref{ch. 1, sect. 3, prop. 6} and using the notation occurring in the proof of Proposition \ref{ch. 1, sect. 3, prop. match}, we recall $\mathfrak{C}_{\vec{\chi}}^{0,fr}$ is represented by a closed subscheme
\begin{equation*}
\mathcal{N}_{\vec{\chi}}^{0} \subset \mathcal{A}' = \prod_{j=1}^{N} \prod_{k=1}^{M-1} \mathrm{Gr}(k,M) \times \prod_{\alpha=1}^{a} \mathrm{Gr}(n_{\alpha}, M). 
\end{equation*}
The new conditions on the fibres (\ref{ch. 1, eqnrefl}) determine an open subset $\mathcal{N}_{\vec{\chi}}^{r} \subset \mathcal{N}_{\vec{\chi}}^{0}$ which represents $\mathfrak{C}_{\vec{\chi}}^{r,fr}$. This can be proved by noting that for any finite product of Grassmannians $\prod_{i} \mathrm{Gr}(n_{i},N)$ the map $\{p_{i}\} \mapsto \mathrm{dim}_{k}\left( \bigcap_{i} p_{i}\right)$ is upper semicontinuous. In fact, $\mathcal{N}_{\vec{\chi}}^{r}$ is naturally a locally closed subscheme of just $\prod_{j=1}^{N} \prod_{k=1}^{M-1} \mathrm{Gr}(k,M)$. This subscheme is invariant under the natural regular action of $G = \mathrm{SL}(M,k)$ on $\prod_{j=1}^{N} \prod_{k=1}^{M-1} \mathrm{Gr}(k,M)$. We need the following variation on Proposition \ref{ch. 1, sect. 3, prop. 8}.
\begin{proposition} \label{ch. 1, sect. 4, prop. 8}
Let $X$ be a normal projective variety with ample line bundle $\mathcal{O}_{X}(1)$. Let $G$ be an affine algebraic group acting regularly on $X$. Let $\mathcal{E}$ be a torsion free $G$-equivariant sheaf on $X$. Then $\mathcal{E}$ is $\mu$-semistable if and only if $\mu_{\mathcal{F}} \leq \mu_{\mathcal{E}}$ for any $G$-equivariant coherent subsheaf $\mathcal{F}$ with $0 < \mathrm{rk}(\mathcal{F}) < \mathrm{rk}(\mathcal{E})$. Now assume in addition $\mathcal{E}$ is reflexive and $G = T$ is an algebraic torus. Then $\mathcal{E}$ is $\mu$-stable if and only if $\mu_{\mathcal{F}} < \mu_{\mathcal{E}}$ for any equivariant coherent subsheaf $\mathcal{F}$ with $0 < \mathrm{rk}(\mathcal{F}) < \mathrm{rk}(\mathcal{E})$. 
\end{proposition}
\begin{proof}
We can copy the proof of Proposition \ref{ch. 1, sect. 3, prop. 8}, but need $\mathcal{E}$ to be reflexive and $G = T$ an algebraic torus for the second part. The reason is that for reflexive sheaves we have the following three claims (see also the discussion at the start of the proof of Proposition \ref{ch. 1, sect. 3, prop. match}). Let $X$ be any projective normal variety with ample line bundle $\mathcal{O}_{X}(1)$. 

\noindent \emph{Claim 1.} Let $\mathcal{E}$ be a reflexive sheaf on $X$. Then $\mathcal{E}$ is $\mu$-semistable if and only if $\mu_{\mathcal{F}} \leq \mu_{\mathcal{E}}$ for any proper reflexive subsheaf $\mathcal{F} \subset \mathcal{E}$. Moreover, $\mathcal{E}$ is $\mu$-stable if and only if $\mu_{\mathcal{F}} < \mu_{\mathcal{E}}$ for any proper reflexive subsheaf $\mathcal{F} \subset \mathcal{E}$.  

\noindent \emph{Claim 2.} A reflexive $\mu$-polystable sheaf on $X$ is a $\mu$-semistable sheaf on $X$ isomorphic to a (finite, nontrivial) direct sum of reflexive $\mu$-stable sheaves. Let $\mathcal{E}$ be a $\mu$-semistable reflexive sheaf on $X$. Then $\mathcal{E}$ contains a unique maximal reflexive $\mu$-polystable subsheaf of the same slope as $\mathcal{E}$. This subsheaf we refer to as the reflexive $\mu$-socle of $\mathcal{E}$. 

\noindent \emph{Claim 3.} Let $\mathcal{E}$, $\mathcal{F}$ be reflexive $\mu$-stable sheaves on $X$ with the same slope. Then 
\begin{equation*}
\mathrm{Hom}(\mathcal{E},\mathcal{F}) = \left \{ \begin{array}{cc} k & \mathrm{if \ } \mathcal{E} \cong \mathcal{F} \\ 0 & \mathrm{otherwise}. \end{array} \right.
\end{equation*}

\noindent \emph{Proof of Claim 1.} We first observe that the saturation of a nonzero coherent subsheaf of a reflexive sheaf is reflexive \cite[Lem.~II.1.1.16]{OSS}. Hence, one only has to test $\mu$-semistability and $\mu$-stability of $\mathcal{E}$ for reflexive subsheaves $\mathcal{F} \subset \mathcal{E}$ with $0 < \mathrm{rk}(\mathcal{F}) < \mathrm{rk}(\mathcal{E})$ \cite[Prop.~1.2.6]{HL}. The claim follows from the statement that for any reflexive sheaf $\mathcal{E}$ on $X$ and reflexive subsheaf $\mathcal{F} \subset \mathcal{E}$ with $\mathrm{rk}(\mathcal{F}) = \mathrm{rk}(\mathcal{E})$ and $\mu_{\mathcal{F}} = \mu_{\mathcal{E}}$ one has $\mathcal{F} = \mathcal{E}$. This can be seen as follows. Suppose $\varnothing \neq Y = \mathrm{Supp}(\mathcal{E}/\mathcal{F})$, then $Y$ is a closed subset with $\mathrm{codim}(Y) \geq 2$. Since $\mathcal{E} |_{X \setminus Y} = \mathcal{F} |_{X \setminus Y}$ and for any open subset $U \subset X$, we have a commutative diagram \cite[Prop.~1.6]{Har2},
\begin{displaymath}
\xymatrix
{
\mathcal{F}(U) \ar[r] \ar[d]_{\cong} & \mathcal{E}(U) \ar[d]^{\cong} \\
\mathcal{F}(U \setminus Y) \ar[r]_{\cong} & \mathcal{E}(U \setminus Y),
}
\end{displaymath}     
we reach a contradiction. 

\noindent \emph{Proof of Claim 2.} Note that the collection $\{\mathcal{F}_{i} \ | \ i \in I\}$ of $\mu$-stable reflexive subsheaves of $\mathcal{E}$ with the same slope as $\mathcal{E}$ is nonempty (first remark in the proof of Claim 1). Consider the subsheaf $\mathcal{S} = \sum_{i \in I} \mathcal{F}_{i}$, which can be written as $\mathcal{S} = \sum_{i \in J} \mathcal{F}_{i}$ for some finite subset $\varnothing \neq J \subset I$. Assume $J = \{1, \ldots, m\}$ and $\mathcal{F}_{i+1} \nsubseteq \mathcal{F}_{1} + \cdots + \mathcal{F}_{i}$ for all $i = 1, \ldots, m-1$. It is enough to prove $(\mathcal{F}_{1} + \cdots + \mathcal{F}_{i}) \cap \mathcal{F}_{i+1} = 0$ for all $i = 1, \ldots, m-1$. Suppose we know the statement for $1, \ldots, k-1$, but $(\mathcal{F}_{1} + \cdots + \mathcal{F}_{k}) \cap \mathcal{F}_{k+1} \neq 0$. By the short exact sequence
\begin{equation*}
0 \longrightarrow (\mathcal{F}_{1} + \cdots + \mathcal{F}_{k}) \cap \mathcal{F}_{k+1} \longrightarrow (\mathcal{F}_{1} + \cdots + \mathcal{F}_{k}) \oplus \mathcal{F}_{k+1} \longrightarrow \mathcal{F}_{1} + \cdots + \mathcal{F}_{k+1} \longrightarrow 0,
\end{equation*} 
and $\mathcal{F}_{1} + \cdots + \mathcal{F}_{k} = \mathcal{F}_{1} \oplus \cdots \oplus \mathcal{F}_{k}$, we see that if $\mu_{(\mathcal{F}_{1} + \cdots + \mathcal{F}_{k}) \cap \mathcal{F}_{k+1}} < \mu_{\mathcal{F}_{1} \oplus \cdots \oplus \mathcal{F}_{k+1}}$, then 
\begin{equation*}
\mu_{\mathcal{F}_{1} + \cdots + \mathcal{F}_{k+1}} > \mu_{\mathcal{F}_{1} \oplus \cdots \oplus \mathcal{F}_{k+1}} = \mu_{\mathcal{F}_{1}} = \cdots = \mu_{\mathcal{F}_{k+1}} = \mu_{\mathcal{E}},
\end{equation*}
which contradicts $\mathcal{E}$ being $\mu$-semistable. Therefore
\begin{equation*}
\mu_{\mathcal{E}} = \mu_{\mathcal{F}_{1}} = \cdots = \mu_{\mathcal{F}_{k+1}} = \mu_{\mathcal{F}_{1} \oplus \cdots \oplus \mathcal{F}_{k+1}} \leq \mu_{(\mathcal{F}_{1} + \cdots + \mathcal{F}_{k}) \cap \mathcal{F}_{k+1}}.
\end{equation*}
On the other hand, $(\mathcal{F}_{1} + \cdots + \mathcal{F}_{k}) \cap \mathcal{F}_{k+1} = (\mathcal{F}_{1} \oplus \cdots \oplus \mathcal{F}_{k}) \cap \mathcal{F}_{k+1}$ is reflexive by \cite[Prop.~1.6]{Har2}, so Claim 1 implies
\begin{equation*}
\mu_{(\mathcal{F}_{1} + \cdots + \mathcal{F}_{k}) \cap \mathcal{F}_{k+1}} < \mu_{\mathcal{F}_{k+1}},
\end{equation*}
which yields a contradiction. 

\noindent \emph{Proof of Claim 3.} Let $\phi : \mathcal{E} \longrightarrow \mathcal{F}$ be a morphism. It suffices to prove $\phi$ is zero or an isomorphism, because $\mathcal{E}$, $\mathcal{F}$ are simple. Let $\mathcal{K}$ be the kernel and $\mathcal{I}$ the image of $\phi$. In the case $\mathcal{K} = \mathcal{E}$ we are done. In the case $\mathcal{K} = 0$, the possibility $0 \neq \mathcal{I} \subsetneq \mathcal{F}$ is ruled out by Claim 1 and we are done in the other cases. Suppose $0 \neq \mathcal{K} \subsetneq \mathcal{E}$, then $\mathcal{K}$ can easily seen to be reflexive by \cite[Prop.~1.6]{Har2}. Consequently, $\mu_{\mathcal{K}} < \mu_{\mathcal{E}}$ by Claim 1. In the case $\mathcal{I} = 0$, we are done. In the case $\mathcal{I} = \mathcal{F}$, we reach a contradiction since $\mu_{\mathcal{E}} = \mu_{\mathcal{K}}$. In the case $0 \neq \mathcal{I} \subsetneq \mathcal{F}$ we reach a contradiction since $\mu_{\mathcal{E}} \leq \mu_{\mathcal{K}}$. 
\end{proof}

\noindent Using Proposition \ref{ch. 1, sect. 4, prop. 8} and the proof of Proposition \ref{ch. 1, sect. 3, prop. match}, it easy to see we can choose an ample equivariant line bundle $\mathcal{L}_{\vec{\chi}}^{r}$ on $\mathcal{N}_{\vec{\chi}}^{r}$ such that the GIT semistable points of $\mathcal{N}_{\vec{\chi}}^{r}$ are precisely the $\mu$-semistable elements and the properly GIT stable points of $\mathcal{N}_{\vec{\chi}}^{r}$ are precisely the $\mu$-stable elements. This ample equivariant line bundle can be deduced from the $a=0$ case of the proof of Proposition \ref{ch. 1, sect. 3, prop. match} and is particularly simple. We choose such an ample equivariant line bundle and denote the GIT quotients by $\mathcal{N}^{\mu ss}_{\vec{\chi}}$, $\mathcal{N}^{\mu s}_{\vec{\chi}}$. We obtain the following theorem. 
\begin{theorem} \label{ch. 1, sect. 4, thm. 2}
Let $X$ be a nonsingular projective toric variety. Let $\mathcal{O}_{X}(1)$ be an ample line bundle on $X$ and $\vec{\chi} \in \mathcal{X}^{r}$ a characteristic function of a reflexive equivariant sheaf on $X$. Then $\underline{\mathcal{N}}_{\vec{\chi}}^{\mu ss}$ is corepresented by the quasi-projective $k$-scheme of finite type $\mathcal{N}^{\mu ss}_{\vec{\chi}}$. Moreover, there is an open subset $\mathcal{N}^{\mu s}_{\vec{\chi}} \subset \mathcal{N}^{\mu ss}_{\vec{\chi}}$ such that $\underline{\mathcal{N}}_{\vec{\chi}}^{\mu s}$ is corepresented by $\mathcal{N}_{\vec{\chi}}^{\mu s}$ and $\mathcal{N}^{\mu s}_{\vec{\chi}}$ is a coarse moduli space. 
\end{theorem}

We now discuss how to define moduli spaces of $\mu$-stable reflexive sheaves on normal projective varieties in general in a way useful for our purposes. Let $X$ be a normal projective variety with ample line bundle $\mathcal{O}_{X}(1)$ (not necessarily nonsingular or toric). Let $P$ be a choice of Hilbert polynomial of a reflexive sheaf on $X$. Then there are natural moduli functors $\underline{\mathcal{M}}_{P}^{ss}$, $\underline{\mathcal{M}}_{P}^{s}$ of flat families with fibres Gieseker semistable resp.~geometrically Gieseker stable sheaves with Hilbert polynomial $P$ as we discussed in subsection 4.1 referring to \cite{HL}. The moduli functor $\underline{\mathcal{M}}_{P}^{ss}$ is corepresented by a projective $k$-scheme $\mathcal{M}_{P}^{ss}$ of finite type and $\mathcal{M}_{P}^{ss}$ contains an open subset $\mathcal{M}_{P}^{s}$, which corepresents $\underline{\mathcal{M}}_{P}^{s}$ and is in fact a coarse moduli space. Let $\mathcal{P}$ be a property of coherent sheaves on $k$-schemes of finite type. We say $\mathcal{P}$ is an open property if for any projective morphism $f : Z \longrightarrow S$ of $k$-schemes of finite type and $\mathcal{F}$ an $S$-flat coherent sheaf on $Z$, the collection of points $x \in S$ such that the fibre $\mathcal{F}_{x}$ satisfies property $\mathcal{P}$ is open (see \cite[Def.~2.1.9]{HL}). We claim that if $\mathcal{P}$ is an open property, then the moduli functor 
\begin{equation*}
\underline{\mathcal{M}}_{P,\mathcal{P}}^{s} \subset \underline{\mathcal{M}}_{P}^{s}, 
\end{equation*}
defined as the subfunctor of all families with every fibre satisfying $\mathcal{P}$, is corepresented by an open subset $\mathcal{M}_{P, \mathcal{P}}^{s} \subset \mathcal{M}_{P}^{s}$ and $\mathcal{M}_{P, \mathcal{P}}^{s}$ is a coarse moduli space. This is immediate in the case we have a universal family for $\underline{\mathcal{M}}_{P}^{s}$ and on the level of Quot schemes we can always define obvious subfunctors represented by obvious open subsets. In the general case, one can prove the claim using arguments involving Luna's \'Etale Slice Theorem and local artinian $k$-algebras with residue field $k$ as in Propositions \ref{ch. 1, sect. 4, prop. 6} and \ref{ch. 1, sect. 4, prop. 7}. We now would like to take $\mathcal{P}$ to be the property ``geometrically $\mu$-stable and reflexive''. Using an argument analogous to the proof of \cite[Prop.~2.3.1]{HL} (which uses a boundedness result by Grothendieck), it is easy to see that geometrically $\mu$-stable is an open property. Using a result by Koll\'ar \cite[Prop.~28]{Kol} and a semicontinuity argument, we see that reflexive is also an open property. Therefore, it makes sense to define a moduli functor
\begin{equation*}
\underline{\mathcal{N}}_{P}^{\mu s} \subset \underline{\mathcal{M}}_{P}^{s}, 
\end{equation*}  
consisting of those families where all fibres are geometrically $\mu$-stable and reflexive. The moduli functor $\underline{\mathcal{N}}_{P}^{\mu s}$ is corepresented by an open subset $\mathcal{N}_{P}^{\mu s} \subset \mathcal{M}_{P}^{s}$ and $\mathcal{N}_{P}^{\mu s}$ is a coarse moduli space coming from a geometric quotient of an open subset of the Quot scheme. In the case $X$ is a nonsingular projective toric variety, we get a regular action $\sigma : T \times \mathcal{N}_{P}^{\mu s} \longrightarrow \mathcal{N}_{P}^{\mu s}$ and the fixed point locus is a closed subscheme $\left( \mathcal{N}_{P}^{\mu s} \right)^{T} \subset \mathcal{N}_{P}^{\mu s}$. We define $\left( \mathcal{X}_{P}^{r} \right)^{fr} = \left( \mathcal{X}_{P}^{0} \right)^{fr} \cap \mathcal{X}^{r}$ to be the collection of framed characteristic functions of reflexive equivariant sheaves on $X$ giving rise to Hilbert polynomial $P$. Completely analogous to subsections 4.1, 4.2, 4.3, we obtain the following theorem.
\begin{theorem} \label{ch. 1, sect. 4, thm. 3}
Let $X$ be a nonsingular projective toric variety. Let $\mathcal{O}_{X}(1)$ be an ample line bundle on $X$ and let $P$ be a choice of Hilbert polynomial of a reflexive sheaf on $X$. Then there is a canonical isomorphism
\begin{equation} \nonumber
\left( \mathcal{N}_{P}^{\mu s} \right)^{T} \cong \coprod_{\vec{\chi} \in \left( \mathcal{X}_{P}^{r} \right)^{fr}} \mathcal{N}_{\vec{\chi}}^{\mu s}. 
\end{equation}
\end{theorem}

\end{document}